\documentclass[12pt,letterpaper,leqno]{amsart}
\usepackage{microtype}
\usepackage{amssymb}
\usepackage{amsfonts}
\usepackage{geometry}
\usepackage{graphicx}
\usepackage{hyperref}

\renewcommand{\Re}{\operatorname{Re}}

\newtheorem{theorem}{Theorem}
\theoremstyle{plain}

\newtheorem{corollary}{Corollary}

\newtheorem{lemma}{Lemma}

\newtheorem{problem}{Problem}

\numberwithin{theorem}{section}  
\numberwithin{proposition}{section}  
\numberwithin{lemma}{section}  
\numberwithin{corollary}{section}  
\input{tcilatex}
\begin{document}
\title[1D Focusing NLS from 3D Focusing Quantum N-Body Dynamics]{Focusing
Quantum Many-body Dynamics II: \\
The Rigorous Derivation of the 1D Focusing Cubic Nonlinear Schr\"{o}dinger
Equation from 3D}
\author{Xuwen Chen}
\address{Department of Mathematics, Brown University, 151 Thayer Street,
Providence, RI 02912}
\email{chenxuwen@math.brown.edu}
\urladdr{http://www.math.brown.edu/\symbol{126}chenxuwen/}
\author{Justin Holmer}
\address{Department of Mathematics, Brown University, 151 Thayer Street,
Providence, RI 02912}
\email{holmer@math.brown.edu}
\urladdr{http://www.math.brown.edu/\symbol{126}holmer/}
\date{06/21/2014}

\begin{abstract}
We consider the focusing 3D quantum many-body dynamic which models a dilute
bose gas strongly confined in two spatial directions. We assume that the
microscopic pair interaction is attractive and given by $a^{3\beta
-1}V(a^{\beta }\cdot )$ where $\int V\leqslant 0$ and $a$ matches the
Gross-Pitaevskii scaling condition. We carefully examine the effects of the
fine interplay between the strength of the confining potential and the
number of particles on the 3D $N$-body dynamic. We overcome the difficulties
generated by the attractive interaction in 3D and establish new focusing
energy estimates. We study the corresponding BBGKY hierarchy which contains
a diverging coefficient as the strength of the confining potential tends to $%
\infty $. We prove that the limiting structure of the density matrices
counterbalances this diverging coefficient. We establish the convergence of
the BBGKY sequence and hence the propagation of chaos for the focusing
quantum many-body system. We derive rigorously the 1D focusing cubic NLS as
the mean-field limit of this 3D focusing quantum many-body dynamic and
obtain the exact 3D to 1D coupling constant.
\end{abstract}

\maketitle
\tableofcontents

\section{Introduction}

Since the Nobel prize winning first observation of Bose-Einstein condensate
(BEC) in 1995 \cite{Anderson, Davis}, the investigation of this new state of
matter has become one of the most active areas of contemporary research. A
BEC, first predicted theoretically by Einstein for non-interacting particles
in 1925, is a peculiar gaseous state that particles of integer spin (bosons)
occupy a macroscopic quantum state.

Let $t\in \mathbb{R}$ be the time variable and $\mathbf{r}_{N}=\left(
r_{1},r_{2},...,r_{N}\right) \in \mathbb{R}^{nN}$ be the position vector of $%
N$ particles in $\mathbb{R}^{n}$, then, naively, BEC means that, up to a
phase factor solely depending on $t,$ the $N$-body wave function $\psi
_{N}(t,\mathbf{r}_{N})$ satisfies 
\begin{equation*}
\psi _{N}(t,\mathbf{r}_{N})\sim \dprod\limits_{j=1}^{N}\varphi (t,r_{j})
\end{equation*}%
for some one particle state $\varphi .$ That is, every particle takes the
same quantum state. Equivalently, there is the Penrose-Onsager formulation
of BEC: if we let $\gamma _{N}^{(k)}$ be the $k$-particle marginal densities
associated with $\psi _{N}$ by 
\begin{equation}
\gamma _{N\,}^{(k)}(t,\mathbf{r}_{k};\mathbf{r}_{k}^{\prime })=\int \psi
_{N}(t,\mathbf{r}_{k},\mathbf{r}_{N-k})\overline{\psi _{N}}(t,\mathbf{r}%
_{k}^{\prime },\mathbf{r}_{N-k})d\mathbf{r}_{N-k},\quad \mathbf{r}_{k},%
\mathbf{r}_{k}^{\prime }\in \mathbb{R}^{nk},  \label{E:marginal}
\end{equation}%
then BEC equivalently means 
\begin{equation}
\gamma _{N}^{(k)}(t,\mathbf{r}_{k};\mathbf{r}_{k}^{\prime })\sim
\dprod\limits_{j=1}^{k}\varphi (t,r_{j})\bar{\varphi}(t,r_{j}^{\prime }).
\label{formula:BEC state}
\end{equation}%
It is widely believed that the cubic nonlinear Schr\"{o}dinger equation (NLS)%
\begin{equation*}
i\partial _{t}\phi =L\phi +\mu \left\vert \phi \right\vert ^{2}\phi ,
\end{equation*}%
where $L$ is the Laplacian $-\triangle $ or the Hermite operator $-\triangle
+\omega ^{2}\left\vert x\right\vert ^{2}$, fully describes the one particle
state $\varphi $ in (\ref{formula:BEC state}), also called the condensate
wave function since it characterizes the whole condensate. Such a belief is
one of the main motivations for studying the cubic NLS. Here, the nonlinear
term $\mu \left\vert \phi \right\vert ^{2}\phi $ represents a strong on-site
interaction taken as a mean-field approximation of the pair interactions
between the particles: a repelling interaction gives a positive $\mu $ while
an attractive interaction yields a $\mu <0$. Gross and Pitaevskii proposed
such a description of the many-body effect. Thus the cubic NLS is also
called the Gross-Pitaevskii equation. Because the cubic NLS is a
phenomenological mean-field type equation, naturally, its validity has to be
established rigorously from the many-body system which it is supposed to
characterize.

In a series of works \cite{Lieb2, AGT, E-E-S-Y1, E-S-Y1,E-S-Y2,E-S-Y5,
E-S-Y3,TChenAndNP,ChenAnisotropic,TChenAndNPSpace-Time, Chen3DDerivation,
SchleinNew, C-H3Dto2D, GM1,Sohinger3}, it has been proven rigorously that,
for a repelling interaction potential with suitable assumptions, relation %
\eqref{formula:BEC state} holds, moreover, the one-particle state $\varphi $
solves the defocusing cubic NLS ($\mu >0$).

It is then natural to ask if BEC happens (whether relation %
\eqref{formula:BEC state} holds) when we have attractive interparticle
interactions and if the condensate wave function $\varphi $ satisfies a
focusing cubic NLS ($\mu <0$) if relation \eqref{formula:BEC
state} does hold. In contemporary experiments, both positive \cite%
{Khaykovich,Streker} and negative \cite{Cornish,JILA2} results exist. To
present the mathematical interpretations of the experiments, we adopt the
notation 
\begin{equation*}
r_{i}=(x_{i},z_{i})\in \mathbb{R}^{2+1}
\end{equation*}%
and investigate the procedure of laboratory experiments of BEC subject to
attractive interactions according to \cite{Cornish,JILA2,Khaykovich,Streker}.

\begin{itemize}
\item[Step A.] Confine a large number of bosons, whose interactions are
originally \emph{repelling}, inside a trap. Reduce the temperature of the
system so that the many-body system reaches its ground state. It is expected
that this ground state is a BEC state / factorized state. This step then
corresponds to the following mathematical problem:

\begin{problem}
\label{Problem:Lieb}Show that if $\psi _{N,0}$ is the ground state of the
N-body Hamiltonian $H_{N,0}$ defined by%
\begin{equation}
H_{N,0}=\sum_{j=1}^{N}\left( -\triangle _{r_{j}}+\omega _{0,x}^{2}\left\vert
x_{j}\right\vert ^{2}+\omega _{0,z}^{2}z_{j}^{2}\right) +\sum_{1\leqslant
i<j\leqslant N}\frac{1}{a^{3\beta -1}}V_{0}\left( \frac{r_{i}-r_{j}}{%
a^{\beta }}\right)  \label{Hamiltonian:Step A}
\end{equation}%
where $V_{0}\geqslant 0$, then the marginal densities $\left\{ \gamma
_{N,0}^{(k)}\right\} $ associated with $\psi _{N,0}$, defined in (\ref%
{E:marginal}), satisfy relation \eqref{formula:BEC state}.
\end{problem}

Here, the quadratic potential $\omega ^{2}\left\vert \cdot \right\vert ^{2}$
stands for the trapping since \cite{Cornish,JILA2,Khaykovich,Streker} and
many other experiments of BEC use the harmonic trap and measure the strength
of the trap with $\omega $. We use $\omega _{0,x}$ to denote the trapping
strength in the $x$ direction and $\omega _{0,z}$ to denote the trapping
strength in the $z$ direction as we will explain later that, at the moment,
in order to have a BEC with attractive interaction, either experimentally or
mathematically, it is important to have $\omega _{0,x}\neq \omega _{0,z}$.
Moreover, we denote 
\begin{equation*}
\frac{1}{a}V_{0,a}\left( r\right) =\frac{1}{a^{3\beta -1}}V_{0}\left( \frac{r%
}{a^{\beta }}\right) \text{, }\beta >0
\end{equation*}%
the interaction potential.\footnote{%
From here on out, we consider the $\beta >0$ case solely. For $\beta =0$
(Hartree dynamic), see \cite%
{Frolich,E-Y1,KnowlesAndPickl,RodnianskiAndSchlein,MichelangeliSchlein,GMM1,GMM2,Chen2ndOrder,Ammari2,Ammari1,LChen}%
.} On the one hand, $V_{0,a}$ is an approximation of the identity as $%
a\rightarrow 0$ and hence matches the Gross-Pitaevskii description that the
many-body effect should be modeled by an on-site strong self interaction. On
the other hand, the extra $1/a$ is to make sure that the Gross-Pitaevskii
scaling condition is satisfied. This step is exactly the same as the
preparation of the experiments with repelling interactions and satisfactory
answers to Problem \ref{Problem:Lieb} have been given in \cite{Lieb3Dto1D}.

\item[Step B.] Use the property of Feshbach resonance, strengthen the trap
(increase $\omega _{0,x}\ $or $\omega _{0,z}$) to make the interaction
attractive and observe the evolution of the many-body system. This technique
continuously controls the sign and the size of the interaction in a certain
range.\footnote{%
See \cite[Fig.1]{Cornish}, \cite[Fig.2]{Khaykovich}, or \cite[Fig.1]{Streker}
for graphs of the relation between $\omega $ and $V$.} The system is then
time dependent. In order to observe BEC, the factorized structure obtained
in Step A must be preserved in time. Assuming this to be the case, we then
reset the time so that $t=0$ represents the point at which this Feshbach
resonance phase is complete. The subsequent evolution should then be
governed by a focusing time-dependent $N$-body Schr\"{o}dinger equation with
an attractive pair interaction $V$ subject to an asymptotically factorized
initial datum. The confining strengths are different from Step A as well and
we denote them by $\omega _{x}$ and $\omega _{z}$. A mathematically precise
statement is the following:

\begin{problem}
\label{Problem:Ours}Let $\psi _{N}\left( t,\mathbf{x}_{N}\right) $ be the
solution to the $N-body$ Schr\"{o}dinger equation%
\begin{equation}
i\partial _{t}\psi _{N}=\sum_{j=1}^{N}\left( -\triangle _{r_{j}}+\omega
_{x}^{2}\left\vert x_{j}\right\vert ^{2}+\omega _{z}^{2}z_{j}^{2}\right)
\psi _{N}+\sum_{1\leqslant i<j\leqslant N}\frac{1}{a^{3\beta -1}}V\left( 
\frac{r_{i}-r_{j}}{a^{\beta }}\right) \psi _{N}  \label{Hamiltonian:Step B}
\end{equation}%
where $V\leqslant 0,$ with $\psi _{N,0}$ from Step A as initial datum. Prove
that the marginal densities $\left\{ \gamma _{N}^{(k)}(t)\right\} $
associated with $\psi _{N}\left( t,\mathbf{x}_{N}\right) $ satisfies
relation \eqref{formula:BEC state}.\footnote{%
Since $\omega \neq \omega _{0}$, $V\neq V_{0}$, one could not expect that $%
\psi _{N,0},$ the ground state of (\ref{Hamiltonian:Step A}), is close to
the ground state of (\ref{Hamiltonian:Step B}).}
\end{problem}
\end{itemize}

In the experiment \cite{Cornish} by Cornell and Wieman's group (the JILA
group), once the interaction is tuned attractive, the condensate suddenly
shrinks to below the resolution limit, then after $\sim 5ms$, the many-body
system blows up. That is, there is no BEC once the interaction becomes
attractive. Moreover, there is no condensate wave function due to the
absence of the condensate. Whence, the current\ NLS theory, which is about
the condensate wave function when there is a condensate, cannot explain this 
$5ms$ of time or the blow up. This is currently an open problem in the study
of quantum many systems. The JILA group later conducted finer experiments 
\cite{JILA2} and remarked on \cite[p.299]{JILA2} that these are simple
systems with dramatic behavior and this behavior is providing puzzling
results when mean-field theory is tested against them.

In \cite{Khaykovich,Streker}, the particles are confined in a strongly
anisotropic cigar-shape trap to simulate a 1D system. That is, $\omega
_{x}\gg $ $\omega _{z}$. In this case, the experiment is a success in the
sense that one obtains a persistent BEC after the interaction is switched to
attractive. Moreover, a soliton is observed in \cite{Khaykovich} and a
soliton train is observed in \cite{Streker}. The solitons in \cite%
{Khaykovich,Streker} have different motion patterns.

In paper I \cite{C-HFocusing}, we have studied the simplified 1D version of (%
\ref{Hamiltonian:Step B}) as a model case and derived the 1D focusing cubic
NLS from it. In the present paper, we consider the full 3D problem of (\ref%
{Hamiltonian:Step B}) as in the experiments \cite{Khaykovich,Streker}: we
take $\omega _{z}=0$ and let $\omega _{x}\rightarrow \infty $ in (\ref%
{Hamiltonian:Step B}). We derive rigorously the 1D cubic focusing NLS
directly from a real 3D quantum many-body system. Here, "directly" means
that we are not passing through any 3D cubic NLS. On the one hand, one
infers from the experiment \cite{Cornish} that not only it is very difficult
to prove the 3D focusing NLS as the mean-field limit of a 3D focusing
quantum many-body dynamic, such a limit also may not be true. On the other
hand, the route which first derives 
\begin{equation}
i\partial _{t}\varphi =-\triangle _{x}+\omega ^{2}\left\vert x\right\vert
^{2}\varphi -\partial _{z}^{2}\varphi -\left\vert \varphi \right\vert
^{2}\varphi \text{,}  \label{eqn:3D Cubic NLS}
\end{equation}%
as a $N\rightarrow \infty $ limit, from the 3D $N$-body dynamic, and then
considers the $\omega \rightarrow \infty $ limit of \eqref{eqn:3D Cubic NLS}%
, corresponds to the iterated limit ($\lim_{\omega \rightarrow \infty
}\lim_{N\rightarrow \infty }$) of the $N$-body dynamic, i.e. the 1D focusing
cubic NLS coming from such a path approximates the 3D focusing $N$-body
dynamic when $\omega $ is large and $N$ is infinity (if not substantially
larger than $\omega $). In experiments, it is fully possible to have $N$ and 
$\omega $ comparable to each other. In fact, $N$ is about $10^{4}$ and $%
\omega $ is about $10^{3}$ in \cite{Kettle3Dto2DExperiment,
FrenchExperiment, NatureExperiment, Another2DExperiment}. Moreover, as seen
in the experiment \cite{JILA2}, even if $\omega _{x}$ is one digit larger
than $\omega _{z}$, negative result persists if $N$ is three digits larger
than $\omega _{x}$. Thus, in this paper, we derive rigorously the 1D
focusing cubic NLS as the double limit ($\lim_{N,\omega \rightarrow \infty }$%
) of a real focusing 3D quantum $N$-body dynamic directly, without passing
through any 3D cubic NLS. Furthermore, the interaction between the two
parameters $N$ and $\omega $ plays a central role. To be specific, we
establish the following theorem.

\begin{theorem}[main theorem]
\label{Theorem:3D->2D BEC (Nonsmooth)}Assume that the pair interaction $V\ $%
is an even Schwartz class function, which has a nonpositive integration,
i.e. $\int_{\mathbb{R}^{3}}V(r)dr\leqslant 0$, but may not be negative
everywhere. Let $\psi _{N,\omega }\left( t,\mathbf{r}_{N}\right) $ be the $%
N-body$ Hamiltonian evolution $e^{itH_{N,\omega }}\psi _{N,\omega }(0)$ with
the focusing $N-body$ Hamiltonian $H_{N,\omega }$ given by 
\begin{equation}
H_{N,\omega }=\sum_{j=1}^{N}\left( -\triangle _{r_{j}}+\omega ^{2}\left\vert
x_{j}\right\vert ^{2}\right) +\sum_{1\leqslant i<j\leqslant N}\left( N\omega
\right) ^{3\beta -1}V\left( \left( N\omega \right) ^{\beta }\left(
r_{i}-r_{j}\right) \right)  \label{Hamiltonian:3D to 1D N-body unscaled}
\end{equation}%
for some $\beta \in \left( 0,3/7\right) $. Let $\left\{ \gamma _{N,\omega
}^{(k)}\right\} $ be the family of marginal densities associated with $\psi
_{N,\omega }$. Suppose that the initial datum $\psi _{N,\omega }(0)$
verifies the following conditions:

\textnormal{(a)} $\psi _{N,\omega }(0)$ is normalized, that is, $\Vert \psi
_{N,\omega }(0)\Vert _{L^{2}}=1$,

\textnormal{(b)} $\psi _{N,\omega }(0)$ is asymptotically factorized in the
sense that 
\begin{equation}
\lim_{N,\omega \rightarrow \infty }\limfunc{Tr}\left\vert \frac{1}{\omega }%
\gamma _{N,\omega }^{(1)}(0,\frac{x_{1}}{\sqrt{\omega }},z_{1};\frac{%
x_{1}^{\prime }}{\sqrt{\omega }},z_{1}^{\prime })-h(x_{1})h(x_{1}^{\prime
})\phi _{0}(z_{1})\overline{\phi _{0}}(z_{1}^{\prime })\right\vert =0,
\label{convergence:initial}
\end{equation}%
for some one particle state $\phi _{0}\in H^{1}\left( \mathbb{R}\right) $
and $h$ is the normalized ground state for the 2D Hermite operator $%
-\triangle _{x}+\left\vert x\right\vert ^{2}$ i.e. $h(x)=\pi ^{-\frac{1}{2}%
}e^{-\left\vert x\right\vert ^{2}/2}$.

\textnormal{(c)} Away from the $x$-directional ground state energy, $\psi
_{N,\omega }(0)$ has finite energy per particle: 
\begin{equation*}
\sup_{\omega ,N}\frac{1}{N}\langle \psi _{N,\omega }(0),(H_{N,\omega
}-2N\omega )\psi _{N,\omega }(0)\rangle \leqslant C,
\end{equation*}%
Then there exist $C_{1}$ and $C_{2}$ which depend solely on $V$ such that $%
\forall k\geqslant 1,t\geqslant 0,$ and $\varepsilon >0$, we have the
convergence in trace norm (propagation of chaos) that 
\begin{equation}
\lim_{\substack{ N,\omega \rightarrow \infty  \\ C_{1}N^{v_{1}(\beta
)}\leqslant \omega \leqslant C_{2}N^{v_{2}(\beta )}}}\limfunc{Tr}\left\vert 
\frac{1}{\omega ^{k}}\gamma _{N,\omega }^{(k)}(t,\frac{\mathbf{x}_{k}}{\sqrt{%
\omega }},\mathbf{z}_{k};\frac{\mathbf{x}_{k}^{\prime }}{\sqrt{\omega }},%
\mathbf{z}_{k}^{\prime })-\dprod\limits_{j=1}^{k}h(x_{j})h(x_{j}^{\prime
})\phi (t,z_{j})\overline{\phi }(t,z_{j}^{\prime })\right\vert =0,
\label{convergence:conclusion of main theorem}
\end{equation}%
where $v_{1}(\beta )$ and $v_{2}(\beta )$ are defined by%
\begin{equation}
v_{1}(\beta )=\frac{\beta}{1-\beta }  \label{E:vofbeta1}
\end{equation}%
\begin{equation}
v_{2}(\beta )=\min \left( \frac{1-\beta }{\beta },\frac{\frac{3}{5}-\beta }{%
\beta -\frac{1}{5}}\mathbf{1}_{\beta \geqslant \frac{1}{5}}+\infty \cdot 
\mathbf{1}_{\beta <\frac{1}{5}},\frac{2\beta }{%
1-2\beta }-,\frac{\frac{7}{8}-\beta }{\beta }\right)   \label{E:vofbeta}
\end{equation}%
(see Fig. \ref{F:vofbeta}) and $\phi (t,z)$ solves the 1D focusing cubic NLS
with the "3D to 1D" coupling constant $b_{0}\left( \int \left\vert
h(x)\right\vert ^{4}dx\right) $ that is 
\begin{equation}
i\partial _{t}\phi =-\partial _{z}\phi -b_{0}\left( \int \left\vert
h(x)\right\vert ^{4}dx\right) \left\vert \phi \right\vert ^{2}\phi \quad 
\text{ in }\mathbb{R}  \label{equation:2D Cubic NLS}
\end{equation}%
with initial condition $\phi \left( 0,z\right) =\phi _{0}(z)$ and $%
b_{0}=\left\vert \int V\left( r\right) dr\right\vert $.
\end{theorem}

Theorem \ref{Theorem:3D->2D BEC (Nonsmooth)} is equivalent to the following
theorem.

\begin{theorem}[main theorem]
\label{Theorem:3D->2D BEC}Assume that the pair interaction $V\ $is an even
Schwartz class function, which has a nonpositive integration, i.e. $\int_{%
\mathbb{R}^{3}}V(r)dr\leqslant 0$, but may not be negative everywhere. Let $%
\psi _{N,\omega }\left( t,\mathbf{r}_{N}\right) $ be the $N-body$
Hamiltonian evolution $e^{itH_{N,\omega }}\psi _{N,\omega }(0)$, where the
focusing $N-body$ Hamiltonian $H_{N,\omega }$ is given by (\ref%
{Hamiltonian:3D to 1D N-body unscaled}) for some $\beta \in \left(
0,3/7\right) $. Let $\left\{ \gamma _{N,\omega }^{(k)}\right\} $ be the
family of marginal densities associated with $\psi _{N,\omega }$. Suppose
that the initial datum $\psi _{N,\omega }(0)$ is normalized, asymptotically
factorized in the sense of (a) and (b) of Theorem \ref{Theorem:3D->2D BEC
(Nonsmooth)} and satisfies the energy condition that

\textnormal{(c')} there is a $C>0$ such that%
\begin{equation}
\langle \psi _{N,\omega }(0),(H_{N,\omega }-2N\omega )^{k}\psi _{N,\omega
}(0)\rangle \leqslant C^{k}N^{k}\text{, }\forall k\geqslant 1,
\label{Condition:EnergyBoundOnInitialData}
\end{equation}%
Then there exists $C_{1}$,$C_{2}$ which depends solely on $V$ such that $%
\forall k\geqslant 1,\forall t\geqslant 0,$ we have the convergence in trace
norm (propagation of chaos) that 
\begin{equation*}
\lim_{\substack{ N,\omega \rightarrow \infty  \\ C_{1}N^{v_{1}(\beta
)}\leqslant \omega \leqslant C_{2}N^{v_{2}(\beta )}}}\limfunc{Tr}\left\vert 
\frac{1}{\omega ^{k}}\gamma _{N,\omega }^{(k)}(t,\frac{\mathbf{x}_{k}}{\sqrt{%
\omega }},\mathbf{z}_{k};\frac{\mathbf{x}_{k}^{\prime }}{\sqrt{\omega }},%
\mathbf{z}_{k}^{\prime })-\dprod\limits_{j=1}^{k}h(x_{j})h(x_{j}^{\prime
})\phi (t,z_{j})\overline{\phi }(t,z_{j}^{\prime })\right\vert =0,
\end{equation*}%
where $v_{1}(\beta )$ and $v_{2}(\beta )$ are given by (\ref{E:vofbeta1})
and (\ref{E:vofbeta}) and $\phi (t,z)$ solves the 1D focusing cubic NLS %
\eqref{equation:2D Cubic NLS}.
\end{theorem}

\begin{figure}[tbp]
\includegraphics[scale=0.75]{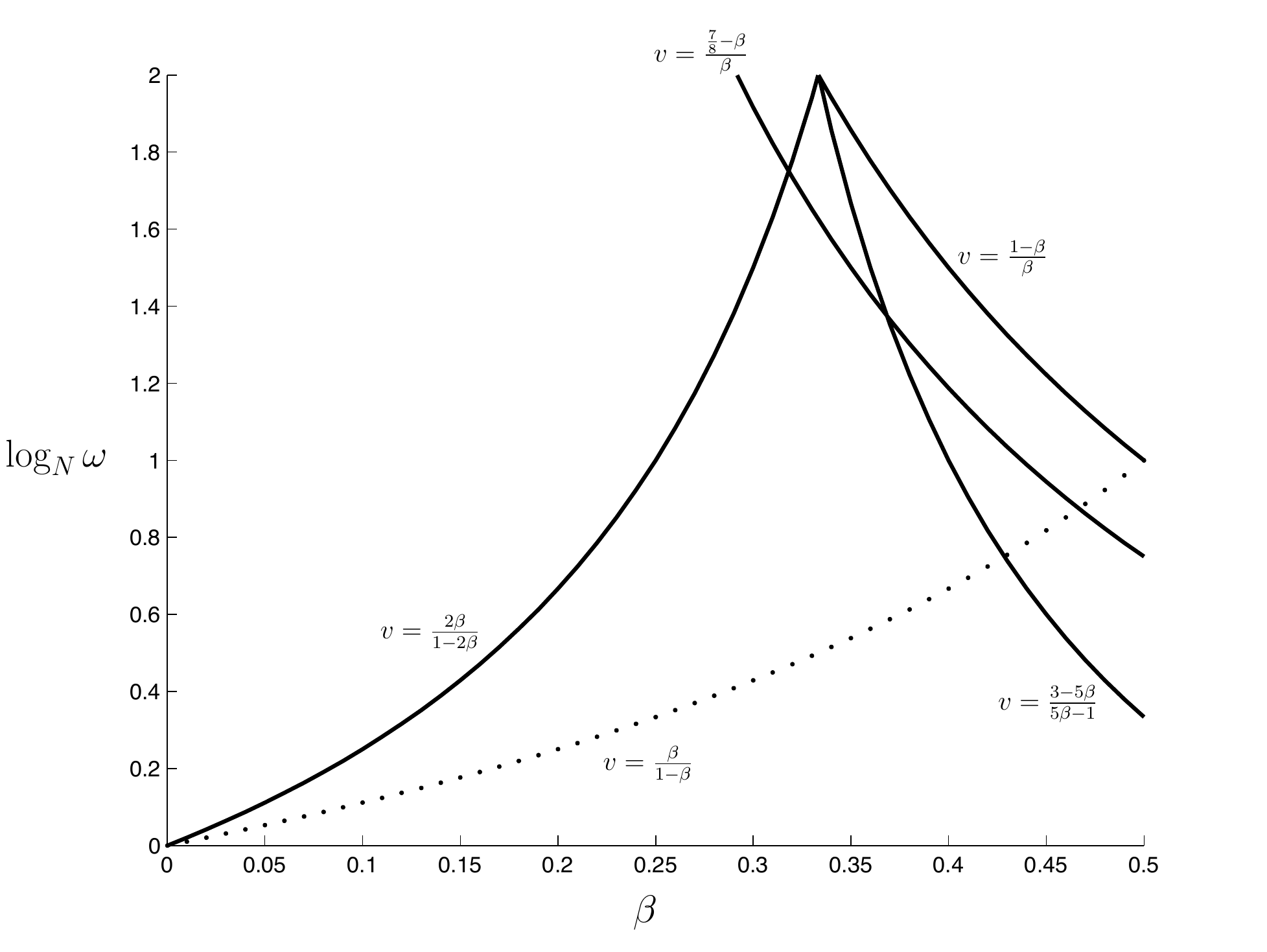}
\caption{ A graph of the various rational functions of $\protect\beta $
appearing in (\protect\ref{E:vofbeta1}) and (\protect\ref{E:vofbeta}). In
Theorems \protect\ref{Theorem:3D->2D BEC (Nonsmooth)}, \protect\ref%
{Theorem:3D->2D BEC}, the limit $(N,\protect\omega )\rightarrow \infty $ is
taken with $v_{1}(\protect\beta )\leqslant \log _{N}\protect\omega \leqslant
v_{2}(\protect\beta )$. The region of validity is above the dashed curve and
below the solid curves. It is a nonempty region for $0<\protect\beta %
\leqslant 3/7$. As shown here, there are values of $\protect\beta $ for
which $v_{1}(\protect\beta )\leqslant 1\leqslant v_{2}(\protect\beta )$,
which allows $N\sim \protect\omega $, as in the experimental paper 
\protect\cite{Cornish,JILA2,Khaykovich,Streker,Kettle3Dto2DExperiment,
FrenchExperiment, NatureExperiment,Another2DExperiment}. Moreover, our
result includes part of the $\protect\beta >1/3$ self-interaction region. We
will explain why we call the $\protect\beta >1/3$ case self-interaction
later in this introduction. At the moment, we remark that it is not a
coincidence that three restrictions intersect at $\protect\beta =1/3$}
\label{F:vofbeta}
\end{figure}

We remark that the assumptions in Theorem \ref{Theorem:3D->2D BEC
(Nonsmooth)} are reasonable assumptions on the initial datum coming from
Step A. In \cite[(1.10)]{Lieb3Dto1D}, a satisfying answer has been found by
Lieb, Seiringer, and Yngvason for Step A (Problem \ref{Problem:Lieb}) in the 
$\omega _{0,x}\gg \omega _{0,z}$ case. For convenience, set $\omega _{0,z}=1$
in the defocusing $N$-body Hamiltonian (\ref{Hamiltonian:Step A}) in Step A.
Let $\func{scat}(W)$ denote the 3D scattering length of the potential $W$.
By \cite[Lemma A.1]{E-S-Y2}, for $0<\beta \leq 1$ and $a\ll 1$, we have 
\begin{equation*}
\func{scat}\left( a\cdot \frac{1}{a^{3\beta }}V\left( \frac{r}{a^{\beta }}%
\right) \right) \sim 
\begin{cases}
\frac{a}{8\pi }\int_{\mathbb{R}^{3}}V & \text{if }0<\beta <1 \\ 
a\func{scat}\left( V\right) & \text{if }\beta =1%
\end{cases}%
\end{equation*}%
In \cite[(1.10)]{Lieb3Dto1D}, Lieb, Seiringer, and Yngvason define the
quantity $g=g(\omega _{0,x},N,a)$ by%
\begin{equation*}
g\overset{\mathrm{def}}{=}8\pi a\omega _{0,x}\left( \int \left\vert
h(x)\right\vert ^{4}dx\right) \text{.}
\end{equation*}%
Then if $Ng\sim 1$, they proved in \cite[Theorem 5.1]{Lieb3Dto1D} that BEC
happens in Step A and the Gross-Pitaevskii limit holds.\footnote{%
This corresponds to Region 2 of \cite{Lieb3Dto1D}. The other four regions
are, the ideal gas case, the 1D Thomas-Fermi case, the Lieb-Liniger case,
and the Girardeau-Tonks case. As mentioned in \cite[p.388]{Lieb3Dto1D}, BEC
is not expected in the Lieb-Liniger case and the Girardeau-Tonks case, and
is an open problem in the Thomas-Fermi case, we deal with Region 2 only in
this paper.} To be specific, they proved that 
\begin{equation*}
\lim_{N,\omega _{0,x}\rightarrow \infty }\limfunc{Tr}\left\vert \frac{1}{%
\omega _{0,x}}\gamma _{N,\omega _{0,x}}^{(1)}(0,\frac{x_{1}}{\sqrt{\omega
_{0,x}}},z_{1};\frac{x_{1}^{\prime }}{\sqrt{\omega _{0,x}}},z_{1}^{\prime
})-h(x_{1})h(x_{1}^{\prime })\phi _{0}(z_{1})\overline{\phi _{0}}%
(z_{1}^{\prime })\right\vert =0
\end{equation*}%
provided that $\phi _{0}$ is the minimizer to the 1D defocusing NLS energy
functional 
\begin{equation}
E_{\omega _{z},Ng}=\int_{\mathbb{R}}\left( |\partial _{z}\phi
(z)|^{2}+z^{2}|\phi (z)|^{2}+4\pi Ng|\phi (z)|^{4}\right) \,dz
\label{E:GP-Hamiltonian}
\end{equation}%
subject to the constraint $\Vert \phi \Vert _{L^{2}(\mathbb{R})}=1$. Hence,
the assumptions in Theorem \ref{Theorem:3D->2D BEC (Nonsmooth)} are
reasonable assumptions on the initial datum drawn from Step A. To be
specific, we have chosen $a=(N\omega )^{-1}$ in the interaction so that $%
Ng\sim 1$ and assumptions (a), (b) and (c) are the conclusions of \cite[%
Theorem 5.1]{Lieb3Dto1D}.

The equivalence of Theorems \ref{Theorem:3D->2D BEC (Nonsmooth)} and \ref%
{Theorem:3D->2D BEC} for asymptotically factorized initial data is
well-known. In the main part of this paper, we prove Theorem \ref%
{Theorem:3D->2D BEC} in full detail. For completeness, we discuss briefly
how to deduce Theorem \ref{Theorem:3D->2D BEC (Nonsmooth)} from Theorem \ref%
{Theorem:3D->2D BEC} in Appendix \ref{A:equivalence}.

To our knowledge, Theorems \ref{Theorem:3D->2D BEC (Nonsmooth)} and \ref%
{Theorem:3D->2D BEC} offer the first rigorous derivation of the 1D focusing
cubic NLS (\ref{equation:2D Cubic NLS}) from the 3D focusing quantum $N$%
-body dynamic (\ref{Hamiltonian:3D to 1D N-body unscaled}). Moreover, our
result covers part of the $\beta >1/3$ self-interaction region in 3D. As
pointed out in \cite{E-E-S-Y1}, the study of Step B is of particular
interest when $\beta \in \left( 1/3,1\right] $ in 3D. The reason is the
following. The initial datum coming from Step A is the ground state of (\ref%
{Hamiltonian:Step A}) with $\omega _{0,x},\omega _{0,z}$ $\neq 0$ and hence
is localized in space. We can assume all $N$ particles are in a box of
length $1$. Let the effective radius of the pair interaction $V$ be $R_{0},$
then the effective radius of $V\left( \left( N\omega \right) ^{\beta }\left(
r_{i}-r_{j}\right) \right) $ is about $R_{0}/\left( N\omega \right) ^{\beta
} $. Thus every particle in the box interacts with $\left( R_{0}/\left(
N\omega \right) ^{\beta }\right) ^{3}\times N$ other particles. Thus, for $%
\beta >1/3$ and large $N$, every particle interacts with only itself. This
exactly matches the Gross-Pitaevskii theory that the many-body effect should
be modeled by a strong on-site self-interaction. Therefore, for the
mathematical justification of the Gross-Pitaevskii theory, it is of
particular interest to prove Theorems \ref{Theorem:3D->2D BEC (Nonsmooth)}
and \ref{Theorem:3D->2D BEC} for self-interaction ($\beta >1/3)$.

A main tool used to prove Theorem \ref{Theorem:3D->2D BEC} is the analysis
of the BBGKY hierarchy of $\left\{ \tilde{\gamma}_{N,\omega }^{(k)}(t)=\frac{%
1}{\omega ^{k}}\gamma _{N,\omega }^{(k)}(t,\frac{\mathbf{x}_{k}}{\sqrt{%
\omega }},\mathbf{z}_{k};\frac{\mathbf{x}_{k}^{\prime }}{\sqrt{\omega }},%
\mathbf{z}_{k}^{\prime })\right\} _{k=1}^{N}$ as $N,\omega \rightarrow
\infty .$ In the classical setting, deriving mean-field type equations by
studying the limit of the BBGKY hierarchy was proposed by Kac and
demonstrated by Landford's work on the Boltzmann equation. In the quantum
setting, the usage of the BBGKY hierarchy was suggested by Spohn \cite{Spohn}
and has been proven to be successful by Elgart, Erd\"{o}s, Schlein, and Yau
in their fundamental papers \cite{E-E-S-Y1, E-S-Y1,E-S-Y2,E-S-Y5, E-S-Y3}%
\footnote{%
Around the same time, there was the 1D defocusing work \cite{AGT}.} which
rigorously derives the 3D cubic defocusing NLS from a 3D quantum many-body
dynamic with repulsive pair interactions and no trapping. The Elgart-Erd\"{o}%
s-Schlein-Yau program\footnote{%
See \cite{SchleinNew,GM1,Pickl} for different approaches.} consists of two
principal parts: in one part, they consider the sequence of the marginal
densities $\left\{ \gamma _{N}^{(k)}\right\} $ associated with the
Hamiltonian evolution $e^{itH_{N}}\psi _{N}(0)$ where 
\begin{equation*}
H_{N}=\sum_{j=1}^{N}-\triangle _{r_{j}}+\frac{1}{N}\sum_{1\leqslant
i<j\leqslant N}N^{3\beta }V(N^{\beta }\left( r_{i}-r_{j}\right) )
\end{equation*}%
and prove that an appropriate limit of as $N\rightarrow \infty $ solves the
3D Gross-Pitaevskii hierarchy 
\begin{equation}
i\partial _{t}\gamma ^{(k)}+\sum_{j=1}^{k}\left[ \triangle _{r_{k}},\gamma
^{(k)}\right] =b_{0}\sum_{j=1}^{k}\limfunc{Tr}\nolimits_{r_{k+1}}[\delta
(r_{j}-r_{k+1}),\gamma ^{(k+1)}],\text{ for all }k\geq 1\,.
\label{equation:Gross-Pitaevskii hiearchy without a trap}
\end{equation}%
In another part, they show that hierarchy 
\eqref{equation:Gross-Pitaevskii
hiearchy without a trap} has a unique solution which is therefore a
completely factorized state. However, the uniqueness theory for hierarchy 
\eqref{equation:Gross-Pitaevskii hiearchy
without a trap} is surprisingly delicate due to the fact that it is a system
of infinitely many coupled equations over an unbounded number of variables.
In \cite{KlainermanAndMachedon}, by assuming a space-time bound on the limit
of $\left\{ \gamma _{N}^{(k)}\right\} $, Klainerman and Machedon gave
another uniqueness theorem regarding (\ref{equation:Gross-Pitaevskii
hiearchy without a trap}) through a collapsing estimate originating from the
multilinear Strichartz estimates and a board game argument inspired by the
Feynman graph argument in \cite{E-S-Y2}.

The method by Klainerman and Machedon \cite{KlainermanAndMachedon} was taken
up by Kirkpatrick, Schlein, and Staffilani \cite{Kirpatrick}, who derived
the 2D cubic defocusing NLS from the 2D quantum many-body dynamic; by Chen
and Pavlovi\'{c} \cite{TChenAndNP}, who considered the 1D and 2D 3-body
repelling interaction problem; by X.C. \cite{ChenAnisotropic,
Chen3DDerivation}, who investigated the defocusing problem with trapping in
2D and 3D; and by X.C. and J.H. \cite{C-H3Dto2D}, who proved the
effectiveness of the defocusing 3D to 2D reduction problem. Such a method
has also inspired the study of the general existence theory of hierarchy $%
\eqref{equation:Gross-Pitaevskii hiearchy without a
trap}$, see \cite{TCNPNT, TCNPNT1, TChenAndNpGP1,Sohinger,SoSt13}.

One main open problem in Klainerman-Machedon theory is the verification of
the uniqueness condition in 3D though it is fully solved in 1D and 2D using
trace theorems by Kirkpatrick, Schlein, and Staffilani \cite{Kirpatrick}. In 
\cite{TChenAndNPSpace-Time}, for the 3D defocusing problem without traps,
Chen and Pavlovi\'{c} showed that, for $\beta \in (0,1/4)$, the limit of the
BBGKY sequence satisfies the uniqueness condition.\footnote{%
See also \cite{TCKT}.} In \cite{Chen3DDerivation}, X.C. extended and
simplified their method to study the 3D trapping problem for $\beta \in
(0,2/7].$ X.C. and J.H. \cite{C-H2/3}\ then extended the $\beta \in (0,2/7]$
result by X.C. to $\beta \in (0,2/3)$ using $X_{b}$ spaces and
Littlewood-Paley theory. The $\beta \in \left( 2/3,1\right] $ case is still
open.

Recently, using a version of the quantum de finite theorem from \cite{Lewin}%
, Chen, Hainzl, Pavlovi\'{c}, and Seiringer provided an alternative proof to
the uniqueness theorem in \cite{E-S-Y2} and showed that it is an
unconditional uniqueness result in the sense of NLS theory. With this
method, Sohinger derived the 3D defocusing cubic NLS in the periodic case 
\cite{Sohinger3}. See also \cite{C-PUniqueness,HoTaXi14}.

\subsection{Organization of the Paper}

We first outline the proof of our main theorem, Theorem \ref{Theorem:3D->2D
BEC}, in \S \ref{Sec:ProofOfMainTHM}. The components of the proof are in \S %
\ref{Section:EnergyEstimate}, \ref{Section:Compactness}, and \ref%
{Section:Convergence of The Infinite Hierarchy}.

The first main part is the proof of the needed focusing energy estimate,
stated and proved as Theorem \ref{Theorem:Energy Estimate} in \S \ref%
{Section:EnergyEstimate}. The main difficulty in establishing the energy
estimate is understanding the interplay between two parameters $N$ and $%
\omega $. On the one hand, as suggested by the experiments \cite%
{Cornish,JILA2,Khaykovich,Streker}, in order to have to a BEC in this
focusing setting, one has to explore "the 1D feature" of the 3D focusing $N$%
-body Hamiltonian (\ref{Hamiltonian:3D to 1D N-body unscaled}) which comes
from a large $\omega $. At the same time, an $N$ too large would allow the
3D effect to dominate, and one has to avoid this. This suggests that an
inequality of the form $N^{v_1(\beta)} \leq \omega$ is a natural
requirement. On the other hand, according to the uncertainty principle, in
3D, as the $x$-component of the particles' position becomes more and more
determined to be $0$, the $x$-component of the momentum and thus the energy
must blow up. Hence the energy of the system is dominated by its $x$%
-directional part which is in fact infinity as $\omega \rightarrow \infty $.
Since the particles are interacting via 3D potential, to avoid the excessive 
$x$-directional energy being transferred to the $z-$ direction, during the $%
N,\omega \rightarrow \infty $ process, $\omega $ can not be too large
either. Such a problem is totally new and does not exists in the 1D model 
\cite{C-HFocusing}. It suggests that an inequality of the form $\omega \leq
N^{\nu_2(\beta)}$ is a natural requirement.

The second main part of the proof is the analysis of the focusing "$\infty
-\infty $" BBGKY hierarchy of $\left\{ \tilde{\gamma}_{N,\omega }^{(k)}(t)=%
\frac{1}{\omega ^{k}}\gamma _{N,\omega }^{(k)}(t,\frac{\mathbf{x}_{k}}{\sqrt{%
\omega }},\mathbf{z}_{k};\frac{\mathbf{x}_{k}^{\prime }}{\sqrt{\omega }},%
\mathbf{z}_{k}^{\prime })\right\} _{k=1}^{N}$ as $N,\omega \rightarrow
\infty $. With our definition, the sequence of the marginal densities $%
\left\{ \tilde{\gamma}_{N,\omega }^{(k)}\right\} _{k=1}^{N}$ satisfies the
BBGKY hierarchy

\begin{eqnarray*}
i\partial _{t}\tilde{\gamma}_{N,\omega }^{(k)} &=&\omega
\sum_{j=1}^{k}[-\Delta _{x_{j}}+\left\vert x_{j}\right\vert ^{2},\tilde{%
\gamma}_{N,\omega }^{(k)}]+\sum_{j=1}^{k}[-\partial _{z_{j}}^{2},\tilde{%
\gamma}_{N,\omega }^{(k)}] \\
&&+\frac{1}{N}\sum_{1\leqslant i<j\leqslant k}[V_{N,\omega }(r_{i}-r_{j}),%
\tilde{\gamma}_{N,\omega }^{(k)}] \\
&&+\frac{N-k}{N}\sum_{j=1}^{k}\limfunc{Tr}\nolimits_{r_{k+1}}[V_{N,\omega
}(r_{j}-r_{k+1}),\tilde{\gamma}_{N,\omega }^{(k+1)}].
\end{eqnarray*}%
where $V_{N,\omega }$ is defined in (\ref{E:V}). We call it an "$\infty
-\infty $" BBGKY hierarchy because it is not clear whether the term 
\begin{equation*}
\omega \lbrack -\Delta _{x_{j}}+\left\vert x_{j}\right\vert ^{2},\tilde{%
\gamma}_{N,\omega }^{(k)}]
\end{equation*}%
tends to a limit as $N,\omega \rightarrow \infty $. Since $\tilde{\gamma}%
_{N,\omega }^{(k)}\ $is not a factorized state for $t>0$, one cannot expect
the commutator to be zero. This is in strong contrast with the "$n$D to $n$%
D" work \cite{AGT, E-E-S-Y1, E-S-Y1,E-S-Y2,E-S-Y5,
E-S-Y3,TChenAndNP,ChenAnisotropic,TChenAndNPSpace-Time, Chen3DDerivation,
Sohinger3} in which the formal limit of the corresponding BBGKY hierarchy is
fairly obvious. With the aforementioned focusing energy estimate, we find
that this diverging coefficient is counterbalanced by the limiting structure
of the density matrices and establish the weak* compactness and convergence
of this focusing BBGKY hierarchy in \S \ref{Section:Compactness} and \S \ref%
{Section:Convergence of The Infinite Hierarchy}.

\subsection{Acknowledgements}

J.H. was supported in part by NSF grant DMS-1200455.

\section{Proof of the Main Theorem\label{Sec:ProofOfMainTHM}}

We start by setting up some notation for the rest of the paper. Recall $%
h(x)=\pi ^{-\frac{1}{2}}e^{-\left\vert x\right\vert ^{2}/2}$, which is the
ground state for the 2D Hermite operator $-\triangle _{x}+\left\vert
x\right\vert ^{2}$ i.e. it solves $(-2-\Delta _{x}+\left\vert x\right\vert
^{2})h=0$. Then the normalized ground state eigenfunction $h_{\omega }(x)$
of $-\triangle _{x}+\omega ^{2}\left\vert x\right\vert ^{2}$ is given by $%
h_{\omega }(x)=\omega ^{1/2}h(\omega ^{1/2}x)$, i.e. it solves $(-2\omega
-\triangle _{x}+\omega ^{2}\left\vert x\right\vert ^{2})h_{\omega }=0$. In
particular, $h_{1}=h$. Noticing that both of the convergences (\ref%
{convergence:initial}) and (\ref{convergence:conclusion of main theorem})
involves scaling, we introduce the rescaled solution 
\begin{equation}
\tilde{\psi}_{N,\omega }(t,\mathbf{r}_{N})\overset{\mathrm{def}}{=}\frac{1}{%
\omega ^{N/2}}\psi _{N,\omega }(t,\frac{\mathbf{x}_{N}}{\sqrt{\omega }},%
\mathbf{z}_{N})  \label{E:rescaled}
\end{equation}%
and the rescaled Hamiltonian 
\begin{equation}
\tilde{H}_{N,\omega }=\left[ \sum_{j=1}^{N}-\partial _{z_{j}}^{2}+\omega
(-\triangle _{x}+\left\vert x\right\vert ^{2})\right] +\frac{1}{N}%
\sum_{1\leq i<j\leq N}V_{N,\omega }(r_{i}-r_{j}),
\label{E:rescaled-Hamiltonian}
\end{equation}%
where 
\begin{equation}
V_{N,\omega }(r)=N^{3\beta }\omega ^{3\beta -1}V\left( \frac{\left( N\omega
\right) ^{\beta }}{\sqrt{\omega }}x,\left( N\omega \right) ^{\beta }z\right)
.  \label{E:V}
\end{equation}%
Then 
\begin{equation*}
(\tilde{H}_{N,\omega }\tilde{\psi}_{N,\omega })(t,\mathbf{x}_{N},\mathbf{z}%
_{N})=\frac{1}{\omega ^{N/2}}(H_{N,\omega }\psi _{N,\omega })(t,\frac{%
\mathbf{x}_{N}}{\sqrt{\omega }},\mathbf{z}_{N}),
\end{equation*}%
and hence when $\psi _{N,\omega }(t)$ is the Hamiltonian evolution given by (%
\ref{Hamiltonian:3D to 1D N-body unscaled}) and $\tilde{\psi}_{N,\omega }$
is defined by \eqref{E:rescaled}, we have 
\begin{equation*}
\tilde{\psi}_{N,\omega }(t,\mathbf{r}_{N})=e^{it\tilde{H}_{N,\omega }}\tilde{%
\psi}(0,\mathbf{r}_{N})\text{.}
\end{equation*}%
If we let $\left\{ \tilde{\gamma}_{N,\omega }^{(k)}\right\} _{k=1}^{N}$ be
the marginal densities associated with $\tilde{\psi}_{N,\omega }$, then $%
\left\{ \tilde{\gamma}_{N,\omega }^{(k)}\right\} _{k=1}^{N}$ satisfies the "$%
\infty -\infty $" focusing BBGKY hierarchy%
\begin{eqnarray}
i\partial _{t}\tilde{\gamma}_{N,\omega }^{(k)} &=&\omega
\sum_{j=1}^{k}[-\Delta _{x_{j}}+\left\vert x_{j}\right\vert ^{2},\tilde{%
\gamma}_{N,\omega }^{(k)}]+\sum_{j=1}^{k}[-\partial _{z_{j}}^{2},\tilde{%
\gamma}_{N,\omega }^{(k)}]
\label{hierarchy:BBGKY hierarchy for scaled marginal densities} \\
&&+\frac{1}{N}\sum_{1\leqslant i<j\leqslant k}[V_{N,\omega }(r_{i}-r_{j}),%
\tilde{\gamma}_{N,\omega }^{(k)}]  \notag \\
&&+\frac{N-k}{N}\sum_{j=1}^{k}\limfunc{Tr}\nolimits_{r_{k+1}}[V_{N,\omega
}(r_{j}-r_{k+1}),\tilde{\gamma}_{N,\omega }^{(k+1)}].  \notag
\end{eqnarray}%
We will always take $\omega \geq 1$. For the rescaled marginals $\left\{ 
\tilde{\gamma}_{N,\omega }^{(k)}\right\} _{k=1}^{N}$, we define 
\begin{equation}
\tilde{S}_{j}\overset{\mathrm{def}}{=}\left[ 1-\partial _{z_{j}}^{2}+\omega
\left( -\Delta _{x_{j}}+\left\vert x_{j}\right\vert ^{2}-2\right) \right] ^{%
\frac{1}{2}}.  \label{E:tilde-S-def}
\end{equation}%
Two immediate properties of $\tilde{S}_{j}$ are the following. On the one
hand, $\tilde{S}_{j}^{2}(h_{1}(x_{j})\phi (z_{j}))=h_{1}(x_{j})(1-\partial
_{z_{j}}^{2})\phi (z_{j})\,$ and thus the diverging parameter $\omega $ has
no consequence when $\tilde{S}_{j}$ is applied to a tensor product function $%
h_{1}(x_{j})\phi (z_{j})$ for which the $x_{j}$-component rests in the
ground state. On the other hand, $\tilde{S}_{j}\geqslant 0$ as an operator
because $-\Delta _{x_{j}}+\left\vert x_{j}\right\vert ^{2}-2\geqslant 0$.

Now, noticing that the eigenvalues of $-\triangle _{x}+\omega ^{2}\left\vert
x\right\vert ^{2}$ in 2D are $\left\{ 2\left( l+1\right) \omega \right\}
_{l=0}^{\infty }$, let $P_{l\omega }$ the orthogonal projection onto the
eigenspace associated with eigenvalue $2\left( l+1\right) \omega $. That is, 
$I=\sum_{l=0}^{\infty }P_{l\omega }$ where $I:L^{2}(\mathbb{R}%
^{3})\rightarrow L^{2}(\mathbb{R}^{3})$. As a matter of notation for our
multi-coordinate problem, $P_{l\omega }^{j}$ will refer to the projection in 
$x_{j}$ coordinate at energy $2\left( l+1\right) \omega $, i.e.%
\begin{equation}
I=\prod_{j=1}^{k}\left( \sum_{l=0}^{\infty }P_{l\omega }^{j}\right) .
\label{def:I projection for omega}
\end{equation}%
In particular, when $\omega =1$, we use simply $P_{l}$. That is, $P_{0}$
denotes the orthogonal projection onto the ground state of $-\Delta
_{x}+\left\vert x\right\vert ^{2}$ and $P_{\geqslant 1}$ means the
orthogonal projection onto all higher energy modes of $-\Delta
_{x}+\left\vert x\right\vert ^{2}$ so that $I=P_{0}+P_{\geqslant 1}$, where $%
I:L^{2}(\mathbb{R}^{3})\rightarrow L^{2}(\mathbb{R}^{3})$. Since we will
only use $P_{0}$ and $P_{\geqslant 1}$ for the $\omega =1$ case, we define%
\begin{eqnarray*}
\mathcal{P}_{0} &=&P_{0} \\
\mathcal{P}_{1} &=&P_{\geqslant 1}
\end{eqnarray*}%
and 
\begin{equation}
\mathcal{P}_{\mathbf{\alpha }}=\mathcal{P}_{\alpha _{1}}^{1}\cdots \mathcal{P%
}_{\alpha _{k}}^{k}  \label{def:multiple projection for omega=1}
\end{equation}%
for a $k$-tuple $\mathbf{\alpha }=(\alpha _{1},\ldots ,\alpha _{k})$ with $%
\alpha _{j}\in \{0,1\}$ and adopt the notation $|\mathbf{\alpha }|=\alpha
_{1}+\cdots +\alpha _{k}$, then%
\begin{equation}
I=\sum_{\alpha }\mathcal{P}_{\mathbf{\alpha }}\text{.}  \label{E:cpct1}
\end{equation}

We next introduce an appropriate topology on the density matrices as was
previously done in \cite{E-E-S-Y1, E-Y1, E-S-Y1,E-S-Y2,E-S-Y5,
E-S-Y3,Kirpatrick,TChenAndNP,ChenAnisotropic,Chen3DDerivation,C-H3Dto2D,C-H2/3,C-HFocusing,Sohinger3}%
. Denote the spaces of compact operators and trace class operators on $%
L^{2}\left( \mathbb{R}^{3k}\right) $ as $\mathcal{K}_{k}$ and $\mathcal{L}%
_{k}^{1}$, respectively. Then $\left( \mathcal{K}_{k}\right) ^{\prime }=%
\mathcal{L}_{k}^{1}$. By the fact that $\mathcal{K}_{k}$ is separable, we
pick a dense countable subset $\{J_{i}^{(k)}\}_{i\geqslant 1}\subset 
\mathcal{K}_{k}$ in the unit ball of $\mathcal{K}_{k}$ (so $\Vert
J_{i}^{(k)}\Vert _{\func{op}}\leqslant 1$ where $\left\Vert \cdot
\right\Vert _{\func{op}}$ is the operator norm). For $\gamma
_{1}^{(k)},\gamma _{2}^{(k)}\in \mathcal{L}_{k}^{1}$, we then define a
metric $d_{k}$ on $\mathcal{L}_{k}^{1}$ by 
\begin{equation*}
d_{k}(\gamma _{1}^{(k)},\gamma _{2}^{(k)})=\sum_{i=1}^{\infty
}2^{-i}\left\vert \limfunc{Tr}J_{i}^{(k)}\left( \gamma _{1}^{(k)}-\gamma
_{2}^{(k)}\right) \right\vert .
\end{equation*}%
A uniformly bounded sequence $\tilde{\gamma}_{N,\omega }^{(k)}\in \mathcal{L}%
_{k}^{1}$ converges to $\tilde{\gamma}^{(k)}\in \mathcal{L}_{k}^{1}$ with
respect to the weak* topology if and only if 
\begin{equation*}
\lim_{N,\omega \rightarrow \infty }d_{k}(\tilde{\gamma}_{N,\omega }^{(k)},%
\tilde{\gamma}^{(k)})=0.
\end{equation*}%
For fixed $T>0$, let $C\left( \left[ 0,T\right] ,\mathcal{L}_{k}^{1}\right) $
be the space of functions of $t\in \left[ 0,T\right] $ with values in $%
\mathcal{L}_{k}^{1}$ which are continuous with respect to the metric $d_{k}.$
On $C\left( \left[ 0,T\right] ,\mathcal{L}_{k}^{1}\right) ,$ we define the
metric 
\begin{equation*}
\hat{d}_{k}(\gamma ^{(k)}\left( \cdot \right) ,\tilde{\gamma}^{(k)}\left(
\cdot \right) )=\sup_{t\in \left[ 0,T\right] }d_{k}(\gamma ^{(k)}\left(
t\right) ,\tilde{\gamma}^{(k)}\left( t\right) ),
\end{equation*}%
and denote by $\tau _{prod}$ the topology on the space $\oplus _{k\geqslant
1}C\left( \left[ 0,T\right] ,\mathcal{L}_{k}^{1}\right) $ given by the
product of topologies generated by the metrics $\hat{d}_{k}$ on $C\left( %
\left[ 0,T\right] ,\mathcal{L}_{k}^{1}\right) .$

With the above topology on the space of marginal densities, we prove Theorem %
\ref{Theorem:3D->2D BEC}. The proof is divided into five steps.

\begin{itemize}
\item[Step I] (Focusing Energy Estimate) We first establish, via an
elaborate calculation in Theorem \ref{Theorem:Energy Estimate}, that one can
compensate the negativity of the interaction in the focusing many-body
Hamiltonian (\ref{Hamiltonian:3D to 1D N-body unscaled}) by adding a product
of $N$ and some constant $\alpha $ depending on $V$, provided that $%
C_{1}N^{v_{1}(\beta )}\leqslant \omega \leqslant C_{2}N^{v_{2}(\beta )}$
where $C_{1}$ and $C_{2}$ depend solely on $V$. Henceforth, though $%
H_{N,\omega }$ is not positive-definite, we derive, from the energy
condition (\ref{Condition:EnergyBoundOnInitialData}), a\ $H^{1}$ type energy
bound:%
\begin{equation*}
\left\langle \psi _{N,\omega },\left( \alpha +N^{-1}H_{N,\omega }-2\omega
\right) ^{k}\psi _{N,\omega }\right\rangle \geqslant C^{k}\left\Vert
\dprod\limits_{j=1}^{k}S_{j}\psi _{N,\omega }\right\Vert _{L^{2}(\mathbb{R}%
^{3N})}^{2}
\end{equation*}%
where%
\begin{equation*}
S_{j}\overset{\mathrm{def}}{=}(1-\Delta _{x_{j}}+\omega ^{2}\left\vert
x_{j}\right\vert ^{2}-2\omega -\partial _{z_{j}}^{2})^{1/2}.
\end{equation*}%
Since the quantity $\left\langle \psi _{N,\omega },\left( H_{N,\omega
}-2N\omega \right) ^{k}\psi _{N,\omega }\right\rangle $ is conserved by the
evolution, via Corollary \ref{Corollary:Energy Bound for Marginal Densities}%
, we deduce the \emph{a priori} bounds, crucial to the analysis of the "$%
\infty -\infty $" BBGKY hierarchy (\ref{hierarchy:BBGKY hierarchy for scaled
marginal densities}), on the scaled marginal densities: 
\begin{equation*}
\sup_{t}\limfunc{Tr}\left( \dprod\limits_{j=1}^{k}\tilde{S}_{j}\right) 
\tilde{\gamma}_{N,\omega }^{(k)}\left( \dprod\limits_{j=1}^{k}\tilde{S}%
_{j}\right) \leqslant C^{k},
\end{equation*}%
\begin{equation*}
\sup_{t}\limfunc{Tr}\dprod\limits_{j=1}^{k}\left( 1-\triangle
_{r_{j}}\right) \tilde{\gamma}_{N,\omega }^{(k)}\leqslant C^{k},
\end{equation*}%
\begin{equation*}
\sup_{t}\func{Tr}\mathcal{P}_{\mathbf{\alpha }}\tilde{\gamma}_{N,\omega
}^{(k)}\mathcal{P}_{\mathbf{\beta }}\leq C^{k}\omega ^{-\frac{1}{2}|\mathbf{%
\alpha }|-\frac{1}{2}\mathbf{|\beta }|},
\end{equation*}%
where $\mathcal{P}_{\mathbf{\alpha }}$ and $\mathcal{P}_{\mathbf{\beta }}$
are defined as in (\ref{def:multiple projection for omega=1}). We remark
that the quantity 
\begin{equation*}
\limfunc{Tr}\left( 1-\triangle _{r_{1}}\right) \tilde{\gamma}_{N,\omega
}^{(1)}
\end{equation*}%
is not the one particle kinetic energy of the system; the one particle
kinetic energy of the system is $\limfunc{Tr}\left( 1-\omega \triangle
_{x_{1}}-\partial _{z_{1}}^{2}\right) \tilde{\gamma}_{N,\omega }^{(1)}$ and
grows like $\omega $. This is also in contrast to the $n$D to $n$D work,

\item[Step II] (Compactness of BBGKY). We fix $T>0$ and work in the
time-interval $t\in \lbrack 0,T].$ In Theorem \ref{Theorem:Compactness of
the scaled marginal density}, we establish the compactness of the BBGKY
sequence $\left\{ \Gamma _{N,\omega }(t)=\left\{ \tilde{\gamma}_{N,\omega
}^{(k)}\right\} _{k=1}^{N}\right\} \subset \oplus _{k\geqslant 1}C\left( %
\left[ 0,T\right] ,\mathcal{L}_{k}^{1}\right) $ with respect to the product
topology $\tau _{prod}$ even though hierarchy 
\eqref{hierarchy:BBGKY hierarchy for scaled
marginal densities} contains attractive interactions and an indefinite $%
\infty -\infty $. Moreover, in Corollary \ref{Corollary:LimitMustBeAProduct}%
, we prove that, to be compatible with the energy bound obtained in Step I,
every limit point $\Gamma (t)=\left\{ \tilde{\gamma}^{(k)}\right\}
_{k=1}^{\infty }$ must take the form 
\begin{equation*}
\tilde{\gamma}^{(k)}\left( t,\left( \mathbf{x}_{k},\mathbf{z}_{k}\right)
;\left( \mathbf{x}_{k}^{\prime },\mathbf{z}_{k}^{\prime }\right) \right)
=\left( \dprod\limits_{j=1}^{k}h_{1}\left( x_{j}\right) h_{1}\left(
x_{j}^{\prime }\right) \right) \tilde{\gamma}_{z}^{(k)}(t,\mathbf{z}_{k};%
\mathbf{z}_{k}^{\prime }),
\end{equation*}%
where $\tilde{\gamma}_{z}^{(k)}=\limfunc{Tr}_{x}\tilde{\gamma}^{(k)}$ is the 
$z$-component of $\tilde{\gamma}^{(k)}.$

\item[Step III] (Limit points of BBGKY satisfy GP). In Theorem \ref%
{Theorem:Convergence to the Coupled Gross-Pitaevskii}, we prove that if $%
\Gamma (t)=\left\{ \tilde{\gamma}^{(k)}\right\} _{k=1}^{\infty }$ is a $%
C_{1}N^{v_{1}(\beta )}\leqslant \omega \leqslant C_{2}N^{v_{2}(\beta )}$
limit point of $\left\{ \Gamma _{N,\omega }(t)=\left\{ \tilde{\gamma}%
_{N,\omega }^{(k)}\right\} _{k=1}^{N}\right\} $ with respect to the product
topology $\tau _{prod}$, then $\left\{ \tilde{\gamma}_{z}^{(k)}=\limfunc{Tr}%
_{x}\tilde{\gamma}^{(k)}\right\} _{k=1}^{\infty }$ is a solution to the
focusing coupled Gross-Pitaevskii (GP) hierarchy subject to initial data $%
\tilde{\gamma}_{z}^{(k)}\left( 0\right) =\left\vert \phi _{0}\right\rangle
\left\langle \phi _{0}\right\vert ^{\otimes k}$ with coupling constant $%
b_{0}=$ $\left\vert \int V\left( r\right) dr\right\vert $, which written in
differential form, is 
\begin{equation}
i\partial _{t}\tilde{\gamma}_{z}^{(k)}=\sum_{j=1}^{k}\left[ -\partial
_{z_{j}}^{2},\tilde{\gamma}_{z}^{(k)}\right] -b_{0}\sum_{j=1}^{k}\limfunc{Tr}%
\nolimits_{z_{k+1}}\limfunc{Tr}\nolimits_{x}\left[ \delta \left(
r_{j}-r_{k+1}\right) ,\tilde{\gamma}^{(k+1)}\right] .
\label{hierarchy:Coupled GP in differential form}
\end{equation}%
Together with the limiting structure concluded in Corollary \ref%
{Corollary:LimitMustBeAProduct}, we can further deduce that $\left\{ \tilde{%
\gamma}_{z}^{(k)}=\limfunc{Tr}_{x}\tilde{\gamma}^{(k)}\right\}
_{k=1}^{\infty }$ is a solution to the 1D focusing GP hierarchy subject to
initial data $\tilde{\gamma}_{z}^{(k)}\left( 0\right) =\left\vert \phi
_{0}\right\rangle \left\langle \phi _{0}\right\vert ^{\otimes k}$ with
coupling constant $b_{0}\left( \int \left\vert h_{1}\left( x\right)
\right\vert ^{4}dx\right) $, which, written in differential form, is 
\begin{equation}
i\partial _{t}\tilde{\gamma}_{z}^{(k)}=\sum_{j=1}^{k}\left[ -\partial
_{z_{j}}^{2},\tilde{\gamma}_{z}^{(k)}\right] -b_{0}\left( \int \left\vert
h_{1}\left( x\right) \right\vert ^{4}dx\right) \sum_{j=1}^{k}\limfunc{Tr}%
\nolimits_{z_{k+1}}\left[ \delta \left( z_{j}-z_{k+1}\right) ,\tilde{\gamma}%
_{z}^{(k+1)}\right] .  \label{hierarchy:1D GP in differential form}
\end{equation}

\item[Step IV] (GP has a unique solution). When $\tilde{\gamma}%
_{z}^{(k)}\left( 0\right) =\left\vert \phi _{0}\right\rangle \left\langle
\phi _{0}\right\vert ^{\otimes k},$ we know one solution to the 1D focusing
GP hierarchy (\ref{hierarchy:1D GP in differential form}), namely $%
\left\vert \phi \right\rangle \left\langle \phi \right\vert ^{\otimes k}$ if 
$\phi $ solves the 1D focusing NLS \eqref{equation:2D Cubic NLS}. Since we
have proven the \emph{a priori} bound 
\begin{equation*}
\sup_{t}\limfunc{Tr}\left( \dprod\limits_{j=1}^{k}\left\langle \partial
_{z_{j}}\right\rangle \right) \tilde{\gamma}_{z}^{(k)}\left(
\dprod\limits_{j=1}^{k}\left\langle \partial _{z_{j}}\right\rangle \right)
\leqslant C^{k},
\end{equation*}%
A trace theorem then shows that $\left\{ \tilde{\gamma}_{z}^{(k)}\right\} $
verifies the requirement of the following uniqueness theorem and hence we
conclude that $\tilde{\gamma}_{z}^{(k)}=\left\vert \phi \right\rangle
\left\langle \phi \right\vert ^{\otimes k}$.
\end{itemize}

\begin{theorem}[{\protect\cite[Theorem 1.3]{C-HFocusing}}]
\footnote{%
For other uniqueness theorems or related estimates regarding the GP
hierarchies, see \cite%
{E-S-Y2,KlainermanAndMachedon,Kirpatrick,GM,ChenDie,ChenAnisotropic,Beckner,Sohinger,TCNPdeFinitte,HoTaXi14,Sohinger3}%
}\label{THM:OptimalUniqueness of 1D GP}Let%
\begin{equation*}
B_{j,k+1}\gamma _{z}^{(k+1)}=\limfunc{Tr}\nolimits_{z_{k+1}}\left[ \delta
\left( z_{j}-z_{k+1}\right) ,\gamma _{z}^{(k+1)}\right] .
\end{equation*}%
If $\left\{ \gamma _{z}^{(k)}\right\} _{k=1}^{\infty }$ solves the 1D
focusing GP hierarchy (\ref{hierarchy:1D GP in differential form}) subject
to zero initial data and the space-time bound\footnote{%
Though the space-time bound (\ref{Condition:Space-Time Bound}) follows from
a simple trace theorem here, verifying such a condition in 3D is highly
nontrivial and is merely partially solved so far. See \cite%
{TChenAndNPSpace-Time,Chen3DDerivation,C-H2/3}} 
\begin{equation}
\int_{0}^{T}\left\Vert \left( \dprod\limits_{j=1}^{k}\left\langle \partial
_{z_{j}}\right\rangle ^{\varepsilon }\left\langle \partial _{z_{j}^{\prime
}}\right\rangle ^{\varepsilon }\right) B_{j,k+1}\gamma _{z}^{(k+1)}(t,%
\mathbf{\cdot };\mathbf{\cdot })\right\Vert _{L_{\mathbf{z,z}^{\prime
}}^{2}}dt\leqslant C^{k}  \label{Condition:Space-Time Bound}
\end{equation}%
for some $\varepsilon ,C>0$ and all $1\leqslant j\leqslant k.$ Then $\forall
k,t\in \lbrack 0,T]$, $\gamma _{z}^{(k+1)}=0.$
\end{theorem}

Thus the compact sequence $\left\{ \Gamma _{N,\omega }(t)=\left\{ \tilde{%
\gamma}_{N,\omega }^{(k)}\right\} _{k=1}^{N}\right\} $ has only one $%
C_{1}N^{v_{1}(\beta )}\leqslant \omega \leqslant C_{2}N^{v_{2}(\beta )}$
limit point, namely 
\begin{equation*}
\tilde{\gamma}^{(k)}=\dprod\limits_{j=1}^{k}h_{1}\left( x_{j}\right)
h_{1}(x_{j}^{\prime })\phi (t,z_{j})\overline{\phi }(t,z_{j}^{\prime })\,.
\end{equation*}%
We then infer from the definition of the topology that as trace class
operators 
\begin{equation*}
\tilde{\gamma}_{N,\omega }^{(k)}\rightarrow
\dprod\limits_{j=1}^{k}h_{1}\left( x_{j}\right) h_{1}(x_{j}^{\prime })\phi
(t,z_{j})\overline{\phi }(t,z_{j}^{\prime })\text{ weak*.}
\end{equation*}

\begin{itemize}
\item[Step V] (Weak* convergence upgraded to strong). Since the limit
concluded in Step IV is an orthogonal projection, the well-known argument in 
\cite{E-S-Y3} upgrades the weak* convergence to strong. In fact, testing the
sequence against the compact observable 
\begin{equation*}
J^{(k)}=\dprod\limits_{j=1}^{k}h_{1}\left( x_{j}\right) h_{1}(x_{j}^{\prime
})\phi (t,z_{j})\overline{\phi }(t,z_{j}^{\prime }),
\end{equation*}%
and noticing the fact that $\left( \tilde{\gamma}_{N,\omega }^{(k)}\right)
^{2}\leqslant \tilde{\gamma}_{N,\omega }^{(k)}$ since the initial data is
normalized, we see that as Hilbert-Schmidt operators 
\begin{equation*}
\tilde{\gamma}_{N,\omega }^{(k)}\rightarrow
\dprod\limits_{j=1}^{k}h_{1}\left( x_{j}\right) h_{1}(x_{j}^{\prime })\phi
(t,z_{j})\overline{\phi }(t,z_{j}^{\prime })\text{ strongly.}
\end{equation*}%
Since $\limfunc{Tr}\tilde{\gamma}_{N,\omega }^{(k)}=\limfunc{Tr}\tilde{\gamma%
}^{(k)},$ we deduce the strong convergence%
\begin{equation*}
\lim_{\substack{ N,\omega \rightarrow \infty  \\ C_{1}N^{v_{1}(\beta
)}\leqslant \omega \leqslant C_{2}N^{v_{2}(\beta )}}}\limfunc{Tr}\left\vert 
\tilde{\gamma}_{N,\omega }^{(k)}(t,\mathbf{x}_{k},\mathbf{z}_{k};\mathbf{x}%
_{k}^{\prime },\mathbf{z}_{k}^{\prime })-\dprod\limits_{j=1}^{k}h_{1}\left(
x_{j}\right) h_{1}(x_{j}^{\prime })\phi (t,z_{j})\overline{\phi }%
(t,z_{j}^{\prime })\right\vert =0,
\end{equation*}%
via the Gr\"{u}mm's convergence theorem \cite[Theorem 2.19]{Simon}.\footnote{%
One can also use the argument in \cite[Appendix A]{Chen3DDerivation} if one
would like to conclude the convergence with general datum.}
\end{itemize}

\section{Focusing Energy Estimate\label{Section:EnergyEstimate}}

We find it more convenient to prove the energy estimate for $\psi _{N,\omega
}$ and then convert it by scaling to an estimate for $\tilde{\psi}_{N,\omega
}$ (see \eqref{E:rescaled}). Note that, as an operator, we have the
positivity: 
\begin{equation*}
-\Delta _{x_{j}}+\omega ^{2}\left\vert x_{j}\right\vert ^{2}-2\omega
\geqslant 0
\end{equation*}%
Define 
\begin{equation*}
S_{j}\overset{\mathrm{def}}{=}(1-\Delta _{x_{j}}+\omega ^{2}\left\vert
x_{j}\right\vert ^{2}-2\omega -\partial _{z_{j}}^{2})^{1/2}=(1-2\omega
-\Delta _{r_{j}}+\omega ^{2}\left\vert x_{j}\right\vert ^{2})^{1/2},
\end{equation*}%
and write%
\begin{equation*}
S^{(k)}=\dprod\limits_{j=1}^{k}S_{j}.
\end{equation*}

\begin{theorem}[energy estimate]
\label{Theorem:Energy Estimate}For $\beta \in (0,\frac{3}{7})$, let\footnote{%
One notices that $v_{E}(\beta )$ is different from $v_{2}(\beta )$ in the
sense that the term $\frac{2\beta }{1-2\beta }-$ is missing. That
restriction comes from Theorem \ref{Theorem:Convergence to the Coupled
Gross-Pitaevskii}.} 
\begin{equation}
v_{E}(\beta )=\min \left( \frac{1-\beta }{\beta },\frac{\frac{3}{5}-\beta }{%
\beta -\frac{1}{5}}\mathbf{1}_{\beta \geqslant \frac{1}{5}}+\infty \cdot 
\mathbf{1}_{\beta <\frac{1}{5}},\frac{\frac{7}{8}%
-\beta }{\beta }\right) \text{.}  \label{E:vofbetaE}
\end{equation}%
There are constants\footnote{%
By \emph{absolute} constant we mean a constant independent of $V$, $N$, $%
\omega $, etc. Formulas for $C_{1},C_{2}$ in terms of $\Vert V\Vert _{L^{1}}$%
, $\Vert V\Vert _{L^{\infty }}$ can, in principle, be extracted from the
proof.} $C_{1}=C_{1}(\Vert V\Vert _{L^{1}},\Vert V\Vert _{L^{\infty }})$, $%
C_{2}=C_{2}(\Vert V\Vert _{L^{1}},\Vert V\Vert _{L^{\infty }})$, and
absolute constant $C_{3}$, and for each $k\in \mathbb{N}$, there is an
integer $N_{0}(k)$, such that for any $k\in \mathbb{N}$, $N\geq N_{0}(k)$
and $\omega $ which satisfy 
\begin{equation}
C_{1}N^{v_{1}(\beta )}\leqslant \omega \leqslant C_{2}N^{v_{E}(\beta )},
\label{E:3Dto1D}
\end{equation}%
there holds%
\begin{equation}
\langle (\alpha +N^{-1}H_{N,\omega }-2\omega )^{k}\psi ,\psi \rangle
\geqslant \frac{1}{2^{k}}(\Vert S^{(k)}\psi \Vert _{L^{2}}^{2}+N^{-1}\Vert
S_{1}S^{(k-1)}\psi \Vert _{L^{2}}^{2}),  \label{E:energy-nonscaled}
\end{equation}%
where 
\begin{equation*}
\alpha =C_{3}\Vert V\Vert _{L^{1}}^{2}+1.
\end{equation*}
\end{theorem}

\begin{proof}
For smoothness of presentation, we postpone the proof to \S \ref{Section:Pf
of EnergyEstimate}.
\end{proof}

Recall the rescaled operator \eqref{E:tilde-S-def} 
\begin{equation*}
\tilde{S}_{j}=\left[ 1-\partial _{z_{j}}^{2}+\omega \left( -\Delta
_{x_{j}}+\left\vert x_{j}\right\vert ^{2}-2\right) \right] ^{\frac{1}{2}},
\end{equation*}%
we notice that 
\begin{equation*}
(S_{j}\psi )(t,\mathbf{x}_{N},\mathbf{z}_{N})=\omega ^{N/2}(\tilde{S}_{j}%
\tilde{\psi})(t,\sqrt{\omega }\mathbf{x}_{N},\mathbf{z}_{N})\,,
\end{equation*}%
if $\tilde{\psi}_{N,\omega }$ is defined via (\ref{E:rescaled}). Thus we can
convert the conclusion of Theorem \ref{Theorem:Energy Estimate} into
statements about $\tilde{\psi}_{N,\omega }$, $\tilde{S}_{j}$, and $\tilde{%
\gamma}_{N,\omega }^{(k)}$ which we will utilize in the rest of the paper.

\begin{corollary}
\label{Corollary:Energy Bound for Marginal Densities}Define%
\begin{equation*}
\tilde{S}^{(k)}=\prod_{j=1}^{k}\tilde{S}_{j}\text{, }L^{(k)}=\prod_{j=1}^{k}%
\left\langle \nabla _{r_{j}}\right\rangle
\end{equation*}%
Assume $C_{1}N^{v_{1}(\beta )}\leqslant \omega \leqslant C_{2}N^{v_{E}(\beta
)}$. Let $\tilde{\psi}_{N,\omega }(t)=e^{it\tilde{H}_{N,\omega }}\tilde{\psi}%
_{N,\omega }(0)$ and $\{\tilde{\gamma}_{N,\omega }^{(k)}(t)\}$ be the
associated marginal densities, then for all $\omega \geqslant 1$ , $%
k\geqslant 0$, $N\ $large enough, we have the uniform-in-time bound 
\begin{equation}
\func{Tr}\tilde{S}^{(k)}\tilde{\gamma}_{N,\omega }^{(k)}\tilde{S}%
^{(k)}=\left\Vert \tilde{S}^{(k)}\tilde{\psi}_{N,\omega }(t)\right\Vert
_{L^{2}(\mathbb{R}^{3N})}^{2}\leqslant C^{k}.  \label{E:e-1}
\end{equation}%
Consequently, 
\begin{equation}
\func{Tr}L^{(k)}\tilde{\gamma}_{N,\omega }^{(k)}L^{(k)}=\left\Vert L^{(k)}%
\tilde{\psi}_{N,\omega }(t)\right\Vert _{L^{2}(\mathbb{R}^{3N})}^{2}%
\leqslant C^{k},  \label{E:e-2}
\end{equation}%
and 
\begin{equation}
\Vert \mathcal{P}_{\mathbf{\alpha }}\tilde{\psi}_{N,\omega }\Vert _{L^{2}(%
\mathbb{R}^{3N})}\leqslant C^{k}\omega ^{-|\mathbf{\alpha }|/2}\,\text{, }%
\left\vert \func{Tr}\mathcal{P}_{\mathbf{\alpha }}\tilde{\gamma}_{N,\omega
}^{(k)}\mathcal{P}_{\mathbf{\beta }}\right\vert \leqslant C^{k}\omega ^{-%
\frac{1}{2}|\mathbf{\alpha }|-\frac{1}{2}\mathbf{|\beta }|}  \label{E:e-3}
\end{equation}%
where $\mathcal{P}_{\mathbf{\alpha }}$ and $\mathcal{P}_{\mathbf{\beta }}$
are defined as in (\ref{def:multiple projection for omega=1}).
\end{corollary}

\begin{proof}
Substituting \eqref{E:rescaled} into estimate \eqref{E:energy-nonscaled} and
rescaling, we obtain%
\begin{equation*}
\left\Vert \tilde{S}^{(k)}\tilde{\psi}_{N,\omega }(t)\right\Vert _{L^{2}(%
\mathbb{R}^{3N})}^{2}\leqslant C^{k}\langle \tilde{\psi}_{N,\omega
}(t),(\alpha +N^{-1}\tilde{H}_{N,\omega }-2\omega )^{k}\tilde{\psi}%
_{N,\omega }(t)\rangle .
\end{equation*}%
The quantity on the right hand side is conserved, therefore%
\begin{equation*}
=C^{k}\langle \tilde{\psi}_{N,\omega }(0),(\alpha +N^{-1}\tilde{H}_{N,\omega
}-2\omega )^{k}\tilde{\psi}_{N,\omega }(0)\rangle .
\end{equation*}%
Apply the binomial theorem twice, 
\begin{eqnarray*}
&\leqslant &C^{k}\sum_{j=0}^{k}%
\begin{pmatrix}
k \\ 
j%
\end{pmatrix}%
\alpha ^{j}\langle \tilde{\psi}_{N,\omega }(0),(N^{-1}\tilde{H}_{N,\omega
}-2\omega )^{k-j}\tilde{\psi}_{N,\omega }(0)\rangle \\
&\leqslant &C^{k}\sum_{j=0}^{k}%
\begin{pmatrix}
k \\ 
j%
\end{pmatrix}%
\alpha ^{j}\left( C\right) ^{k-j} \\
&=&C^{k}\left( \alpha +C\right) ^{k}\leqslant \tilde{C}^{k}.
\end{eqnarray*}%
where we used condition (\ref{Condition:EnergyBoundOnInitialData}) in the
second to last line. So we have proved (\ref{E:e-1}). Putting \eqref{E:e-1}
and \eqref{E:tilde-S-1} together, estimate \eqref{E:e-2} then follows.%
\footnote{%
We remark that, though $L^{(k)}\leqslant 3^{k}\tilde{S}^{(k)}$, it is not
true that $L^{(k)}\leqslant C^{k}S^{(k)}$ for any $C$ independent of $\omega 
$ because of the ground state case.} The first inequality of \eqref{E:e-3}
follows from \eqref{E:e-1} and \eqref{E:tilde-S-3}. By Lemma \ref%
{L:trace-of-tp-kernel}, $\func{Tr}\mathcal{P}_{\mathbf{\alpha }}\tilde{\gamma%
}_{N,\omega }^{(k)}\mathcal{P}_{\mathbf{\beta }}=\langle \mathcal{P}_{%
\mathbf{\alpha }}\tilde{\psi}_{N,\omega },\mathcal{P}_{\mathbf{\beta }}%
\tilde{\psi}_{N,\omega }\rangle $, so the second inequality of \eqref{E:e-3}
follows by Cauchy-Schwarz.
\end{proof}

\subsection{Proof of the Focusing Energy Estimate\label{Section:Pf of
EnergyEstimate}}

Note that 
\begin{equation*}
N^{-1}H_{N,\omega }-2\omega =N^{-1}\sum_{i=1}^{N}(-\Delta _{r_{i}}+\omega
^{2}|x_{i}|^{2}-2\omega )+N^{-2}\omega ^{-1}\sum_{1\leq i<j\leq N}V_{N\omega
}(r_{i}-r_{j}),
\end{equation*}%
where we have used the notation\footnote{%
We remind the reader that this $V_{N\omega }$ is different from $V_{N,\omega
}$ defined in (\ref{E:V}).} 
\begin{equation*}
V_{N\omega }(r)=(N\omega )^{3\beta }V((N\omega )^{\beta }r).
\end{equation*}

Define 
\begin{equation*}
H_{Kij}=(\alpha -\Delta _{r_{i}}+\omega ^{2}|x_{i}|^{2}-2\omega )+(\alpha
-\Delta _{r_{j}}+\omega ^{2}|x_{j}|^{2}-2\omega )
\end{equation*}%
where the $K$ stands for \textquotedblleft kinetic\textquotedblright\ and 
\begin{equation*}
H_{Iij}=\omega ^{-1}V_{N\omega ij}=\omega ^{-1}V_{N\omega }(r_{i}-r_{j})
\end{equation*}%
where the $I$ is for \textquotedblleft interaction\textquotedblright . If we
write%
\begin{equation*}
H_{ij}=H_{Kij}+H_{Iij},
\end{equation*}%
then 
\begin{equation}
\alpha +N^{-1}H_{N,\omega }-2\omega =\frac{1}{2}N^{-2}\sum_{1\leq i\neq
j\leq N}H_{ij}=N^{-2}\sum_{1\leq i<j\leq N}H_{ij}.  \label{E:3Dto1D6}
\end{equation}%
We will first prove Theorem \ref{Theorem:Energy Estimate} for $k=1$ and $k=2$%
. Then, by a two-step induction (result known for $k$ implies result for $%
k+2 $), we establish the general case. Before we proceed, we prove some
estimates regarding the Hermite operator.

\subsubsection{Estimates Needed to Prove Theorem \protect\ref{Theorem:Energy
Estimate}}

\begin{lemma}
Let $P_{l\omega }$ be defined as in (\ref{def:I projection for omega}).
There is a constant independent of $\ell $ and $\omega $ such that 
\begin{equation}
\Vert P_{\ell \omega }f\Vert _{L_{x}^{\infty }}\leqslant C\omega ^{1/2}\Vert
f\Vert _{L_{x}^{2}}.  \label{E:3Dto1D4}
\end{equation}%
with constant independent of $\ell $ and $\omega $.
\end{lemma}

\begin{proof}
This estimate has more than one proof. It is a special result in 2D. It does
not follow from the Strichartz estimates. For a modern argument which proves
the estimate for general at most quadratic potentials, see \cite[Corollary
2.2]{KochTataru}.  In the special case of the quantum harmonic oscillator, one can also use a special
property of 2D Hermite projection kernels to yield a direct proof without using Littlewood-Paley theory -- see  \cite[Lemma 3.2.2]{Thangavelu}, \cite[Remark 8]{ChenDie}.
\end{proof}

\begin{lemma}
\label{L:potential-bound}There is an absolute constant $C_{3}>0$ and a
constant $C_{1}=C\left( \Vert V\Vert _{L^{1}},\Vert V\Vert _{L^{\infty
}}\right) $ such that if 
\begin{equation*}
\omega \geq C_{1}N^{\beta /(1-\beta )}
\end{equation*}%
then 
\begin{eqnarray}
&&\frac{1}{\omega }\int \left\vert V_{N\omega }(r_{1}-r_{2})\right\vert
\left\vert \psi (r_{1},r_{2})\right\vert ^{2}dr_{1}  \label{E:3Dto1D14} \\
&\leqslant &\frac{1}{100}\left\langle \psi (r_{1},r_{2}),(-\Delta
_{r_{1}}+\omega ^{2}|x_{1}|^{2}-2\omega )\psi (r_{1},r_{2})\right\rangle
_{r_{1}}+C_{3}\left\Vert V\right\Vert _{L^{1}}^{2}\left\Vert \psi
(r_{1},r_{2})\right\Vert _{L_{r_{1}}^{2}}^{2}.  \notag
\end{eqnarray}%
The above estimate is performed in one coordinate only (taken to be $r_{1}$%
), and the other coordinate $r_{2}$ are effectively \textquotedblleft
frozen\textquotedblright . In particular, let 
\begin{equation*}
f(r_{2},\ldots ,r_{N})=\int \left\vert V_{N\omega }(r_{1}-r_{2})\right\vert
\left\vert \psi _{1}(r_{1},\ldots ,r_{N})\right\vert \left\vert \psi
_{2}(r_{1},\ldots ,r_{N})\right\vert dr_{1}
\end{equation*}%
Then 
\begin{equation}
f(r_{2},\ldots ,r_{N})\lesssim \omega \Vert S_{1}\psi _{1}(r_{1},\cdots
,r_{N})\Vert _{L_{r_{1}}^{2}}\Vert S_{1}\psi _{2}(r_{1},\cdots ,r_{N})\Vert
_{L_{r_{1}}^{2}},  \label{E:3Dto1D11}
\end{equation}%
The implicit constant in $\lesssim $ is an absolute constant times $\Vert
V\Vert _{L^{1}}+\Vert V\Vert _{L^{\infty }}$.
\end{lemma}

\begin{proof}
By Cauchy-Schwarz, 
\begin{equation*}
\int |V_{N\omega 12}|\,|\psi _{1}|\,|\psi _{2}|dr_{1}\leqslant \left( \int
|V_{N\omega 12}||\psi _{1}|^{2}\,dr_{1}\right) ^{1/2}\left( \int |V_{N\omega
12}||\psi _{2}|^{2}\,dr_{1}\right) ^{1/2}.
\end{equation*}%
Thus, assuming (\ref{E:3Dto1D14}) and using the facts that%
\begin{eqnarray*}
S_{1}^{2} &\geqslant &1, \\
S_{1}^{2} &\geqslant &(-\Delta _{r_{1}}+\omega ^{2}|x_{1}|^{2}-2\omega ),
\end{eqnarray*}%
we obtain (\ref{E:3Dto1D11}). So we only need to to prove (\ref{E:3Dto1D14}).

Taking $P_{l\omega }$ to be the projection onto the $x_{1}$ component, we
decompose $\psi $ into ground state, middle energies, and high energies as
follows: 
\begin{equation*}
\psi =P_{0\omega }\psi +\sum_{\ell =1}^{e-1}P_{\ell \omega }\psi +P_{\geq
e\omega }\psi
\end{equation*}%
where $e$ is an integer, and the optimal choice of $e$ is determined below.
It then suffices to bound 
\begin{equation}
A_{\text{low}}\ \overset{\mathrm{def}}{=}\frac{1}{\omega }\int |V_{N\omega
}(r_{1}-r_{2})||P_{0\omega }\psi (r_{1},r_{2})|^{2}dr_{1}  \label{E:3Dto1D1}
\end{equation}%
\begin{equation}
A_{\text{mid}}\ \overset{\mathrm{def}}{=}\frac{1}{\omega }\int |V_{N\omega
}(r_{1}-r_{2})||\sum_{\ell =2}^{e-1}P_{\ell \omega }\psi
(r_{1},r_{2})|^{2}dr_{1}  \label{E:3Dto1D2}
\end{equation}%
\begin{equation}
A_{\text{high}}\ \overset{\mathrm{def}}{=}\frac{1}{\omega }\int |V_{N\omega
}(r_{1}-r_{2})||P_{\geq e\omega }\psi (r_{1},r_{2})|^{2}dr_{1}
\label{E:3Dto1D3}
\end{equation}%
For each estimate, we will only work in the $r_{1}=(x_{1},z_{1})$ component,
and thus will not even write the $r_{2}$ variable. First we consider %
\eqref{E:3Dto1D1}. 
\begin{equation*}
A_{\text{low}}\ \leqslant \frac{1}{\omega }\Vert V_{N\omega }\Vert
_{L^{1}}\Vert P_{0\omega }\psi \Vert _{L_{x}^{\infty }L_{z}^{\infty }}^{2}
\end{equation*}%
By the standard 1D Sobolev-type estimate 
\begin{equation*}
A_{\text{low}}\lesssim \frac{1}{\omega }\Vert V\Vert _{L^{1}}\Vert
P_{0\omega }\partial _{z}\psi \Vert _{L_{x}^{\infty }L_{z}^{2}}\Vert
P_{0\omega }\psi \Vert _{L_{x}^{\infty }L_{z}^{2}}
\end{equation*}%
Then use the estimate \eqref{E:3Dto1D4} 
\begin{align*}
A_{\text{low}}\ & \lesssim \Vert V\Vert _{L^{1}}\Vert P_{0\omega }\partial
_{z}\psi \Vert _{L_{r}^{2}}\Vert P_{0\omega }\psi \Vert _{L_{r}^{2}} \\
& \lesssim \Vert V\Vert _{L^{1}}\Vert \partial _{z}\psi \Vert _{L^{2}}\Vert
\psi \Vert _{L^{2}} \\
& \lesssim \epsilon \Vert \partial _{z}\psi \Vert _{L^{2}}^{2}+\frac{\Vert
V\Vert _{L^{1}}^{2}}{\epsilon }\Vert \psi \Vert _{L^{2}}^{2}\text{.}
\end{align*}%
Since, $(-\Delta _{r}+\omega ^{2}|x|^{2}-2\omega )$ is a sum of two positive
operators, namely, $-\Delta _{x}+\omega ^{2}|x|^{2}-2\omega $ and $-\partial
_{z}^{2}$, we conclude the estimate for $A_{\text{low}}$.

Now consider the middle harmonic energies given by \eqref{E:3Dto1D2}, and we
aim to estimate $A_{\text{mid}}$. For any $\ell \geq 1$, we have 
\begin{equation*}
\Vert P_{\ell \omega }\psi \Vert _{L_{z}^{\infty }L_{x}^{\infty }}\leq \Vert
P_{\ell \omega }\partial _{z}\psi \Vert _{L_{z}^{2}L_{x}^{\infty
}}^{1/2}\Vert P_{\ell \omega }\psi \Vert _{L_{z}^{2}L_{x}^{\infty }}^{1/2}
\end{equation*}%
By \eqref{E:3Dto1D4}, 
\begin{align*}
\Vert P_{\ell \omega }\psi \Vert _{L_{z}^{\infty }L_{x}^{\infty }}& \lesssim
\omega ^{1/2}\Vert P_{\ell \omega }\partial _{z}\psi \Vert
_{L_{z}^{2}L_{x}^{2}}^{1/2}\Vert P_{\ell \omega }\psi \Vert
_{L_{z}^{2}L_{x}^{2}}^{1/2} \\
& =\omega ^{1/4}\Vert P_{\ell \omega }\partial _{z}\psi \Vert
_{L^{2}}^{1/2}(\Vert P_{\ell \omega }\psi \Vert _{L^{2}}\ell ^{1/2}\omega
^{1/2})^{1/2}\,\ell ^{-1/4} \\
& =\omega ^{1/4}\Vert P_{\ell \omega }\partial _{z}\psi \Vert
_{L_{r}^{2}}^{1/2}\Vert P_{\ell \omega }(-\Delta _{x}+\omega
^{2}|x|^{2}-2\omega )^{1/2}\psi \Vert _{L^{2}}^{1/2}\,\ell ^{-1/4}
\end{align*}%
Sum over $1\leq \ell \leq e-1$, and do H\"{o}lder with exponents $4$, $4$,
and $2$: 
\begin{align*}
\sum_{\ell =1}^{e-1}\Vert P_{\ell \omega }\psi \Vert _{L_{z}^{\infty
}L_{x}^{\infty }}& \lesssim \omega ^{1/4}\left( \sum_{\ell =1}^{e-1}\Vert
P_{\ell \omega }\partial _{z}\psi \Vert _{L^{2}}^{2}\right) ^{1/4} \\
& \quad \quad \times \left( \sum_{\ell =1}^{e-1}\Vert P_{\ell \omega
}(-\Delta _{x}+\omega ^{2}|x|^{2}-2\omega )^{1/2}\psi \Vert
_{L^{2}}^{2}\right) ^{1/4}\left( \sum_{\ell =1}^{e}\ell ^{-1/2}\right) ^{1/2}
\\
& \lesssim \omega ^{1/4}e^{1/4}\Vert \partial _{z}\psi \Vert
_{L^{2}}^{1/2}\Vert (-\Delta _{x}+\omega ^{2}|x|^{2}-2\omega )^{1/2}\psi
\Vert _{L^{2}}^{1/2}
\end{align*}%
Applying this to estimate \eqref{E:3Dto1D2}, 
\begin{equation*}
A_{\text{mid}}\lesssim \omega ^{-1/2}e^{1/2}\Vert V\Vert _{L^{1}}\Vert
\partial _{z}\psi \Vert _{L^{2}}\Vert (-\Delta _{x}+\omega
^{2}|x|^{2}-2\omega )^{1/2}\psi \Vert _{L^{2}}
\end{equation*}%
Take $e$ so that $\omega ^{-1/2}e^{1/2}\Vert V\Vert _{L^{1}}=\epsilon $,
i.e. 
\begin{equation}
e=\frac{\epsilon ^{2}}{\Vert V\Vert _{L^{1}}^{2}}\omega  \label{E:3Dto1D5}
\end{equation}%
and then we have 
\begin{equation*}
A_{\text{mid}}\ \lesssim \epsilon \Vert \partial _{z}\psi \Vert
_{L^{2}}^{2}+\epsilon \Vert (-\Delta _{x}+\omega ^{2}|x|^{2}-2\omega
)^{1/2}\psi \Vert _{L^{2}}^{2}
\end{equation*}%
For \eqref{E:3Dto1D3}, 
\begin{align*}
A_{\text{high}}& \lesssim \omega ^{-1}\Vert V_{N\omega }\Vert _{L^{\infty
}}\Vert P_{\geq e\omega }\psi \Vert _{L^{2}}^{2} \\
& \lesssim \omega ^{-2}e^{-1}\Vert V_{N\omega }\Vert _{L^{\infty }}\Vert
e^{1/2}\omega ^{1/2}P_{\geq e\omega }\psi \Vert _{L^{2}}^{2} \\
& \lesssim \omega ^{-2}e^{-1}(N\omega )^{3\beta }\Vert V\Vert _{L^{\infty
}}\Vert (-\Delta _{x}+\omega ^{2}|x|^{2}-2\omega )^{1/2}\psi \Vert
_{L^{2}}^{2}
\end{align*}%
We need 
\begin{equation*}
\omega ^{-2}e^{-1}(N\omega )^{3\beta }\leq \epsilon
\end{equation*}%
Substituting the specification of $e$ given by \eqref{E:3Dto1D5}, we obtain 
\begin{equation*}
N^{3\beta }\omega ^{3\beta -3}\leq \frac{\epsilon ^{2}}{\Vert V\Vert
_{L^{1}}^{2}\Vert V\Vert _{L^{\infty }}}.
\end{equation*}%
That is $\omega \geq C_{1}N^{\beta /(1-\beta )}$ as required in the
statement of Lemma \ref{L:potential-bound}.
\end{proof}

In the following lemma, we have excited state estimates and ground state
estimates, and the ground state estimates are weaker (involve a loss of $%
\omega^{1/2}$)

\begin{lemma}
Taking $\psi =\psi (r)$, we have the following \textquotedblleft excited
state\textquotedblright\ estimate: 
\begin{equation}
\Vert \omega ^{1/2}P_{\geq 1\omega }\psi \Vert _{L^{2}}+\Vert \omega
|x|P_{\geq 1\omega }\psi \Vert _{L^{2}}+\Vert \nabla _{r}P_{\geq 1\omega
}\psi \Vert _{L^{2}}\lesssim \Vert S\psi \Vert _{L^{2}},  \label{E:3Dto1D7}
\end{equation}%
and the following \textquotedblleft ground state\textquotedblright\ estimate 
\begin{equation}
\Vert \omega ^{1/2}P_{0\omega }\psi \Vert _{L^{2}}+\Vert \omega
|x|P_{0\omega }\psi \Vert _{L^{2}}+\Vert \nabla _{r}P_{0\omega }\psi \Vert
_{L^{2}}\lesssim \omega ^{1/2}\Vert \psi \Vert _{L^{2}}  \label{E:3Dto1D8}
\end{equation}%
We are, however, spared from the $\omega ^{1/2}$ loss when working only with
the $z$-derivative 
\begin{equation}
\Vert \partial _{z}P_{0\omega }\psi \Vert _{L^{2}}\lesssim \Vert S\psi \Vert
_{L^{2}}  \label{E:3Dto1D9}
\end{equation}%
Putting the excited state and ground state estimates together gives 
\begin{equation}
\Vert \omega ^{1/2}\psi \Vert _{L^{2}}+\Vert \omega |x|\psi \Vert
_{L^{2}}+\Vert \nabla _{r}\psi \Vert _{L^{2}}\lesssim \omega ^{1/2}\Vert
S\psi \Vert _{L^{2}}  \label{E:3Dto1D10}
\end{equation}
\end{lemma}

\begin{proof}
For the excited state estimates, we note 
\begin{equation*}
0\leq \langle P_{\geq 1\omega }\psi ,(-\Delta _{x}+\omega
^{2}|x|^{2}-4\omega )P_{\geq 1\omega }\psi \rangle
\end{equation*}%
Adding $\frac{3}{2}\Vert \partial _{z}P_{\geq 1\omega }\psi \Vert
_{L^{2}}^{2}+\frac{1}{2}\Vert \nabla _{x}P_{\geq 1\omega }\psi \Vert
_{L^{2}}^{2}+\frac{1}{2}\Vert \omega |x|P_{\geq 1\omega }\psi \Vert
_{L^{2}}^{2}+\Vert \omega ^{1/2}P_{\geq 1\omega }\psi \Vert _{L^{2}}^{2}$ to
both sides 
\begin{equation*}
\frac{3}{2}\Vert \partial _{z}P_{\geq 1\omega }\psi \Vert _{L^{2}}^{2}+\frac{%
1}{2}\Vert \nabla _{x}P_{\geq 1\omega }\psi \Vert _{L^{2}}^{2}+\frac{1}{2}%
\Vert \omega |x|P_{\geq 1\omega }\psi \Vert _{L^{2}}^{2}+\Vert \omega
^{1/2}P_{\geq 1\omega }\psi \Vert _{L^{2}}^{2}
\end{equation*}%
\begin{equation*}
\leq \frac{3}{2}\langle P_{\geq 1\omega }\psi ,(-\Delta _{r}+\omega
^{2}|x|^{2}-2\omega )P_{\geq 1\omega }\psi \rangle
\end{equation*}%
This proves \eqref{E:3Dto1D7}. The ground state estimate \eqref{E:3Dto1D8}
and \eqref{E:3Dto1D9} are straightforward from the explicit definition of $%
P_{0\omega }$ which is merely projecting onto a Gaussian.
\end{proof}

\begin{lemma}
\label{L:higher-k-2}We have the following estimates:%
\begin{eqnarray}
\Vert |V_{N\omega 12}|^{1/2}S_{1}P_{0\omega }^{1}\psi _{2}\Vert
_{L_{r_{1}}^{2}} &\lesssim &\omega ^{\frac{1}{2}}N^{\frac{1}{4}}\Vert
S_{1}\psi _{2}\Vert _{L^{2}}^{1/2}\left( N^{-\frac{1}{4}}\Vert S_{1}^{2}\psi
_{2}\Vert _{L^{2}}^{1/2}\right)  \label{e:higher-k-2-1} \\
\Vert |V_{N\omega 12}|^{1/2}S_{1}P_{\geqslant 1\omega }^{1}\psi _{2}\Vert
_{L_{r_{1}}^{2}} &\lesssim &N^{\frac{\beta }{2}+\frac{1}{2}}\omega ^{\frac{%
\beta }{2}}\left( N^{-1/2}\Vert S_{1}^{2}\psi _{2}\Vert
_{L_{r_{1}}^{2}}\right)  \label{e:higher-k-2-2}
\end{eqnarray}%
In particluar, if $\omega \geqslant C_{1}N^{\beta /(1-\beta )}$ then%
\begin{eqnarray}
&&\int_{r_{1}}|V_{N\omega 12}||\psi _{1}||S_{1}\psi _{2}|\,dr_{1}
\label{e:higher-k-2-3} \\
&\lesssim &\omega N^{\frac{1}{4}}\left\Vert S_{1}\psi _{1}\right\Vert
_{L^{2}}\left\Vert S_{1}\psi _{2}\right\Vert _{L^{2}}^{\frac{1}{2}}N^{-\frac{%
1}{4}}\left\Vert S_{1}^{2}\psi _{2}\right\Vert _{L^{2}}^{\frac{1}{2}}  \notag
\\
&&+\left( N\omega \right) ^{\frac{\beta }{2}+\frac{1}{2}}\left\Vert
S_{1}\psi _{1}\right\Vert _{L^{2}}N^{-\frac{1}{2}}\left\Vert S_{1}^{2}\psi
_{2}\right\Vert _{L^{2}}  \notag
\end{eqnarray}
\end{lemma}

\begin{proof}
To prove (\ref{e:higher-k-2-3}), substituting $\psi _{2}=P_{0\omega
}^{1}\psi _{2}+P_{\geq 1\omega }^{1}\psi _{2}$, we obtain 
\begin{equation*}
\int_{r_{1}}|V_{N\omega 12}||\psi _{1}||S_{1}\psi _{2}|dr_{1}\lesssim
F_{1}+F_{2}
\end{equation*}%
where%
\begin{eqnarray*}
F_{1} &=&\int_{r_{1}}|V_{N\omega 12}||\psi _{1}||P_{0\omega }^{1}S_{1}\psi
_{2}|dr_{1} \\
&\leqslant &\Vert |V_{N\omega 12}|^{1/2}\psi _{1}\Vert _{L_{r_{1}}^{2}}\Vert
|V_{N\omega 12}|^{1/2}P_{0\omega }^{1}S_{1}\psi _{2}\Vert _{L_{r_{1}}^{2}} \\
&\leqslant &\omega ^{1/2}\Vert S_{1}\psi _{1}\Vert _{L_{r_{1}}^{2}}\Vert
|V_{N\omega 12}|^{1/2}P_{0\omega }^{1}S_{1}\psi _{2}\Vert _{L_{r_{1}}^{2}}
\end{eqnarray*}%
\begin{eqnarray*}
F_{2} &=&\int_{r_{1}}|V_{N\omega 12}||\psi _{1}||P_{\geq 1\omega
}^{1}S_{1}\psi _{2}|dr_{1} \\
&\leqslant &\omega ^{1/2}\Vert S_{1}\psi _{1}\Vert _{L_{r_{1}}^{2}}\Vert
|V_{N\omega 12}|^{1/2}P_{\geq 1\omega }^{1}S_{1}\psi _{2}\Vert
_{L_{r_{1}}^{2}}
\end{eqnarray*}
by Cauchy-Schwarz and estimate (\ref{E:3Dto1D11}). Hence we only need to
prove (\ref{e:higher-k-2-1}) and (\ref{e:higher-k-2-2}).

On the one hand, use the fact that $P_{0\omega }^{1}S_{1}=(1-\partial
_{z_{1}}^{2})^{1/2}P_{0\omega }^{1}$, 
\begin{eqnarray*}
\Vert |V_{N\omega 12}|^{1/2}S_{1}P_{0\omega }^{1}\psi _{2}\Vert
_{L_{r_{1}}^{2}} &=&\Vert |V_{N\omega 12}|^{1/2}(1-\partial
_{z_{1}}^{2})^{1/2}P_{0\omega }^{1}\psi _{2}\Vert _{L_{r_{1}}^{2}} \\
&\leq &\Vert V_{N\omega 12}\Vert _{L_{r_{1}}^{1}}^{\frac{1}{2}}\Vert
(1-\partial _{z_{1}}^{2})^{1/2}P_{0\omega }^{1}\psi _{2}\Vert
_{L_{r_{1}}^{\infty }}
\end{eqnarray*}%
By Sobolev in $z_{1}$ and the estimate \eqref{E:3Dto1D4} in $x_{1}$, 
\begin{equation*}
\Vert |V_{N\omega 12}|^{1/2}S_{1}P_{0\omega }^{1}\psi _{2}\Vert
_{L_{r_{1}}^{2}}\lesssim \omega ^{1/2}\Vert (1-\partial
_{z_{1}}^{2})^{1/2}\psi _{2}\Vert _{L_{r_{1}}^{2}}^{1/2}\Vert (1-\partial
_{z_{1}}^{2})\psi _{2}\Vert _{L_{r_{1}}^{2}}^{1/2}
\end{equation*}%
That is (\ref{e:higher-k-2-1}): 
\begin{equation*}
\Vert |V_{N\omega 12}|^{1/2}S_{1}P_{0\omega }^{1}\psi _{2}\Vert
_{L_{r_{1}}^{2}}\lesssim \omega ^{\frac{1}{2}}N^{1/4}\Vert S_{1}\psi
_{2}\Vert _{L^{2}}^{1/2}\left( N^{-1/4}\Vert S_{1}^{2}\psi _{2}\Vert
_{L^{2}}^{1/2}\right)
\end{equation*}

On the other hand,%
\begin{eqnarray*}
&&\Vert |V_{N\omega 12}|^{1/2}S_{1}P_{\geqslant 1\omega }^{1}\psi _{2}\Vert
_{L_{r_{1}}^{2}} \\
&\lesssim &\left\Vert |V_{N\omega 12}|^{1/2}\right\Vert _{L^{3}}\left\Vert
P_{\geq 1\omega }^{1}S_{1}\psi _{2}\right\Vert _{L_{r_{1}}^{6}} \\
&\lesssim &(N\omega )^{\beta /2}\Vert S_{1}^{2}\psi _{2}\Vert
_{L_{r_{1}}^{2}} \\
&=&N^{\frac{\beta }{2}+\frac{1}{2}}\omega ^{\frac{\beta }{2}}\left(
N^{-1/2}\Vert S_{1}^{2}\psi _{2}\Vert _{L_{r_{1}}^{2}}\right)
\end{eqnarray*}%
which is (\ref{e:higher-k-2-2}).
\end{proof}

\subsubsection{The $k=1$ Case}

Recall (\ref{E:3Dto1D6}),

\begin{equation*}
\langle \psi ,(\alpha +N^{-1}H_{N,\omega }-2\omega )\psi \rangle =\tfrac{1}{2%
}N^{-2}\sum_{1\leq i\neq j\leq N}\langle H_{ij}\psi ,\psi \rangle
\end{equation*}%
By symmetry 
\begin{equation*}
=\frac{1}{2}\langle H_{12}\psi ,\psi \rangle
\end{equation*}%
Hence we need to prove 
\begin{equation}
\langle H_{12}\psi ,\psi \rangle \geqslant \Vert S_{1}\psi \Vert
_{L^{2}}^{2}.  \label{E:3Dto1D16}
\end{equation}%
We prove (\ref{E:3Dto1D16}) with the following lemma.

\begin{lemma}
\label{L:higher-k-3}Recall $\alpha =C_{3}\Vert V\Vert _{L^{2}}^{2}+1$. If $%
\omega \geq C_{1}N^{\beta /(1-\beta )}$ and $\psi _{j}(r_{1},r_{2})=\psi
_{j}(r_{2},r_{1})$ for $j=1,2$, then 
\begin{equation}
\left\vert \langle H_{12}\psi _{1},\psi _{2}\rangle _{r_{1}r_{2}}\right\vert
\lesssim \Vert S_{1}\psi _{1}\Vert _{L_{r_{1}r_{2}}^{2}}\Vert S_{1}\psi
_{2}\Vert _{L_{r_{1}r_{2}}^{2}}  \label{E:3Dto1D18}
\end{equation}%
Moreover 
\begin{equation}
\Vert S_{1}\psi \Vert _{L^{2}}^{2}\leqslant \langle H_{12}\psi ,\psi \rangle
\leqslant C\Vert S_{1}\psi \Vert _{L^{2}}^{2}  \label{E:3Dto1D17}
\end{equation}
\end{lemma}

\begin{proof}
By Cauchy-Schwarz and \eqref{E:3Dto1D14},%
\begin{eqnarray*}
\left\vert \langle \psi _{1},H_{I12}\psi _{2}\rangle
_{r_{1}r_{2}}\right\vert &=&\omega ^{-1}\left\vert \left\langle V_{N\omega
12}\psi _{1},\psi _{2}\right\rangle \right\vert \\
&\lesssim &\left( \omega ^{-1}\int |V_{N\omega 12}||\psi _{1}|^{2}\right)
^{1/2}\left( \omega ^{-1}\int |V_{N\omega 12}||\psi _{2}|^{2}\right) ^{1/2}
\\
&\lesssim &\Vert S_{1}\psi _{1}\Vert _{L^{2}}\Vert S_{1}\psi _{2}\Vert
_{L^{2}}
\end{eqnarray*}%
Thus%
\begin{eqnarray*}
\left\vert \langle H_{12}\psi _{1},\psi _{2}\rangle _{r_{1}r_{2}}\right\vert
&\leqslant &\left\vert \langle H_{K12}\psi _{1},\psi _{2}\rangle
_{r_{1}r_{2}}\right\vert +\left\vert \langle H_{I12}\psi _{1},\psi
_{2}\rangle _{r_{1}r_{2}}\right\vert \\
&\lesssim &\Vert S_{1}\psi _{1}\Vert _{L_{r_{1}r_{2}}^{2}}\Vert S_{1}\psi
_{2}\Vert _{L_{r_{1}r_{2}}^{2}}\text{,}
\end{eqnarray*}%
which is \eqref{E:3Dto1D18}. It remains to prove the first inequality in %
\eqref{E:3Dto1D17}.

On the one hand, by \eqref{E:3Dto1D14}, we have the lower bound for the
potential term: 
\begin{equation*}
-\frac{1}{100}\langle \psi ,(-\Delta _{r_{1}}+\omega ^{2}|x_{1}|^{2}-2\omega
)\psi \rangle _{r_{1}r_{2}}-C_{3}\Vert V\Vert _{L^{1}}^{2}\Vert \psi \Vert
_{L_{r_{1}r_{2}}^{2}}^{2}\leq \omega ^{-1}\langle V_{N\omega 12}\psi ,\psi
\rangle _{r_{1}r_{2}}
\end{equation*}%
Adding $\langle \psi ,(\alpha -\Delta _{r_{1}}+\omega
^{2}|x_{1}|^{2}-2\omega )\psi \rangle _{r_{1}r_{2}}$ to both sides and
noticing the trivial inequalities: $\alpha -C_{3}\Vert V\Vert
_{L^{2}}^{2}=1\geqslant \frac{1}{2}$ and $\frac{99}{100}\geqslant \frac{1}{2}
$, we have 
\begin{equation}
\frac{1}{2}\langle \psi ,(1-\Delta _{r_{1}}+\omega ^{2}|x_{1}|^{2}-2\omega
)\psi \rangle _{r_{1}r_{2}}\leqslant \langle \psi ,(\alpha -\Delta
_{r_{1}}+\omega ^{2}|x_{1}|^{2}-2\omega +\omega ^{-1}V_{N\omega 12})\psi
\rangle _{r_{1}r_{2}}.  \label{estimate:part 1 of lower bound in k=1}
\end{equation}

On the other hand, we trivially have%
\begin{equation}
\frac{1}{2}\langle \psi ,(1-\Delta _{r_{2}}+\omega ^{2}|x_{2}|^{2}-2\omega
)\psi \rangle _{r_{1}r_{2}}\leqslant \langle \psi ,(\alpha -\Delta
_{r_{2}}+\omega ^{2}|x_{2}|^{2}-2\omega )\psi \rangle _{r_{1}r_{2}}
\label{estimate:part 2 of lower bound in k=1}
\end{equation}%
because $\alpha >\frac{1}{2}$.

Adding estimates (\ref{estimate:part 1 of lower bound in k=1}) and (\ref%
{estimate:part 2 of lower bound in k=1}) together, we have 
\begin{equation*}
\frac{1}{2}\langle \psi ,S_{1}^{2}\psi \rangle +\frac{1}{2}\langle \psi
,S_{2}^{2}\psi \rangle \leqslant \langle H_{12}\psi ,\psi \rangle .
\end{equation*}%
By symmetry in $r_{1}$ and $r_{2}$, this is precisely \eqref{E:3Dto1D17}.
\end{proof}

\subsubsection{The $k=2$ Case}

The $k=2$ energy estimate is the lower bound 
\begin{equation*}
\frac{1}{4}(\langle S_{1}^{2}S_{2}^{2}\psi ,\psi \rangle +N^{-1}\langle
S_{1}^{4}\psi ,\psi \rangle )\leq \langle (\alpha +N^{-1}H-2\omega )^{2}\psi
,\psi \rangle
\end{equation*}%
We will prove it under the hypothesis 
\begin{equation*}
N^{\beta /(1-\beta )}\leq \omega \leq N^{\min {((1-\beta )/\beta ,2})}
\end{equation*}

We substitute \eqref{E:3Dto1D6} to obtain%
\begin{eqnarray*}
\langle (\alpha +N^{-1}H-2\omega )^{2}\psi ,\psi \rangle &=&\frac{1}{4}%
N^{-4}\sum_{\substack{ 1\leqslant i_{1}\neq j_{1}\leqslant N  \\ 1\leqslant
i_{2}\neq j_{2}\leqslant N}}\langle H_{i_{1}j_{1}}H_{i_{2}j_{2}}\psi ,\psi
\rangle \\
&=&A_{1}+A_{2}+A_{3}
\end{eqnarray*}%
where

\begin{itemize}
\item $A_1$ consists of those terms with $\{i_1,j_1\} \cap \{i_2,j_2\} =
\varnothing$

\item $A_2$ consists of those terms with $|\{i_1,j_1\} \cap \{i_2,j_2\}| = 1$

\item $A_3$ consists of those terms with $|\{i_1,j_1\} \cap \{i_2,j_2\}| = 2$%
.
\end{itemize}

By symmetry, we have 
\begin{align*}
A_{1}& =\tfrac{1}{4}\langle H_{12}H_{34}\psi ,\psi \rangle \\
A_{2}& =\tfrac{1}{2}N^{-1}\langle H_{12}H_{23}\psi ,\psi \rangle \\
A_{3}& =\tfrac{1}{2}N^{-2}\langle H_{12}H_{12}\psi ,\psi \rangle
\end{align*}%
We discard $A_{3}$ since $A_{3}\geq 0$. By the analysis used in the $k=1$
case, 
\begin{equation*}
A_{1}\geq \tfrac{1}{4}\Vert S_{1}S_{3}\psi \Vert _{L^{2}}^{2}
\end{equation*}%
The main piece of work in the $k=2$ case is to estimate $A_{2}$.
Substituting $H_{12}=H_{K12}+H_{I12}$ and $H_{23}=H_{K23}+H_{I23}$, we
obtain the expansion 
\begin{equation*}
A_{2}=B_{0}+B_{1}+B_{2}
\end{equation*}%
where 
\begin{align*}
B_{0}& =\tfrac{1}{2}N^{-1}\langle H_{K12}H_{K23}\psi ,\psi \rangle \\
B_{1}& =\tfrac{1}{2}N^{-1}\langle H_{K12}H_{I23}\psi ,\psi \rangle +\tfrac{1%
}{2}N^{-1}\langle H_{I12}H_{K23}\psi ,\psi \rangle \\
B_{2}& =\tfrac{1}{2}N^{-1}\langle H_{I12}H_{I23}\psi ,\psi \rangle
\end{align*}%
Let $\sigma =\alpha -1\geq 0$. First note that 
\begin{equation*}
B_{0}=\tfrac{1}{2}N^{-1}\langle (S_{1}^{2}+S_{2}^{2}+2\sigma
)(S_{2}^{2}+S_{3}^{2}+2\sigma )\psi ,\psi \rangle
\end{equation*}%
Since $S_{1}^{2}$, $S_{2}^{2}$, $S_{3}^{2}$ all commute, 
\begin{equation*}
B_{0}\geq \tfrac{1}{2}N^{-1}\langle S_{2}^{4}\psi ,\psi \rangle
\end{equation*}%
which is a component of the claimed lower bound.

Next, we consider $B_{1}$. By symmetry 
\begin{equation*}
B_{1}=N^{-1}\Re \langle H_{K12}H_{I23}\psi ,\psi \rangle 
\end{equation*}%
Since every term in $B_{1}$ is estimated, we do not drop the imaginary part.
Decompose $I=P_{0\omega }^{2}+P_{\geq 1\omega }^{2}$ in the right $\psi $
factor 
\begin{equation*}
B_{1}=B_{10}+B_{11}+B_{12}
\end{equation*}%
where%
\begin{equation*}
B_{10}=(N\omega )^{-1}\langle \left[ \left( 2\alpha -1\right) +S_{1}^{2}%
\right] V_{N\omega 23}\psi ,\psi \rangle 
\end{equation*}
\begin{equation*}
B_{11}=(N\omega )^{-1}\langle (-\Delta _{r_{2}}+\omega
^{2}|x_{2}|^{2}-2\omega )V_{N\omega 23}\psi ,P_{0\omega }^{2}\psi \rangle 
\end{equation*}%
\begin{equation*}
B_{12}=(N\omega )^{-1}\langle (-\Delta _{r_{2}}+\omega
^{2}|x_{2}|^{2}-2\omega )V_{N\omega 23}\psi ,P_{\geq 1\omega }^{2}\psi
\rangle 
\end{equation*}%
The term $B_{10}$ is the simplest. In fact, by estimate (\ref{E:3Dto1D11})
at the $r_{2}$ coordinate, we have%
\begin{eqnarray*}
\left\vert B_{10}\right\vert  &=&\left\vert (N\omega )^{-1}\langle \left[
\left( 2\alpha -1\right) +S_{1}^{2}\right] V_{N\omega 23}\psi ,\psi \rangle
\right\vert  \\
&\lesssim &N^{-1}\left( \left\Vert S_{2}\psi \right\Vert
_{L^{2}}^{2}+\left\Vert S_{1}S_{2}\psi \right\Vert _{L^{2}}^{2}\right) \text{%
.}
\end{eqnarray*}
For $B_{12}$, we consider the four terms separately 
\begin{equation*}
B_{12}=B_{121}+B_{122}+B_{123}+B_{124}
\end{equation*}%
where 
\begin{equation*}
B_{121}=(N\omega )^{\beta -1}\langle (\nabla V)_{N\omega 23}\psi ,\nabla
_{r_{2}}P_{\geq 1\omega }^{2}\psi \rangle 
\end{equation*}%
\begin{equation*}
B_{122}=(N\omega )^{-1}\langle V_{N\omega 23}\nabla _{r_{2}}\psi ,\nabla
_{r_{2}}P_{\geq 1\omega }^{2}\psi \rangle 
\end{equation*}%
\begin{equation*}
B_{123}=(N\omega )^{-1}\langle V_{N\omega 23}\omega |x_{2}|\psi ,\omega
|x_{2}|P_{\geq 1\omega }^{2}\psi \rangle 
\end{equation*}%
\begin{equation*}
B_{124}=-2(N\omega )^{-1}\langle V_{N\omega 23}\omega ^{1/2}\psi ,\omega
^{1/2}P_{\geq 1\omega }^{2}\psi \rangle 
\end{equation*}%
By \eqref{E:3Dto1D11} applied with $r_{1}$ replaced by $r_{3}$, we obtain 
\begin{equation*}
|B_{121}|\lesssim (N\omega )^{\beta -1}\omega \Vert S_{3}\psi \Vert
_{L^{2}}\Vert \nabla _{r_{2}}P_{\geq 1\omega }^{2}S_{3}\psi \Vert _{L^{2}}
\end{equation*}%
By \eqref{E:3Dto1D7}, 
\begin{equation*}
|B_{121}|\lesssim (N\omega )^{\beta -1}\omega \Vert S_{3}\psi \Vert
_{L^{2}}\Vert S_{2}S_{3}\psi \Vert _{L^{2}}
\end{equation*}%
which yields the requirement $\omega \leq N^{(1-\beta )/\beta }$. By %
\eqref{E:3Dto1D11} applied with $r_{1}$ replaced by $r_{3}$, we obtain 
\begin{equation*}
|B_{122}|\lesssim (N\omega )^{-1}\omega \Vert \nabla _{r_{2}}S_{3}\psi \Vert
_{L^{2}}\Vert \nabla _{r_{2}}P_{\geq 1\omega }S_{3}\psi \Vert _{L^{2}}
\end{equation*}%
Utilizing \eqref{E:3Dto1D10} for the $\Vert \nabla _{r_{2}}S_{3}\psi \Vert
_{L^{2}}$ term and \eqref{E:3Dto1D7} for the $\Vert \nabla _{r_{2}}P_{\geq
1\omega }S_{3}\psi \Vert _{L^{2}}$ term, 
\begin{equation*}
|B_{122}|\lesssim (N\omega )^{-1}\omega ^{3/2}\Vert S_{2}S_{3}\Vert
_{L^{2}}^{2}
\end{equation*}%
This requires $\omega \leq N^{2}$. The terms $B_{123}$ and $B_{124}$ are
estimated in the same way as $B_{122}$, yielding the requirement $\omega
\leq N^{2}$. This completes the treatment of $B_{12}$.

For $B_{11}$, we move the operator $(-\Delta _{r_{2}}+\omega
^{2}|x_{2}|^{2}-2\omega )$ over to the right, and use the fact that $%
(-\Delta _{r_{2}}+\omega ^{2}|x_{2}|^{2}-2\omega )P_{0\omega }^{2}\psi
=-\partial _{z_{2}}^{2}P_{0\omega }^{2}\psi $ to obtain 
\begin{equation*}
B_{11}=B_{111}+B_{112}
\end{equation*}%
where 
\begin{equation*}
B_{111}=(N\omega )^{\beta -1}\langle (\partial _{z}V)_{N\omega 23}\psi
,\partial _{z_{2}}P_{0\omega }^{2}\psi \rangle
\end{equation*}%
\begin{equation*}
B_{112}=(N\omega )^{-1}\langle V_{N\omega 23}\partial _{z_{2}}\psi ,\partial
_{z_{2}}P_{0\omega }^{2}\psi \rangle
\end{equation*}%
By \eqref{E:3Dto1D11} applied with $r_{1}$ replaced by $r_{3}$, we obtain 
\begin{equation*}
|B_{111}|\lesssim (N\omega )^{\beta -1}\omega \Vert S_{3}\psi \Vert
_{L^{2}}\Vert \partial _{z_{2}}P_{0\omega }^{2}S_{3}\psi \Vert _{L^{2}}
\end{equation*}%
Using \eqref{E:3Dto1D9} for the $\Vert \partial _{z_{2}}P_{0\omega
}^{2}S_{3}\psi \Vert _{L^{2}}$ term (which saves us from the $\omega ^{1/2}$
loss), 
\begin{equation*}
|B_{111}|\lesssim (N\omega )^{\beta -1}\omega \Vert S_{3}\psi \Vert
_{L^{2}}\Vert S_{2}S_{3}\psi \Vert _{L^{2}}
\end{equation*}%
which again requires that $\omega \leq N^{(1-\beta )/\beta }$. By %
\eqref{E:3Dto1D11} applied with $r_{1}$ replaced by $r_{3}$, we obtain 
\begin{equation*}
|B_{112}|\lesssim (N\omega )^{-1}\omega \Vert \partial _{z_{2}}S_{3}\psi
\Vert _{L^{2}}\Vert \partial _{z_{2}}P_{0\omega }^{2}S_{3}\psi \Vert _{L^{2}}
\end{equation*}%
Using \eqref{E:3Dto1D9} 
\begin{equation*}
|B_{112}|\lesssim (N\omega )^{-1}\omega \Vert S_{2}S_{3}\psi \Vert
_{L^{2}}^{2}
\end{equation*}%
which has no requirement on $\omega $. This completes the treatment of $%
B_{11}$, and hence also $B_{1}$. Now let us proceed to consider $B_{2}$. 
\begin{equation*}
B_{2}=N^{-1}\omega ^{-2}\langle V_{N\omega 12}V_{N\omega 23}\psi ,\psi
\rangle
\end{equation*}%
\begin{equation*}
|B_{2}|\leq N^{-1}\omega ^{-2}\int |V_{N\omega 23}|\left(
\int_{r_{1}}|V_{N\omega 12}|\,|\psi (r_{1},\ldots
,r_{N})|^{2}\,dr_{1}\right) \,dr_{2}\cdots dr_{N}
\end{equation*}%
In the parenthesis, apply estimate (\ref{E:3Dto1D11}) in the $r_{1}$
coordinate to obtain 
\begin{equation*}
|B_{2}|\lesssim N^{-1}\omega ^{-2}\omega \int_{r_{2},\ldots
,r_{N}}|V_{N\omega 23}|\Vert S_{1}\psi \Vert
_{L_{r_{1}}^{2}}^{2}\,dr_{2}\cdots dr_{N}
\end{equation*}%
By Fubini, 
\begin{equation*}
=N^{-1}\omega ^{-2}\omega \int_{r_{1}}\left( \int_{r_{2},\ldots
,r_{N}}|V_{N\omega 23}||S_{1}\psi (r_{1},\cdots ,r_{N})|^{2}\,dr_{2}\cdots
dr_{N}\right) dr_{1}
\end{equation*}%
In the parenthesis, apply estimate (\ref{E:3Dto1D11}) in the $r_{2}$
coordinate to obtain 
\begin{equation*}
|B_{2}|\lesssim N^{-1}\omega ^{-2}\omega ^{2}\Vert S_{1}S_{2}\psi \Vert
_{L^{2}}^{2}
\end{equation*}%
Hence $B_{2}$ is bounded without additional restriction on $\omega $.
Therefore we end the proof for the $k=2$ case.

\subsubsection{The $k$ Case Implies The $k+2$ Case}

We assume that (\ref{E:energy-nonscaled}) holds for $k$. Applying it with $%
\psi $ replaced by $(\alpha +N^{-1}H_{N,\omega }-2\omega )\psi $, 
\begin{equation*}
\frac{1}{2^{k}}\Vert S^{(k)}(\alpha +N^{-1}H_{N,\omega }-2\omega )\psi \Vert
_{L^{2}}\leq \langle (\alpha +N^{-1}H_{N,\omega }-2\omega )^{k+2}\psi ,\psi
\rangle 
\end{equation*}%
Hence, to prove (\ref{E:energy-nonscaled}) in the case $k+2$, it suffices to
prove 
\begin{equation}
\frac{1}{4}\Big(\Vert S^{(k+2)}\psi \Vert _{L^{2}}^{2}+N^{-1}\Vert
S_{1}S^{(k+1)}\psi \Vert _{L^{2}}^{2}\Big)\leq \Vert S^{(k)}(\alpha
+N^{-1}H_{N,\omega }-2\omega )\psi \Vert _{L^{2}}^{2}  \label{E:3Dto1D23}
\end{equation}%
To prove \eqref{E:3Dto1D23}, we substitute \eqref{E:3Dto1D6} into 
\begin{equation*}
\langle S^{(k)}(\alpha +N^{-1}H_{N,\omega }-2\omega )\psi ,S^{(k)}(\alpha
+N^{-1}H_{N,\omega }-2\omega )\psi \rangle 
\end{equation*}%
which gives 
\begin{equation*}
N^{-4}\sum_{\substack{ 1\leq i_{1}<j_{1}\leq N \\ 1\leq i_{2}<j_{2}\leq N}}%
\langle S^{(k)}H_{i_{1}j_{1}}\psi ,S^{(k)}H_{i_{2}j_{2}}\psi \rangle 
\end{equation*}%
We decompose into three terms 
\begin{equation*}
=E_{1}+E_{2}+E_{3}
\end{equation*}%
according to the location of $i_{1}$ and $i_{2}$ relative to $k$. We place
no restriction on $j_{1}$, $j_{2}$ (other than $i_{1}<j_{1}$, $i_{2}<j_{2}$.)

\begin{itemize}
\item $E_1$ consists of those terms for which $i_1\leq k$ and $i_2 \leq k$.

\item $E_2$ consists of those terms for which both $i_1 > k$ and $i_2 >k$.

\item $E_3$ consists of those terms for which either ($i_1 \leq k$ and $i_2
>k$) or ($i_1>k$ and $i_2<k$).
\end{itemize}

We have $E_{1}\geq 0$, and we discard this term. We extract the key lower
bound from $E_{2}$ exactly as in the $k=2$ case. In fact, inside $E_{2}$, $%
H_{i_{1}j_{1}}$ and $H_{i_{2}j_{2}}$ commute with $S^{(k)}$ because $%
j_{1}>i_{1}>k$ and $j_{2}>i_{2}>k$, hence we indeed face the $k=2$ case
again. This leaves us with $E_{3}$. 
\begin{equation*}
E_{3}=2N^{-4}\sum_{\substack{ 1\leq i_{1}<j_{1}\leq N  \\ 1\leq
i_{2}<j_{2}\leq N  \\ i_{1}\leq k,i_{2}>k}}\Re \langle
S^{(k)}H_{i_{1}j_{1}}\psi ,S^{(k)}H_{i_{2}j_{2}}\psi \rangle
\end{equation*}%
We decompose 
\begin{equation*}
E_{3}=D_{1}+D_{2}+D_{3}
\end{equation*}%
where, in each case we require $i_{1}\leq k$ and $i_{2}>k$, but make the
additional distinctions as follows:

\begin{itemize}
\item $D_1$ consists of those terms where $j_1 \leq k$

\item $D_2$ consists of those terms where $j_1 > k$ and $j_1 \in \{i_2,j_2\}$

\item $D_3$ consists of those terms where $j_1 > k$ and $j_1 \notin
\{i_2,j_2\}$
\end{itemize}

By symmetry, 
\begin{equation*}
D_1= k^2 N^{-2} \langle S_1 \cdots S_k H_{12} \psi, S_1 \cdots S_k
H_{(k+1)(k+2)} \psi \rangle
\end{equation*}
\begin{equation*}
D_2 = k N^{-2} \langle S_1 \cdots S_k H_{1(k+1)} \psi, S_1 \cdots S_k
H_{(k+1)(k+2)} \psi \rangle
\end{equation*}
\begin{equation*}
D_3 = N^{-1} \langle S_1 \cdots S_k H_{1(k+1)} \psi, S_1 \cdots S_k
H_{(k+2)(k+3)} \psi \rangle
\end{equation*}

\paragraph{\textsc{Estimates for Term D}$_{1}$}

\begin{equation*}
D_{1}=D_{11}+D_{12}
\end{equation*}%
where 
\begin{equation*}
D_{11}=N^{-2}\langle H_{(k+1)(k+2)}[S_{1}S_{2},H_{12}]S_{3}\cdots S_{k}\psi
,S_{1}\cdots S_{k}\psi \rangle 
\end{equation*}%
\begin{equation*}
D_{12}=N^{-2}\langle H_{(k+1)(k+2)}H_{12}S_{1}\cdots S_{k}\psi ,S_{1}\cdots
S_{k}\psi \rangle 
\end{equation*}%
By Lemmas \ref{L:higher-k-3} and \ref{L:op-stuff}, $D_{12}$ is positive
because $H_{(k+1)(k+2)}$ and $H_{12}$ commutes. Therefore we discard $D_{12}$%
. For $D_{11}$, we take $[V_{N\omega 12},S_{1}S_{2}]\sim (N\omega )^{2\beta
}(\Delta V)_{N\omega 12}$. This gives

\begin{equation*}
|D_{11}|\lesssim N^{2\beta -2}\omega ^{2\beta -1}\langle
H_{(k+1)(k+2)}(\Delta V)_{N\omega 12}S_{3}\cdots S_{k}\psi ,S_{1}\cdots
S_{k}\psi \rangle
\end{equation*}%
By Lemma \ref{L:higher-k-3} in the $r_{k+1}$ coordinate to handle $%
H_{(k+1)(k+2)}$%
\begin{equation*}
|D_{11}|\lesssim N^{2\beta -2}\omega ^{2\beta -1}\left\Vert \left\vert
(\Delta V)_{N\omega 12}\right\vert ^{\frac{1}{2}}S_{3}\cdots S_{k+1}\psi
\right\Vert _{L^{2}}\left\Vert \left\vert (\Delta V)_{N\omega 12}\right\vert
^{\frac{1}{2}}S_{1}\cdots S_{k+1}\psi \right\Vert _{L^{2}}
\end{equation*}%
Use (\ref{E:3Dto1D11}) in the first factor%
\begin{equation*}
|D_{11}|\lesssim N^{2\beta -2}\omega ^{2\beta -\frac{1}{2}}\left\Vert
S_{1}S_{3}\cdots S_{k+1}\psi \right\Vert _{L^{2}}\left\Vert \left\vert
(\Delta V)_{N\omega 12}\right\vert ^{\frac{1}{2}}S_{1}\cdots S_{k+1}\psi
\right\Vert _{L^{2}}
\end{equation*}%
Decompose $\psi $ in the second factor into $P_{0\omega }^{1}\psi
+P_{\geqslant 1\omega }^{1}\psi $ 
\begin{eqnarray*}
&\lesssim &N^{2\beta -2}\omega ^{2\beta -\frac{1}{2}}\left\Vert
S_{1}S_{3}\cdots S_{k+1}\psi \right\Vert _{L^{2}} \\
&&\times \left( \left\Vert \left\vert (\Delta V)_{N\omega 12}\right\vert ^{%
\frac{1}{2}}S_{1}\cdots S_{k+1}P_{0\omega }^{1}\psi \right\Vert
_{L^{2}}+\left\Vert \left\vert (\Delta V)_{N\omega 12}\right\vert ^{\frac{1}{%
2}}S_{1}\cdots S_{k+1}P_{\geqslant 1\omega }^{1}\psi \right\Vert
_{L^{2}}\right)
\end{eqnarray*}%
Apply Lemma \ref{L:higher-k-2}%
\begin{eqnarray*}
&\lesssim &N^{2\beta -2}\omega ^{2\beta -\frac{1}{2}}\left\Vert
S_{1}S_{3}\cdots S_{k+1}\psi \right\Vert _{L^{2}}\omega ^{\frac{1}{2}}N^{%
\frac{1}{4}}\left\Vert S_{1}\cdots S_{k+1}\psi \right\Vert _{L^{2}}^{\frac{1%
}{2}}\left( N^{-\frac{1}{4}}\left\Vert S_{1}^{2}\cdots S_{k+1}\psi
\right\Vert _{L^{2}}^{\frac{1}{2}}\right) \\
&&+N^{2\beta -2}\omega ^{2\beta -\frac{1}{2}}\left\Vert S_{1}S_{3}\cdots
S_{k+1}\psi \right\Vert _{L^{2}}N^{\frac{\beta }{2}+\frac{1}{2}}\omega ^{%
\frac{\beta }{2}}\left( N^{-\frac{1}{2}}\left\Vert S_{1}^{2}\cdots
S_{k+1}\psi \right\Vert _{L^{2}}\right)
\end{eqnarray*}%
The coefficients simplify to $N^{2\beta -\frac{7}{4}}\omega ^{2\beta }$ and $%
N^{\frac{5}{2}\beta -\frac{3}{2}}\omega ^{\frac{5}{2}\beta -\frac{1}{2}}$.
This gives the constraints%
\begin{equation*}
\omega \leq N^{\frac{\frac{7}{4}-2\beta }{2\beta }}\text{ and }\omega \leq
N^{\frac{\frac{3}{5}-\beta }{\beta -\frac{1}{5}}}\text{.}
\end{equation*}%
The second one is the worst one. When combined with the lower bound $N^{%
\frac{\beta }{1-\beta }}\leq \omega $, it restricts us to $\beta \leq \frac{3%
}{7}$. Moreover, at $\beta =\frac{2}{5}$, the relation $\omega =N$ is within
the allowable range.

\paragraph{\textsc{Estimates for Term D}$_{2}$}

We write 
\begin{equation*}
D_{2}=D_{21}+D_{22}
\end{equation*}%
where 
\begin{equation*}
D_{21}=N^{-2}\langle H_{(k+1)(k+2)}[S_{1},H_{1(k+1)}]S_{2}\cdots S_{k}\psi
,S_{1}\cdots S_{k}\psi \rangle 
\end{equation*}%
\begin{equation*}
D_{22}=N^{-2}\langle H_{(k+1)(k+2)}H_{1(k+1)}S_{1}\cdots S_{k}\psi
,S_{1}\cdots S_{k}\psi \rangle 
\end{equation*}%
Let us begin with $D_{21}$. Use 
\begin{equation*}
\lbrack S_{1},H_{1(k+1)}]\sim (N\omega )^{\beta }\omega ^{-1}(\nabla
V)_{N\omega 1(k+1)}
\end{equation*}%
and 
\begin{equation*}
H_{(k+1)(k+2)}=2\sigma +S_{k+1}^{2}+S_{k+2}^{2}+\omega ^{-1}V_{N\omega
(k+1)(k+2)}
\end{equation*}%
to get 
\begin{equation*}
D_{21}=D_{210}+D_{211}+D_{212}+D_{213}
\end{equation*}%
where%
\begin{equation*}
D_{210}=2\sigma N^{-1}(N\omega )^{\beta -1}\langle (\nabla V)_{N\omega
1(k+1)}S_{2}\cdots S_{k}\psi ,S_{1}\cdots S_{k}\psi \rangle 
\end{equation*}
\begin{equation*}
D_{211}=N^{-1}(N\omega )^{\beta -1}\langle S_{k+1}^{2}(\nabla V)_{N\omega
1(k+1)}S_{2}\cdots S_{k}\psi ,S_{1}\cdots S_{k}\psi \rangle 
\end{equation*}%
\begin{equation*}
D_{212}=N^{-1}(N\omega )^{\beta -1}\langle S_{k+2}^{2}(\nabla V)_{N\omega
1(k+1)}S_{2}\cdots S_{k}\psi ,S_{1}\cdots S_{k}\psi \rangle 
\end{equation*}%
\begin{equation*}
D_{213}=N^{-2}(N\omega )^{\beta }\omega ^{-2}\langle V_{N\omega
(k+1)(k+2)}(\nabla V)_{N\omega 1(k+1)}S_{2}\cdots S_{k}\psi ,S_{1}\cdots
S_{k}\psi \rangle 
\end{equation*}%
For $D_{211}$,%
\begin{eqnarray*}
D_{211} &=&N^{-1}\left( N\omega \right) ^{\beta -1}\left\langle \left[
S_{k+1},\left( \nabla V\right) _{Nw1(k+1)}\right] S_{2}...S_{k}\psi
,S_{1}\cdots S_{k}\psi \right\rangle  \\
&&+N^{-1}\left( N\omega \right) ^{\beta -1}\left\langle \left( \nabla
V\right) _{Nw1(k+1)}S_{2}...S_{k}S_{k+1}\psi ,S_{1}\cdots S_{k}\psi
\right\rangle 
\end{eqnarray*}%
The first piece is estimated the same way as $D_{11}$. For the second term,
use Lemma \ref{L:higher-k-2} in the $r_{1}$ coordinate 
\begin{eqnarray*}
|\cdot | &\lesssim &N^{-1}\left( N\omega \right) ^{\beta -1}\omega N^{\frac{1%
}{4}}\Vert S_{1}\cdots S_{k+1}\psi \Vert _{L^{2}}\left\Vert S_{1}\cdots
S_{k}\psi \right\Vert _{L^{2}}^{\frac{1}{2}}\left( N^{-\frac{1}{4}}\Vert
S_{1}S_{1}\cdots S_{k}\psi \Vert _{L^{2}}\right)  \\
&&+N^{-1}\left( N\omega \right) ^{\beta -1}\left( N\omega \right) ^{\frac{%
\beta }{2}+\frac{1}{2}}\left\Vert S_{1}\cdots S_{k+1}\psi \right\Vert
_{L^{2}}\left( N^{-\frac{1}{2}}\left\Vert S_{1}S_{1}\cdots S_{k}\psi
\right\Vert _{L^{2}}\right) 
\end{eqnarray*}%
which gives the conditions $\omega \leqslant N^{\frac{^{\frac{7}{4}-\beta }}{%
\beta }}$ and $\omega \leqslant N^{\frac{^{3-3\beta }}{3\beta -1}}$. Since
this results in conditions better than those produced for $D_{11}$, we
neglect them.

For $D_{213}$, we apply estimate (\ref{E:3Dto1D11}) in the $r_{k+2}$
coordinate and again in the $r_{k+1}$ coordinate to obtain 
\begin{equation*}
|D_{213}|\lesssim N^{-2}(N\omega )^{\beta }\omega ^{-2}\omega ^{2}\Vert
S_{2}\cdots S_{k+2}\psi \Vert _{L^{2}}\Vert S_{1}\cdots S_{k+2}\psi \Vert
_{L^{2}}
\end{equation*}%
This gives the requirement $\omega \leqslant N^{\frac{2-\beta }{\beta }}$, which is clearly weaker than $\omega \leq N^{\frac{1-\beta}{\beta}}$, so we drop it.
The terms $D_{210}$ and $D_{212}$ are estimated in the same way. In fact,
utilizing estimate (\ref{E:3Dto1D11}) in the $r_{k+1}$ coordinate yields%
\begin{equation*}
|D_{210}|\lesssim N^{-1}(N\omega )^{\beta -1}\omega \Vert S_{2}\cdots
S_{k}\psi \Vert _{L^{2}}\Vert S_{1}\cdots S_{k}\psi \Vert _{L^{2}}
\end{equation*}%
and%
\begin{equation*}
|D_{212}|\lesssim N^{-1}(N\omega )^{\beta -1}\omega \Vert S_{2}\cdots
S_{k+2}\psi \Vert _{L^{2}}\Vert S_{1}\cdots S_{k+2}\psi \Vert _{L^{2}}.
\end{equation*}%
They give the same weaker condition $\omega \leqslant N^{\frac{2-\beta }{\beta }}$.

We now turn to $D_{22}$. Since $H_{(k+1)(k+2)}$ and $H_{1(k+1)}$ do not
commute, we can not directly quote Lemma \ref{L:higher-k-3} and conclude it
is positive. We estimate it. By the definition of $H_{ij}$, we only need to
look at the following terms%
\begin{equation*}
D_{220}=N^{-2}\omega ^{-1}\langle \sigma V_{N\omega 1(k+1)}S_{1}\cdots
S_{k}\psi ,S_{1}\cdots S_{k}\psi \rangle 
\end{equation*}%
\begin{equation*}
D_{221}=N^{-2}\omega ^{-1}\langle S_{k+1}^{2}V_{N\omega 1(k+1)}S_{1}\cdots
S_{k}\psi ,S_{1}\cdots S_{k}\psi \rangle 
\end{equation*}%
\begin{equation*}
D_{222}=N^{-2}\omega ^{-1}\langle S_{k+2}^{2}V_{N\omega 1(k+1)}S_{1}\cdots
S_{k}\psi ,S_{1}\cdots S_{k}\psi \rangle 
\end{equation*}%
\begin{equation*}
D_{223}=N^{-2}\omega ^{-2}\langle V_{N\omega (k+1)(k+2)}V_{N\omega
1(k+1)}S_{1}\cdots S_{k}\psi ,S_{1}\cdots S_{k}\psi \rangle 
\end{equation*}%
\begin{equation*}
D_{224}=N^{-2}\omega ^{-1}\langle \sigma V_{N\omega (k+1)(k+2)}S_{1}\cdots
S_{k}\psi ,S_{1}\cdots S_{k}\psi \rangle 
\end{equation*}%
\begin{equation*}
D_{225}=N^{-2}\omega ^{-1}\langle V_{N\omega (k+1)(k+2)}S_{1}^{2}S_{1}\cdots
S_{k}\psi ,S_{1}\cdots S_{k}\psi \rangle 
\end{equation*}%
\begin{equation*}
D_{226}=N^{-2}\omega ^{-1}\langle V_{N\omega
(k+1)(k+2)}S_{k+1}^{2}S_{1}\cdots S_{k}\psi ,S_{1}\cdots S_{k}\psi \rangle 
\end{equation*}%
because all the other terms inside the expansion of $D_{22}$ are positive.
It is easy to tell the following: $D_{220}$ and $D_{224}$ can be estimated
in the same way as $D_{210}$, $D_{221}$ and $D_{226}$ can be estimated in
the same way as $D_{211}$, $D_{222}$ and $D_{225}$ can be estimated in the
same way as $D_{212}$, and $D_{223}$ can be estimated in the same way as $%
D_{213}$. Moreover, all the $D_{22}$ terms are better than the corresponding 
$D_{21}$ terms since they do not have a $(N\omega )^{\beta }$ in front of
them. Hence, we get no new restrictions from $D_{22}$ and we conclude the
estimate for $D_{22}$.

\paragraph{\textsc{Estimates for Term D}$_{3}$}

Commuting terms as usual: 
\begin{equation*}
D_{3}=D_{31}+D_{32}
\end{equation*}%
where 
\begin{equation*}
D_{31}=N^{-1}\langle H_{(k+2)(k+3)}[S_{1},H_{1(k+1)}]S_{2}\cdots S_{k}\psi
,S_{1}\cdots S_{k}\psi \rangle 
\end{equation*}%
\begin{equation*}
D_{32}=N^{-1}\langle H_{(k+2)(k+3)}H_{1(k+1)}S_{1}\cdots S_{k}\psi
,S_{1}\cdots S_{k}\psi \rangle 
\end{equation*}%
Since $H_{(k+2)(k+3)}$ and $H_{1(k+1)}$ commute, $D_{32}$ is positive due to
Lemmas \ref{L:higher-k-3} and \ref{L:op-stuff}. Thus we discard $D_{32}$.
For $D_{31}$, we use that 
\begin{equation*}
\lbrack S_{1},H_{1(k+1)}]\sim (N\omega )^{\beta }\omega ^{-1}(\nabla
V)_{N\omega 1(k+1)}
\end{equation*}%
together with estimate (\ref{E:3Dto1D11}) in the $r_{k+1}$ coordinate (to
handle $[S_{1},H_{1(k+1)}]$) and Lemma \ref{L:higher-k-3} in the $r_{k+2}$
coordinate (to handle $H_{(k+2)(k+3)}$) 
\begin{equation*}
|D_{31}|\lesssim N^{-1}(N\omega )^{\beta }\Vert S_{2}\cdots S_{k+2}\psi
\Vert _{L^{2}}\Vert S_{1}\cdots S_{k+2}\psi \Vert _{L^{2}}
\end{equation*}%
This term again yields to the restriction 
\begin{equation*}
\omega \leq N^{\frac{1-\beta }{\beta }}
\end{equation*}%
So far, we have proved that all the terms in $E_{3}$ can be absorbed into
the key lower bound exacted from $E_{2}$ for all $N$ large enough as long as 
$C_{1}N^{v_{1}(\beta )}\leqslant \omega \leqslant C_{2}N^{v_{E}(\beta )}$.
Thence we have finished the two step induction argument and established
Theorem \ref{Theorem:Energy Estimate}.

\section{Compactness of the BBGKY sequence \label{Section:Compactness}}

\begin{theorem}
\label{Theorem:Compactness of the scaled marginal density}Assume $%
C_{1}N^{v_{1}(\beta )}\leqslant \omega \leqslant C_{2}N^{v_{2}(\beta )}$,
then the sequence 
\begin{equation*}
\left\{ \Gamma _{N,\omega }(t)=\left\{ \tilde{\gamma}_{N,\omega
}^{(k)}\right\} _{k=1}^{N}\right\} \subset \bigoplus_{k\geqslant 1}C\left( %
\left[ 0,T\right] ,\mathcal{L}_{k}^{1}\right)
\end{equation*}%
which satisfies the focusing "$\infty -\infty $" BBGKY hierarchy 
\eqref{hierarchy:BBGKY
hierarchy for scaled marginal densities}, is compact with respect to the
product topology $\tau _{prod}$. For any limit point $\Gamma (t)=\left\{ 
\tilde{\gamma}^{(k)}\right\} _{k=1}^{N},$ $\tilde{\gamma}^{(k)}$ is a
symmetric nonnegative trace class operator with trace bounded by $1$.
\end{theorem}

\begin{proof}
By the standard diagonalization argument, it suffices to show the
compactness of $\tilde{\gamma}_{N,\omega }^{(k)}$ for fixed $k$ with respect
to the metric $\hat{d}_{k}$. By the Arzel\`{a}-Ascoli theorem, this is
equivalent to the equicontinuity of $\tilde{\gamma}_{N,\omega }^{(k)}$. By 
\cite[Lemma 6.2]{E-S-Y3}, it suffice to prove that for every test function $%
J^{(k)}$ from a dense subset of $\mathcal{K}(L^{2}(\mathbb{R}^{3k}))$ and
for every $\varepsilon >0$, there exists $\delta (J^{(k)},\varepsilon )$
such that for all $t_{1},t_{2}\in \left[ 0,T\right] $ with $\left\vert
t_{1}-t_{2}\right\vert \leqslant \delta $, we write 
\begin{equation}
\sup_{N,\omega }\left\vert \func{Tr}J^{(k)}\tilde{\gamma}_{N,\omega
}^{(k)}(t_{1})-\func{Tr}J^{(k)}\tilde{\gamma}_{N,\omega
}^{(k)}(t_{2})\right\vert \leqslant \varepsilon \,.  \label{E:cpct5}
\end{equation}%
Here, we assume that our compact operators $J^{(k)}$ have been cut off in
frequency as in Lemma \ref{L:compact-operator-truncation}. Assume $%
t_{1}\leqslant t_{2}$. Inserting the decomposition \eqref{E:cpct1} on the
left and right side of $\gamma _{N,\omega }^{(k)}$, we obtain 
\begin{equation*}
\tilde{\gamma}_{N,\omega }^{(k)}=\sum_{\mathbf{\alpha },\mathbf{\beta }}%
\mathcal{P}_{\mathbf{\alpha }}\tilde{\gamma}_{N,\omega }^{(k)}\mathcal{P}_{%
\mathbf{\beta }}
\end{equation*}%
where the sum is taken over all $k$-tuples $\mathbf{\alpha }$ and $\mathbf{%
\beta }$ of the type described in (\ref{E:cpct1}).

To establish \eqref{E:cpct5} it suffices to prove that, for each $\mathbf{%
\alpha }$ and $\mathbf{\beta }$, we have 
\begin{equation}
\sup_{N,\omega }\left\vert \func{Tr}J^{(k)}\mathcal{P}_{\mathbf{\alpha }}%
\tilde{\gamma}_{N,\omega }^{(k)}\mathcal{P}_{\mathbf{\beta }}(t_{1})-\func{Tr%
}J^{(k)}\mathcal{P}_{\mathbf{\alpha }}\tilde{\gamma}_{N,\omega }^{(k)}%
\mathcal{P}_{\mathbf{\beta }}(t_{2})\right\vert \leqslant \varepsilon \,.
\label{E:cpct4}
\end{equation}%
To this end, we establish the estimate 
\begin{eqnarray}
&&\left\vert \func{Tr}J^{(k)}\mathcal{P}_{\mathbf{\alpha }}\tilde{\gamma}%
_{N,\omega }^{(k)}\mathcal{P}_{\mathbf{\beta }}(t_{1})-\func{Tr}J^{(k)}%
\mathcal{P}_{\mathbf{\alpha }}\tilde{\gamma}_{N,\omega }^{(k)}\mathcal{P}_{%
\mathbf{\beta }}(t_{2})\right\vert  \label{E:cpct2} \\
&\lesssim &C\left\vert t_{2}-t_{1}\right\vert \left( \mathbf{1}_{\alpha
=0\&\&\beta =0}+\max (1,\omega ^{1-\frac{\left\vert a\right\vert }{2}-\frac{%
\left\vert \beta \right\vert }{2}})\mathbf{1}_{\alpha \neq 0||\beta \neq
0}\right)  \notag
\end{eqnarray}%
At a glance, \eqref{E:cpct2} seems not quite enough in the $\left\vert 
\mathbf{\alpha }\right\vert =0$ and $\left\vert \mathbf{\beta }\right\vert
=1 $ case (or vice versa) because it grows in $\omega $. However, we can
also prove the (comparatively simpler) bound 
\begin{equation}
\left\vert \func{Tr}J^{(k)}\mathcal{P}_{\mathbf{\alpha }}\tilde{\gamma}%
_{N,\omega }^{(k)}\mathcal{P}_{\mathbf{\beta }}(t_{2})-\func{Tr}J^{(k)}%
\mathcal{P}_{\mathbf{\alpha }}\tilde{\gamma}_{N,\omega }^{(k)}\mathcal{P}_{%
\mathbf{\beta }}(t_{1})\right\vert \lesssim \omega ^{-\frac{1}{2}|\mathbf{%
\alpha }|-\frac{1}{2}|\mathbf{\beta }|}  \label{E:cpct3}
\end{equation}%
which provides a better power of $\omega $ but no gain as $t_{2}\rightarrow
t_{1}$. Interpolating between \eqref{E:cpct2} and \eqref{E:cpct3} in the $%
\left\vert \mathbf{\alpha }\right\vert =0$ and $\left\vert \mathbf{\beta }%
\right\vert =1$ case (or vice versa), we acquire 
\begin{equation*}
\left\vert \func{Tr}J^{(k)}\mathcal{P}_{\mathbf{\alpha }}\tilde{\gamma}%
_{N,\omega }^{(k)}\mathcal{P}_{\mathbf{\beta }}(t_{2})-\func{Tr}J^{(k)}%
\mathcal{P}_{\mathbf{\alpha }}\tilde{\gamma}_{N,\omega }^{(k)}\mathcal{P}_{%
\mathbf{\beta }}(t_{1})\right\vert \lesssim |t_{2}-t_{1}|^{1/2}
\end{equation*}%
which suffices to establish \eqref{E:cpct4}.

Below, we prove \eqref{E:cpct2} and \eqref{E:cpct3}. We first prove %
\eqref{E:cpct2}. The BBGKY hierarchy (\ref{hierarchy:BBGKY hierarchy for
scaled marginal densities}) yields 
\begin{equation}
\partial _{t}\func{Tr}J^{(k)}\mathcal{P}_{\mathbf{\alpha }}\tilde{\gamma}%
_{N,\omega }^{(k)}\mathcal{P}_{\mathbf{\beta }}=\text{I}+\text{II}+\text{III}%
+\text{IV.}  \label{E:cpct50}
\end{equation}%
where%
\begin{equation*}
\text{I}=-i\omega \sum_{j=1}^{k}\func{Tr}J^{(k)}[-\Delta _{x_{j}}+\left\vert
x_{j}\right\vert ^{2},\mathcal{P}_{\mathbf{\alpha }}\tilde{\gamma}_{N,\omega
}^{(k)}\mathcal{P}_{\mathbf{\beta }}]
\end{equation*}%
\begin{equation*}
\text{II}=-i\sum_{j=1}^{k}\func{Tr}J^{(k)}[-\partial _{z_{j}}^{2},\mathcal{P}%
_{\mathbf{\alpha }}\tilde{\gamma}_{N,\omega }^{(k)}\mathcal{P}_{\mathbf{%
\beta }}]
\end{equation*}%
\begin{equation*}
\text{III}=\frac{-i}{N}\sum_{1\leqslant i<j\leqslant k}\func{Tr}J^{(k)}%
\mathcal{P}_{\mathbf{\alpha }}[V_{N,\omega }(r_{i}-r_{j}),\tilde{\gamma}%
_{N,\omega }^{(k)}]\mathcal{P}_{\mathbf{\beta }}
\end{equation*}%
\begin{equation*}
\text{IV}=-i\frac{N-k}{N}\sum_{j=1}^{k}\func{Tr}J^{(k)}\mathcal{P}_{\mathbf{%
\alpha }}[V_{N,\omega }(r_{j}-r_{k+1}),\tilde{\gamma}_{N,\omega }^{(k+1)}]%
\mathcal{P}_{\mathbf{\beta }}
\end{equation*}%
We first consider \text{I}. When $\mathbf{\alpha =\beta }=0$, 
\begin{eqnarray*}
\text{I} &=&-i\omega \sum_{j=1}^{k}\func{Tr}J^{(k)}[-\Delta
_{x_{j}}+\left\vert x_{j}\right\vert ^{2},\mathcal{P}_{\mathbf{0}}\tilde{%
\gamma}_{N,\omega }^{(k)}\mathcal{P}_{\mathbf{0}}] \\
&=&-i\omega \sum_{j=1}^{k}\func{Tr}J^{(k)}[-2-\Delta _{x_{j}}+\left\vert
x_{j}\right\vert ^{2},\mathcal{P}_{\mathbf{0}}\tilde{\gamma}_{N,\omega
}^{(k)}\mathcal{P}_{\mathbf{0}}] \\
&=&0,
\end{eqnarray*}%
since constants commute with everything. When $\mathbf{\alpha }\neq 0$ or $%
\mathbf{\beta }\neq 0$, we apply Lemma \ref{L:trace-of-tp-kernel} and
integrate by parts to obtain 
\begin{align*}
\left\vert \text{I}\right\vert & \leqslant \omega \sum_{j=1}^{k}\left\vert
\langle J^{(k)}H_{j}\mathcal{P}_{\mathbf{\alpha }}\tilde{\psi}_{N,\omega },%
\mathcal{P}_{\mathbf{\beta }}\tilde{\psi}_{N,\omega }\rangle -\langle J^{(k)}%
\mathcal{P}_{\mathbf{\alpha }}\tilde{\psi}_{N,\omega },H_{j}\mathcal{P}_{%
\mathbf{\beta }}\tilde{\psi}_{N,\omega }\rangle \right\vert \\
& \leqslant \omega \sum_{j=1}^{k}\left( \left\vert \langle J^{(k)}H_{j}%
\mathcal{P}_{\mathbf{\alpha }}\tilde{\psi}_{N,\omega },\mathcal{P}_{\mathbf{%
\beta }}\tilde{\psi}_{N,\omega }\rangle \right\vert +\left\vert \langle
H_{j}J^{(k)}\mathcal{P}_{\mathbf{\alpha }}\tilde{\psi}_{N,\omega },\mathcal{P%
}_{\mathbf{\beta }}\tilde{\psi}_{N,\omega }\rangle \right\vert \right)
\end{align*}%
where $H_{j}=-\Delta _{x_{j}}+\left\vert x_{j}\right\vert ^{2}$. Hence 
\begin{equation*}
\left\vert \text{I}\right\vert \lesssim \omega \sum_{j=1}^{k}(\Vert
J^{(k)}H_{j}\Vert _{\func{op}}+\Vert H_{j}J^{(k)}\Vert _{\func{op}})\Vert 
\mathcal{P}_{\mathbf{\alpha }}\tilde{\psi}_{N,\omega }\Vert _{L^{2}(\mathbb{R%
}^{3N})}\Vert \mathcal{P}_{\mathbf{\beta }}\tilde{\psi}_{N,\omega }\Vert
_{L^{2}(\mathbb{R}^{3N})}
\end{equation*}%
By the energy estimate (\ref{E:e-3}), 
\begin{equation}
\left\vert \text{I}\right\vert 
\begin{cases}
=0 & \text{if }\mathbf{\alpha }=0\text{ and }\mathbf{\beta }=0 \\ 
\lesssim C_{k,J^{(k)}}\omega ^{1-\frac{1}{2}|\mathbf{\alpha }|-\frac{1}{2}|%
\mathbf{\beta }|} & \text{otherwise}%
\end{cases}
\label{E:cpct52}
\end{equation}%
Next, consider II. Proceed as in I, we have 
\begin{equation*}
\left\vert \text{II}\right\vert \leqslant \sum_{j=1}^{k}\left( \left\vert
\left\langle J^{(k)}\partial _{z_{j}}^{2}\mathcal{P}_{\mathbf{\alpha }}%
\tilde{\psi}_{N,\omega },\mathcal{P}_{\mathbf{\beta }}\tilde{\psi}_{N,\omega
}\right\rangle \right\vert +\left\vert \left\langle \partial
_{z_{j}}^{2}J^{(k)}\mathcal{P}_{\mathbf{\alpha }}\tilde{\psi}_{N,\omega },%
\mathcal{P}_{\mathbf{\beta }}\tilde{\psi}_{N,\omega }\right\rangle
\right\vert \right)
\end{equation*}%
That is 
\begin{equation}
\left\vert \text{II}\right\vert \leqslant \sum_{j=1}^{k}(\Vert
J^{(k)}\partial _{z_{j}}^{2}\Vert _{\func{op}}+\Vert \partial
_{z_{j}}^{2}J^{(k)}\Vert _{\func{op}})\Vert \mathcal{P}_{\mathbf{\alpha }}%
\tilde{\psi}_{N,\omega }\Vert _{L^{2}(\mathbb{R}^{3N})}\Vert \mathcal{P}_{%
\mathbf{\beta }}\tilde{\psi}_{N,\omega }\Vert _{L^{2}(\mathbb{R}%
^{3N})}\leqslant C_{k,J^{(k)}}.  \label{E:cpct51}
\end{equation}%
Now, consider $\text{III}$. 
\begin{eqnarray*}
\left\vert \text{III}\right\vert &\leqslant &N^{-1}\sum_{1\leqslant
i<j\leqslant k}\left\vert \left\langle J^{(k)}\mathcal{P}_{\mathbf{\alpha }%
}V_{N,\omega }(r_{i}-r_{j})\tilde{\psi}_{N,\omega },\mathcal{P}_{\mathbf{%
\beta }}\tilde{\psi}_{N,\omega }\right\rangle \right\vert + \\
&&N^{-1}\sum_{1\leqslant i<j\leqslant k}\left\vert \left\langle J^{(k)}%
\mathcal{P}_{\mathbf{\alpha }}\tilde{\psi}_{N,\omega },\mathcal{P}_{\mathbf{%
\beta }}V_{N,\omega }(r_{i}-r_{j})\tilde{\psi}_{N,\omega }\right\rangle
\right\vert
\end{eqnarray*}%
That is%
\begin{eqnarray*}
\left\vert \text{III}\right\vert &\leqslant &N^{-1}\sum_{1\leqslant
i<j\leqslant k}\left\vert \left\langle J^{(k)}\mathcal{P}_{\mathbf{\alpha }%
}L_{i}L_{j}W_{ij}L_{i}L_{j}\tilde{\psi}_{N,\omega },\mathcal{P}_{\mathbf{%
\beta }}\tilde{\psi}_{N,\omega }\right\rangle \right\vert \\
&&+N^{-1}\sum_{1\leqslant i<j\leqslant k}\left\vert \left\langle J^{(k)}%
\mathcal{P}_{\mathbf{\alpha }}\tilde{\psi}_{N,\omega },\mathcal{P}_{\mathbf{%
\beta }}L_{i}L_{j}W_{ij}L_{i}L_{j}\tilde{\psi}_{N,\omega }\right\rangle
\right\vert
\end{eqnarray*}%
if we write $L_{i}=(1-\Delta _{r_{i}})^{1/2}$ and 
\begin{equation*}
W_{ij}=L_{i}^{-1}L_{j}^{-1}V_{N,\omega }(r_{i}-r_{j})L_{i}^{-1}L_{j}^{-1}\,.
\end{equation*}%
Hence 
\begin{eqnarray*}
\left\vert \text{III}\right\vert &\leqslant &N^{-1}\sum_{1\leqslant
i<j\leqslant k}\left\Vert J^{(k)}L_{i}L_{j}\right\Vert _{op}\left\Vert
W_{ij}\right\Vert _{op}\left\Vert L_{i}L_{j}\tilde{\psi}_{N,\omega
}\right\Vert _{L^{2}(\mathbb{R}^{3N})}\left\Vert \mathcal{P}_{\mathbf{\beta }%
}\tilde{\psi}_{N,\omega }\right\Vert _{L^{2}(\mathbb{R}^{3N})} \\
&&+N^{-1}\sum_{1\leqslant i<j\leqslant k}\left\Vert
L_{i}L_{j}J^{(k)}\right\Vert _{op}\left\Vert W_{ij}\right\Vert
_{op}\left\Vert L_{i}L_{j}\tilde{\psi}_{N,\omega }\right\Vert _{L^{2}(%
\mathbb{R}^{3N})}\left\Vert \mathcal{P}_{\mathbf{\alpha }}\tilde{\psi}%
_{N,\omega }\right\Vert _{L^{2}(\mathbb{R}^{3N})}
\end{eqnarray*}%
Since $\Vert W_{ij}\Vert _{\func{op}}\lesssim \Vert V_{N,\omega }\Vert
_{L^{1}}=\Vert V\Vert _{L^{1}}$ (independent of $N$, $\omega $) by Lemma \ref%
{Lemma:ESYSoblevLemma}, the energy estimates (Corollary \ref%
{Corollary:Energy Bound for Marginal Densities}) imply that 
\begin{equation}
\left\vert \text{III}\right\vert \lesssim \frac{C_{k,J^{(k)}}}{N}
\label{E:cpct53}
\end{equation}%
Apply the same ideas to $\text{IV}$. 
\begin{eqnarray*}
\left\vert \text{IV}\right\vert &\leqslant &\sum_{j=1}^{k}\left\vert
\left\langle J^{(k)}\mathcal{P}_{\mathbf{\alpha }}L_{j}L_{k+1}W_{j\left(
k+1\right) }L_{j}L_{k+1}\tilde{\psi}_{N,\omega },\mathcal{P}_{\mathbf{\beta }%
}\tilde{\psi}_{N,\omega }\right\rangle \right\vert \\
&&\sum_{j=1}^{k}\left\vert \left\langle J^{(k)}\mathcal{P}_{\mathbf{\alpha }}%
\tilde{\psi}_{N,\omega },\mathcal{P}_{\mathbf{\beta }}L_{j}L_{k+1}W_{j\left(
k+1\right) }L_{j}L_{k+1}\tilde{\psi}_{N,\omega }\right\rangle \right\vert
\end{eqnarray*}%
Then, since $J^{(k)}L_{k+1}=L_{k+1}J^{(k)}$, 
\begin{eqnarray}
&&\left\vert \text{IV}\right\vert  \label{E:cpct54} \\
&\leqslant &\sum_{j=1}^{k}\left( \left\Vert J^{(k)}L_{j}\right\Vert
_{op}+\left\Vert L_{j}J^{(k)}\right\Vert _{op}\right) \left\Vert W_{j\left(
k+1\right) }\right\Vert _{op}\left\Vert L_{j}L_{k+1}\tilde{\psi}_{N,\omega
}\right\Vert _{L^{2}(\mathbb{R}^{3N})}\left\Vert L_{j}\tilde{\psi}_{N,\omega
}\right\Vert _{L^{2}(\mathbb{R}^{3N})}  \notag \\
&\lesssim &C_{k,J^{(k)}}.  \notag
\end{eqnarray}%
Integrating \eqref{E:cpct50} from $t_{1}$ to $t_{2}$ and applying the bounds
obtained in \eqref{E:cpct52}, \eqref{E:cpct51}, \eqref{E:cpct53}, and %
\eqref{E:cpct54}, we obtain \eqref{E:cpct2}.

Finally, we prove \eqref{E:cpct3}. By Lemma \ref{L:trace-of-tp-kernel},%
\begin{eqnarray*}
&&\left\vert \func{Tr}J^{(k)}\mathcal{P}_{\mathbf{\alpha }}\tilde{\gamma}%
_{N,\omega }^{(k)}\mathcal{P}_{\mathbf{\beta }}(t_{2})-\func{Tr}J^{(k)}%
\mathcal{P}_{\mathbf{\alpha }}\tilde{\gamma}_{N,\omega }^{(k)}\mathcal{P}_{%
\mathbf{\beta }}(t_{1})\right\vert \\
&\leqslant &2\sup_{t}\left\vert \left\langle J^{(k)}\mathcal{P}_{\mathbf{%
\alpha }}\tilde{\psi}_{N,\omega }(t),\mathcal{P}_{\mathbf{\beta }}\tilde{\psi%
}_{N,\omega }(t)\right\rangle \right\vert \\
&\lesssim &\Vert J^{(k)}\Vert _{\func{op}}\Vert \mathcal{P}_{\mathbf{\alpha }%
}\tilde{\psi}_{N,\omega }(t)\Vert _{L^{2}(\mathbb{R}^{3N})}\Vert \mathcal{P}%
_{\mathbf{\beta }}\tilde{\psi}_{N,\omega }(t)\Vert _{L^{2}(\mathbb{R}^{3N})}
\end{eqnarray*}%
that is%
\begin{equation*}
\left\vert \func{Tr}J^{(k)}\mathcal{P}_{\mathbf{\alpha }}\tilde{\gamma}%
_{N,\omega }^{(k)}\mathcal{P}_{\mathbf{\beta }}(t_{2})-\func{Tr}J^{(k)}%
\mathcal{P}_{\mathbf{\alpha }}\tilde{\gamma}_{N,\omega }^{(k)}\mathcal{P}_{%
\mathbf{\beta }}(t_{1})\right\vert \lesssim \omega ^{-\frac{1}{2}|\mathbf{%
\alpha }|-\frac{1}{2}|\mathbf{\beta }|}.
\end{equation*}%
once we apply \eqref{E:e-3}.
\end{proof}

With Theorem \ref{Theorem:Compactness of the scaled marginal density}, we
can start talking about the limit points of $\left\{ \Gamma _{N,\omega
}(t)=\{\tilde{\gamma}_{N,\omega }^{(k)}\}_{k=1}^{N}\right\} .$

\begin{corollary}
\label{Corollary:LimitMustBeAProduct}Let $\Gamma (t)=\{\tilde{\gamma}%
^{(k)}\}_{k=1}^{\infty }$ be a limit point of $\left\{ \Gamma _{N,\omega
}(t)=\{\tilde{\gamma}_{N,\omega }^{(k)}\}_{k=1}^{N}\right\} $, with respect
to the product topology $\tau _{prod}$, then $\tilde{\gamma}^{(k)}$
satisfies the a priori bound%
\begin{equation}
\limfunc{Tr}L^{(k)}\tilde{\gamma}^{(k)}L^{(k)}\leqslant C^{k}  \label{E:e-7}
\end{equation}%
and takes the structure%
\begin{equation}
\tilde{\gamma}^{(k)}\left( t,\left( \mathbf{x}_{k},\mathbf{z}_{k}\right)
;\left( \mathbf{x}_{k}^{\prime },\mathbf{z}_{k}^{\prime }\right) \right)
=\left( \dprod\limits_{j=1}^{k}h_{1}\left( x_{j}\right) h_{1}\left(
x_{j}^{\prime }\right) \right) \tilde{\gamma}_{z}^{(k)}(t,\mathbf{z}_{k};%
\mathbf{z}_{k}^{\prime }),  \label{E:e-8}
\end{equation}%
where $\tilde{\gamma}_{z}^{(k)}=\limfunc{Tr}_{x}\tilde{\gamma}^{(k)}$.
\end{corollary}

\begin{proof}
We only need to prove (\ref{E:e-8}) because the a priori bound \eqref{E:e-7}
directly follows from \eqref{E:e-2} in Corollary \ref{Corollary:Energy Bound
for Marginal Densities} and Theorem \ref{Theorem:Compactness of the scaled
marginal density}.

To prove \eqref{E:e-8}, it suffices to prove 
\begin{equation*}
\mathcal{P}_{\mathbf{\alpha }}\tilde{\gamma}^{(k)}\mathcal{P}_{\mathbf{\beta 
}}=0\text{, if }\mathbf{\alpha }\neq 0\text{ or }\mathbf{\beta }\neq 0.
\end{equation*}%
This is equivalent to the statement that 
\begin{equation*}
\func{Tr}J^{(k)}\mathcal{P}_{\mathbf{\alpha }}\tilde{\gamma}^{(k)}\mathcal{P}%
_{\mathbf{\beta }}=0\text{, }\forall J^{(k)}\in \mathcal{K}_{k}.
\end{equation*}%
In fact, 
\begin{equation}
\func{Tr}J^{(k)}\mathcal{P}_{\mathbf{\alpha }}\tilde{\gamma}^{(k)}\mathcal{P}%
_{\mathbf{\beta }}=\lim_{(N,\omega )\rightarrow \infty }\func{Tr}J^{(k)}%
\mathcal{P}_{\mathbf{\alpha }}\tilde{\gamma}_{N,\omega }^{(k)}\mathcal{P}_{%
\mathbf{\beta }}  \label{E:e-4}
\end{equation}%
where 
\begin{equation*}
\func{Tr}J^{(k)}\mathcal{P}_{\mathbf{\alpha }}\tilde{\gamma}_{N,\omega
}^{(k)}\mathcal{P}_{\mathbf{\beta }}=\langle J^{(k)}\mathcal{P}_{\mathbf{%
\alpha }}\tilde{\psi}_{N,\omega },\mathcal{P}_{\mathbf{\beta }}\tilde{\psi}%
_{N,\omega }\rangle .
\end{equation*}%
by Lemma \ref{L:trace-of-tp-kernel}. We remind the reader that, in the
above, $\mathcal{P}_{\mathbf{\alpha }}$ and $\mathcal{P}_{\mathbf{\beta }}$
are acting only on the first $k$ variables of $\tilde{\psi}_{N,\omega }$ as
defined in (\ref{def:multiple projection for omega=1}).

Applying Cauchy-Schwarz, we reach%
\begin{equation*}
\left\vert \func{Tr}J^{(k)}\mathcal{P}_{\mathbf{\alpha }}\tilde{\gamma}%
_{N,\omega }^{(k)}\mathcal{P}_{\mathbf{\beta }}\right\vert \leqslant \Vert
J^{(k)}\Vert _{\func{op}}\Vert \mathcal{P}_{\mathbf{\alpha }}\tilde{\psi}%
_{N,\omega }\Vert _{L^{2}(\mathbb{R}^{3N})}\Vert \mathcal{P}_{\mathbf{\beta }%
}\tilde{\psi}_{N,\omega }\Vert _{L^{2}(\mathbb{R}^{3N})}.
\end{equation*}%
Use \eqref{E:e-3}, we have 
\begin{equation*}
\left\vert \func{Tr}J^{(k)}\mathcal{P}_{\mathbf{\alpha }}\tilde{\gamma}%
_{N,\omega }^{(k)}\mathcal{P}_{\mathbf{\beta }}\right\vert \leqslant
C^{k}\omega ^{-\frac{1}{2}|\mathbf{\alpha }|-\frac{1}{2}|\mathbf{\beta }%
|}\rightarrow 0\text{ as }\omega \rightarrow \infty
\end{equation*}%
as claimed.
\end{proof}

We see from Corollary \ref{Corollary:LimitMustBeAProduct} that, the study of
the limit point of $\left\{ \Gamma _{N,\omega }(t)=\left\{ \tilde{\gamma}%
_{N,\omega }^{(k)}\right\} _{k=1}^{N}\right\} $ is directly related to the
sequence $\left\{ \Gamma _{z,N,\omega }(t)=\left\{ \tilde{\gamma}%
_{z,N,\omega }^{(k)}=\limfunc{Tr}_{x}\tilde{\gamma}_{N,\omega
}^{(k)}\right\} _{k=1}^{N}\right\} \subset \oplus _{k\geqslant 1}C\left( %
\left[ 0,T\right] ,\mathcal{L}_{k}^{1}\left( \mathbb{R}^{k}\right) \right) .$
Thus we analyze $\left\{ \Gamma _{z,N,\omega }(t)\right\} $ in \S \ref%
{Section:Convergence of The Infinite Hierarchy}. At the moment, we prove
that $\left\{ \Gamma _{z,N,\omega }(t)\right\} $ is compact with respect to
the one dimensional version of the product topology $\tau _{prod}$ used in
Theorem \ref{Theorem:Compactness of the scaled marginal density}. This is
straightforward since we do not need to deal with $\infty -\infty $ here.

\begin{theorem}
\label{Theorem:Compactness of the x-marginal density}Assume $%
C_{1}N^{v_{1}(\beta )}\leqslant \omega \leqslant C_{2}N^{v_{2}(\beta )}$,
then the sequence 
\begin{equation*}
\left\{ \Gamma _{z,N,\omega }(t)=\left\{ \tilde{\gamma}_{z,N,\omega }^{(k)}=%
\limfunc{Tr}\nolimits_{x}\tilde{\gamma}_{N,\omega }^{(k)}\right\}
_{k=1}^{N}\right\} \subset \bigoplus_{k\geqslant 1}C\left( \left[ 0,T\right]
,\mathcal{L}_{k}^{1}\left( \mathbb{R}^{k}\right) \right) .
\end{equation*}%
is compact with respect to the one dimensional version of the product
topology $\tau _{prod}$ used in Theorem \ref{Theorem:Compactness of the
scaled marginal density}.
\end{theorem}

\begin{proof}
Similar to Theorem \ref{Theorem:Compactness of the scaled marginal density},
we show that for every test function $J_{z}^{(k)}$ from a dense subset of $%
\mathcal{K}\left( L^{2}\left( \mathbb{R}^{k}\right) \right) $ and for every $%
\varepsilon >0,$ $\exists \delta (J_{z}^{(k)},\varepsilon )$ s.t. $\forall
t_{1},t_{2}\in \left[ 0,T\right] $ with $\left\vert t_{1}-t_{2}\right\vert
\leqslant \delta ,$ we have%
\begin{equation*}
\sup_{N,\omega }\left\vert \limfunc{Tr}J_{z}^{(k)}\left( \tilde{\gamma}%
_{z,N,\omega }^{(k)}\left( t_{1}\right) -\tilde{\gamma}_{z,N,\omega
}^{(k)}\left( t_{2}\right) \right) \right\vert \leqslant \varepsilon .
\end{equation*}%
We again assume that our test function $J_{z}^{(k)}$ has been cut off in
frequency as in Lemma \ref{L:compact-operator-truncation}. Due to the fact
that $\tilde{\gamma}_{z,N,\omega }^{(k)}$ acts on $L^{2}\left( \mathbb{R}%
^{k}\right) $ instead of $L^{2}\left( \mathbb{R}^{3k}\right) $, the test
functions here are similar but different from the ones in the proof of
Theorem \ref{Theorem:Compactness of the scaled marginal density}. This does
not make any differences when we deal with the terms involving $\tilde{\gamma%
}_{N,\omega }^{(k)}$ though. In fact, since $J_{z}^{(k)}$ has no $x$%
-dependence, we have 
\begin{eqnarray*}
\left\Vert L_{j}^{-1}J_{z}^{(k)}L_{j}\right\Vert _{\func{op}} &\sim
&\left\Vert \frac{1}{\left( \left\langle \nabla _{x_{j}}\right\rangle
+\partial _{z_{j}}\right) }J_{z}^{(k)}\left( \left\langle \nabla
_{x_{j}}\right\rangle +\partial _{z_{j}}\right) \right\Vert _{\func{op}} \\
&\leqslant &\left\Vert \frac{1}{\left( \left\langle \nabla
_{x_{j}}\right\rangle +\partial _{z_{j}}\right) }J_{z}^{(k)}\left\langle
\partial _{z_{j}}\right\rangle \right\Vert _{\func{op}}+\left\Vert \frac{%
\left\langle \nabla _{x_{j}}\right\rangle }{\left( \left\langle \nabla
_{x_{j}}\right\rangle +\partial _{z_{j}}\right) }J_{z}^{(k)}\right\Vert _{%
\func{op}} \\
&\leqslant &\left\Vert \left\langle \partial _{z_{j}}\right\rangle
J_{z}^{(k)}\left\langle \partial _{z_{j}}\right\rangle ^{-1}\right\Vert _{%
\func{op}}+\left\Vert J_{z}^{(k)}\right\Vert _{\func{op}}\text{.}
\end{eqnarray*}%
For the same reason, $\Vert L_{j}J_{z}^{(k)}L_{j}^{-1}\Vert _{\func{op}%
},\Vert L_{i}L_{j}J_{z}^{(k)}L_{i}^{-1}L_{j}^{-1}\Vert _{\func{op}}$ and $%
\Vert L_{i}^{-1}L_{j}^{-1}J_{z}^{(k)}L_{i}L_{j}\Vert _{\func{op}}$ are all
finite. Although $J_{z}^{(k)}$ and the related operators listed are only in $%
\mathcal{L}^{\infty }\left( L^{2}\left( \mathbb{R}^{3k}\right) \right) $,
they are good enough for our purpose.

Taking $\limfunc{Tr}_{x}$ on both sides of hierarchy 
\eqref{hierarchy:BBGKY
hierarchy for scaled marginal densities}, we have that $\tilde{\gamma}%
_{z,N,\omega }^{(k)}$ satisfies the coupled BBGKY hierarchy:%
\begin{eqnarray}
i\partial _{t}\tilde{\gamma}_{z,N,\omega }^{(k)} &=&\sum_{j=1}^{k}\left[
-\partial _{z_{j}}^{2},\tilde{\gamma}_{z,N,\omega }^{(k)}\right] +\frac{1}{N}%
\sum_{i<j}^{k}\limfunc{Tr}\nolimits_{x}\left[ V_{N,\omega }\left(
r_{i}-r_{j}\right) ,\tilde{\gamma}_{N,\omega }^{(k)}\right]
\label{hierarchy:coupled BBGKY for the x-component} \\
&&+\frac{N-k}{N}\sum_{j=1}^{k}\limfunc{Tr}\nolimits_{z_{k+1}}\limfunc{Tr}%
\nolimits_{x}\left[ V_{N,\omega }\left( r_{j}-r_{k+1}\right) ,\tilde{\gamma}%
_{N,\omega }^{(k+1)}\right] .  \notag
\end{eqnarray}

Assume $t_{1}\leqslant t_{2},$ the above hierarchy yields%
\begin{eqnarray*}
&&\left\vert \limfunc{Tr}J_{z}^{(k)}\left( \tilde{\gamma}_{z,N,\omega
}^{(k)}\left( t_{1}\right) -\tilde{\gamma}_{z,N,\omega }^{(k)}\left(
t_{2}\right) \right) \right\vert \\
&\leqslant &\sum_{j=1}^{k}\int_{t_{1}}^{t_{2}}\left\vert \limfunc{Tr}%
J_{z}^{(k)}\left[ -\partial _{z_{j}}^{2},\tilde{\gamma}_{z,N,\omega }^{(k)}%
\right] \right\vert dt+\frac{1}{N}\sum_{i<j}^{k}\int_{t_{1}}^{t_{2}}\left%
\vert \limfunc{Tr}J_{z}^{(k)}\left[ V_{N,\omega }\left( r_{i}-r_{j}\right) ,%
\tilde{\gamma}_{N,\omega }^{(k)}\right] \right\vert dt \\
&&+\frac{N-k}{N}\sum_{j=1}^{k}\int_{t_{1}}^{t_{2}}\left\vert \limfunc{Tr}%
J_{z}^{(k)}\left[ V_{N,\omega }\left( r_{j}-r_{k+1}\right) ,\tilde{\gamma}%
_{N,\omega }^{(k+1)}\right] \right\vert dt. \\
&=&\sum_{j=1}^{k}\int_{t_{1}}^{t_{2}}\text{I}\left( t\right) dt+\frac{1}{N}%
\sum_{i<j}^{k}\int_{t_{1}}^{t_{2}}\text{II}\left( t\right) dt+\frac{N-k}{N}%
\sum_{j=1}^{k}\int_{t_{1}}^{t_{2}}\text{III}\left( t\right) dt.
\end{eqnarray*}%
For I, we have 
\begin{eqnarray*}
\text{I} &=&\left\vert \limfunc{Tr}J_{z}^{(k)}\left[ \left\langle \partial
_{z_{j}}\right\rangle ^{2},\tilde{\gamma}_{z,N,\omega }^{(k)}\right]
\right\vert \\
&=&\left\vert \limfunc{Tr}\left\langle \partial _{z_{j}}\right\rangle
^{-1}J_{z}^{(k)}\left\langle \partial _{z_{j}}\right\rangle ^{2}\tilde{\gamma%
}_{z,N,\omega }^{(k)}\left\langle \partial _{z_{j}}\right\rangle -\limfunc{Tr%
}\left\langle \partial _{z_{j}}\right\rangle J_{z}^{(k)}\left\langle
\partial _{z_{j}}\right\rangle ^{-1}\left\langle \partial
_{z_{j}}\right\rangle \tilde{\gamma}_{z,N,\omega }^{(k)}\left\langle
\partial _{z_{j}}\right\rangle \right\vert \\
&\leqslant &\left( \left\Vert \left\langle \partial _{z_{j}}\right\rangle
^{-1}J_{z}^{(k)}\left\langle \partial _{z_{j}}\right\rangle \right\Vert _{%
\func{op}}+\left\Vert \left\langle \partial _{z_{j}}\right\rangle
J_{z}^{(k)}\left\langle \partial _{z_{j}}\right\rangle ^{-1}\right\Vert _{%
\func{op}}\right) \limfunc{Tr}\left\langle \partial _{z_{j}}\right\rangle 
\tilde{\gamma}_{z,N,\omega }^{(k)}\left\langle \partial _{z_{j}}\right\rangle
\\
&=&C_{J}\limfunc{Tr}\left\langle \partial _{z_{j}}\right\rangle \tilde{\gamma%
}_{N,\omega }^{(k)}\left\langle \partial _{z_{j}}\right\rangle \\
&\leqslant &C_{J}
\end{eqnarray*}%
by the energy estimates (Corollary \ref{Corollary:Energy Bound for Marginal
Densities}).

Consider II and III, we have%
\begin{eqnarray*}
\text{II} &=&\left\vert \limfunc{Tr}J_{z}^{(k)}\left[ V_{N,\omega }\left(
r_{i}-r_{j}\right) ,\tilde{\gamma}_{N,\omega }^{(k)}\right] \right\vert \\
&=&|\limfunc{Tr}L_{i}^{-1}L_{j}^{-1}J_{z}^{(k)}L_{i}L_{j}W_{ij}L_{i}L_{j}%
\tilde{\gamma}_{N,\omega }^{(k)}L_{i}L_{j}-\limfunc{Tr}%
L_{i}L_{j}J_{z}^{(k)}L_{i}^{-1}L_{j}^{-1}L_{i}L_{j}\tilde{\gamma}_{N,\omega
}^{(k)}L_{i}L_{j}W_{ij}| \\
&\leqslant &\left( \left\Vert
L_{i}^{-1}L_{j}^{-1}J_{z}^{(k)}L_{i}L_{j}\right\Vert _{\func{op}}+\left\Vert
L_{i}L_{j}J_{z}^{(k)}L_{i}^{-1}L_{j}^{-1}\right\Vert _{\func{op}}\right)
\left\Vert W_{ij}\right\Vert _{\func{op}}\limfunc{Tr}L_{i}L_{j}\tilde{\gamma}%
_{N,\omega }^{(k)}L_{i}L_{j} \\
&\leqslant &C_{J}\text{,}
\end{eqnarray*}%
and similarly,%
\begin{eqnarray*}
\text{III} &=&\left\vert \limfunc{Tr}J_{z}^{(k)}\left[ V_{N,\omega }\left(
r_{j}-r_{k+1}\right) ,\tilde{\gamma}_{N,\omega }^{(k+1)}\right] \right\vert
\\
&=&|\limfunc{Tr}%
L_{j}^{-1}L_{k+1}^{-1}J_{z}^{(k)}L_{j}L_{k+1}W_{j(k+1)}L_{j}L_{k+1}\tilde{%
\gamma}_{N,\omega }^{(k+1)}L_{j}L_{k+1} \\
&&-\limfunc{Tr}L_{j}L_{k+1}J_{z}^{(k)}L_{j}^{-1}L_{k+1}^{-1}L_{j}L_{k+1}%
\tilde{\gamma}_{N,\omega }^{(k+1)}L_{j}L_{k+1}W_{j(k+1)}| \\
&\leqslant &\left( \left\Vert L_{j}^{-1}J_{z}^{(k)}L_{j}\right\Vert _{\func{%
op}}+\left\Vert L_{j}J_{z}^{(k)}L_{j}^{-1}\right\Vert _{\func{op}}\right)
\left\Vert W_{j(k+1)}\right\Vert _{\func{op}}\limfunc{Tr}L_{j}L_{k+1}\tilde{%
\gamma}_{N,\omega }^{(k+1)}L_{j}L_{k+1} \\
&\leqslant &C_{J},
\end{eqnarray*}%
where we have used the fact that $L_{k+1}$ and $L_{k+1}^{-1}$ commutes with $%
J_{z}^{(k)}$.

Collecting the estimates for I - III, we conclude the compactness of the
sequence $\Gamma _{z,N,\omega }(t)=\left\{ \tilde{\gamma}_{z,N,\omega
}^{(k)}\right\} _{k=1}^{N}$.
\end{proof}

\section{Limit Points Satisfy GP Hierarchy\label{Section:Convergence of The
Infinite Hierarchy}}

\begin{theorem}
\label{Theorem:Convergence to the Coupled Gross-Pitaevskii} Let $\Gamma
(t)=\left\{ \tilde{\gamma}^{(k)}\right\} _{k=1}^{\infty }$ be a $%
C_{1}N^{v_{1}(\beta )}\leqslant \omega \leqslant C_{2}N^{v_{2}(\beta )}$
limit point of $\left\{ \Gamma _{N,\omega }(t)=\left\{ \tilde{\gamma}%
_{N,\omega }^{(k)}\right\} _{k=1}^{N}\right\} $ with respect to the product
topology $\tau _{prod}$, then $\left\{ \tilde{\gamma}_{z}^{(k)}=\limfunc{Tr}%
_{x}\tilde{\gamma}^{(k)}\right\} _{k=1}^{\infty }$ is a solution to the
coupled focusing Gross-Pitaevskii hierarchy (\ref{hierarchy:Coupled GP in
differential form}) subject to initial data $\tilde{\gamma}_{z}^{(k)}\left(
0\right) =\left\vert \phi _{0}\right\rangle \left\langle \phi
_{0}\right\vert ^{\otimes k}$ with coupling constant $b_{0}=$ $\left\vert
\int V\left( r\right) dr\right\vert $, which, rewritten in integral form, is 
\begin{eqnarray}
\tilde{\gamma}_{z}^{(k)} &=&U^{(k)}(t)\tilde{\gamma}_{z}^{(k)}\left( 0\right)
\label{hierarchy:coupled Gross-Pitaevskii} \\
&&+ib_{0}\sum_{j=1}^{k}\int_{0}^{t}U^{(k)}(t-s)\limfunc{Tr}%
\nolimits_{z_{k+1}}\limfunc{Tr}\nolimits_{x}\left[ \delta \left(
r_{j}-r_{k+1}\right) ,\tilde{\gamma}^{(k+1)}\left( s\right) \right] ds, 
\notag
\end{eqnarray}%
where $U^{(k)}(t)=\dprod\limits_{j=1}^{k}e^{it\partial
_{z_{j}}^{2}}e^{-it\partial _{z_{j}^{\prime }}^{2}}.$
\end{theorem}

\begin{proof}
Passing to subsequences if necessary, we have 
\begin{eqnarray}
\lim_{\substack{ N,\omega \rightarrow \infty  \\ C_{1}N^{v_{1}(\beta
)}\leqslant \omega \leqslant C_{2}N^{v_{2}(\beta )}}}\sup_{t}\limfunc{Tr}%
J^{(k)}(\tilde{\gamma}_{N,\omega }^{(k)}(t)-\tilde{\gamma}^{(k)}(t)) &=&0%
\text{, }\forall J^{(k)}\in \mathcal{K}(L^{2}(\mathbb{R}^{3k})),
\label{condition:fast convergence} \\
\lim_{\substack{ N,\omega \rightarrow \infty  \\ C_{1}N^{v_{1}(\beta
)}\leqslant \omega \leqslant C_{2}N^{v_{2}(\beta )}}}\sup_{t}\limfunc{Tr}%
J_{z}^{(k)}(\tilde{\gamma}_{z,N,\omega }^{(k)}(t)-\tilde{\gamma}%
_{z}^{(k)}(t)) &=&0\text{, }\forall J_{z}^{(k)}\in \mathcal{K}(L^{2}(\mathbb{%
R}^{k})),  \notag
\end{eqnarray}%
via Theorems \ref{Theorem:Compactness of the scaled marginal density} and %
\ref{Theorem:Compactness of the x-marginal density}.

To establish \eqref{hierarchy:coupled Gross-Pitaevskii}, it suffices to test
the limit point against the test functions $J_{z}^{(k)}\in \mathcal{K}(L^{2}(%
\mathbb{R}^{k}))$ as in the proof of Theorem \ref{Theorem:Compactness of the
x-marginal density}. We will prove that the limit point satisfies 
\begin{equation}
\limfunc{Tr}J_{z}^{(k)}\tilde{\gamma}_{z}^{(k)}\left( 0\right) =\limfunc{Tr}%
J_{z}^{(k)}\left\vert \phi _{0}\right\rangle \left\langle \phi
_{0}\right\vert ^{\otimes k}
\label{equality:testing the limit pt with initial data}
\end{equation}%
and 
\begin{eqnarray}
&&\limfunc{Tr}J_{z}^{(k)}\tilde{\gamma}_{z}^{(k)}\left( t\right)
\label{hierarchy:testing the limit point} \\
&=&\limfunc{Tr}J_{z}^{(k)}U^{(k)}\left( t\right) \tilde{\gamma}%
_{z}^{(k)}\left( 0\right)  \notag \\
&&+ib_{0}\sum_{j=1}^{k}\int_{0}^{t}\limfunc{Tr}J_{z}^{(k)}U^{(k)}\left(
t-s\right) \left[ \delta \left( r_{j}-r_{k+1}\right) ,\tilde{\gamma}%
^{(k+1)}\left( s\right) \right] ds  \notag
\end{eqnarray}%
To this end, we use the coupled focusing BBGKY hierarchy 
\eqref{hierarchy:coupled
BBGKY for the x-component} satisfied by $\tilde{\gamma}_{z,N,\omega }^{(k)}$%
, which, written in the form needed here, is 
\begin{align*}
& \limfunc{Tr}J_{z}^{(k)}\tilde{\gamma}_{z,N,\omega }^{(k)}\left( t\right) \\
=& A+\frac{i}{N}\sum_{i<j}^{k}B+i\left( 1-\frac{k}{N}\right) \sum_{j=1}^{k}D,
\end{align*}%
where%
\begin{equation*}
A=\limfunc{Tr}J_{z}^{(k)}U^{(k)}\left( t\right) \tilde{\gamma}_{z,N,\omega
}^{(k)}\left( 0\right) ,
\end{equation*}%
\begin{equation*}
B=\int_{0}^{t}\limfunc{Tr}J_{z}^{(k)}U^{(k)}\left( t-s\right) \left[
-V_{N,\omega }\left( r_{i}-r_{j}\right) ,\tilde{\gamma}_{N,\omega
}^{(k)}\left( s\right) \right] ds,
\end{equation*}%
\begin{equation*}
D=\int_{0}^{t}\limfunc{Tr}J_{z}^{(k)}U^{(k)}\left( t-s\right) \left[
-V_{N,\omega }\left( r_{j}-r_{k+1}\right) ,\tilde{\gamma}_{N,\omega
}^{(k+1)}\left( s\right) \right] ds.
\end{equation*}%
By \eqref{condition:fast convergence}, we know 
\begin{eqnarray*}
\lim_{\substack{ N,\omega \rightarrow \infty  \\ C_{1}N^{v_{1}(\beta
)}\leqslant \omega \leqslant C_{2}N^{v_{2}(\beta )}}}\limfunc{Tr}J_{z}^{(k)}%
\tilde{\gamma}_{z,N,\omega }^{(k)}\left( t\right) &=&\limfunc{Tr}J_{z}^{(k)}%
\tilde{\gamma}_{z}^{(k)}\left( t\right) , \\
\lim_{\substack{ N,\omega \rightarrow \infty  \\ C_{1}N^{v_{1}(\beta
)}\leqslant \omega \leqslant C_{2}N^{v_{2}(\beta )}}}\limfunc{Tr}%
J_{z}^{(k)}U^{(k)}\left( t\right) \tilde{\gamma}_{z,N,\omega }^{(k)}\left(
0\right) &=&\limfunc{Tr}J_{z}^{(k)}U^{(k)}\left( t\right) \tilde{\gamma}%
_{z}^{(k)}\left( 0\right) .
\end{eqnarray*}%
With the argument in \cite[p.64]{Lieb2}, we infer, from assumption (b) of
Theorem \ref{Theorem:3D->2D BEC (Nonsmooth)}: 
\begin{equation*}
\tilde{\gamma}_{N,\omega }^{(1)}\left( 0\right) \rightarrow \left\vert
h_{1}\otimes \phi _{0}\right\rangle \left\langle h_{1}\otimes \phi
_{0}\right\vert \,,\quad \text{strongly in trace norm,}
\end{equation*}%
that 
\begin{equation*}
\tilde{\gamma}_{N,\omega }^{(k)}\left( 0\right) \rightarrow \left\vert
h_{1}\otimes \phi _{0}\right\rangle \left\langle h_{1}\otimes \phi
_{0}\right\vert ^{\otimes k}\,,\quad \text{strongly in trace norm.}
\end{equation*}%
Thus we have checked \eqref{equality:testing the limit pt with initial
data}, the left-hand side of \eqref{hierarchy:testing the limit point}, and
the first term on the right-hand side of 
\eqref{hierarchy:testing the limit
point} for the limit point. We are left to prove that 
\begin{eqnarray*}
\lim_{\substack{ N,\omega \rightarrow \infty  \\ C_{1}N^{v_{1}(\beta
)}\leqslant \omega \leqslant C_{2}N^{v_{2}(\beta )}}}\frac{B}{N} &=&0, \\
\lim_{\substack{ N,\omega \rightarrow \infty  \\ C_{1}N^{v_{1}(\beta
)}\leqslant \omega \leqslant C_{2}N^{v_{2}(\beta )}}}\left( 1-\frac{k}{N}%
\right) D &=&b_{0}\int_{0}^{t}J_{x}^{(k)}U^{(k)}(t-s)\left[ \delta \left(
r_{j}-r_{k+1}\right) ,\tilde{\gamma}^{(k+1)}\left( s\right) \right] ds.
\end{eqnarray*}%
We first use an argument similar to the estimate of $\text{II}$ and III in
the proof of Theorem \ref{Theorem:Compactness of the x-marginal density} to
prove that $\left\vert B\right\vert $ and $\left\vert D\right\vert $ are
bounded for every finite time $t$. In fact, since $U^{(k)}$ is a unitary
operator which commutes with Fourier multipliers, we have 
\begin{eqnarray*}
\left\vert B\right\vert &\leqslant &\int_{0}^{t}\left\vert \limfunc{Tr}%
J_{z}^{(k)}U^{(k)}\left( t-s\right) \left[ V_{N,\omega }\left(
r_{i}-r_{j}\right) ,\tilde{\gamma}_{N,\omega }^{(k)}\left( s\right) \right]
\right\vert ds \\
&=&\int_{0}^{t}ds|\limfunc{Tr}%
L_{i}^{-1}L_{j}^{-1}J_{z}^{(k)}L_{i}L_{j}U^{(k)}\left( t-s\right)
W_{ij}L_{i}L_{j}\tilde{\gamma}_{N,\omega }^{(k)}\left( s\right) L_{i}L_{j} \\
&&-\limfunc{Tr}L_{i}L_{j}J_{z}^{(k)}L_{i}^{-1}L_{j}^{-1}U^{(k)}\left(
t-s\right) L_{i}L_{j}\tilde{\gamma}_{N,\omega }^{(k)}\left( s\right)
L_{i}L_{j}W_{ij}| \\
&\leqslant &\int_{0}^{t}ds\left\Vert
L_{i}^{-1}L_{j}^{-1}J_{z}^{(k)}L_{i}L_{j}\right\Vert _{\func{op}}\left\Vert
U^{(k)}\right\Vert _{\func{op}}\left\Vert W_{ij}\right\Vert \limfunc{Tr}%
L_{i}L_{j}\tilde{\gamma}_{N,\omega }^{(k)}\left( s\right) L_{i}L_{j} \\
&&+\int_{0}^{t}ds\left\Vert
L_{i}L_{j}J_{z}^{(k)}L_{i}^{-1}L_{j}^{-1}\right\Vert _{\func{op}}\left\Vert
U^{(k)}\right\Vert _{\func{op}}\left\Vert W_{ij}\right\Vert \limfunc{Tr}%
L_{i}L_{j}\tilde{\gamma}_{N,\omega }^{(k)}\left( s\right) L_{i}L_{j} \\
&\leqslant &C_{J}t.
\end{eqnarray*}%
That is 
\begin{equation*}
\lim_{\substack{ N,\omega \rightarrow \infty  \\ C_{1}N^{v_{1}(\beta
)}\leqslant \omega \leqslant C_{2}N^{v_{2}(\beta )}}}\frac{B}{N}=\lim 
_{\substack{ N,\omega \rightarrow \infty  \\ C_{1}N^{v_{1}(\beta )}\leqslant
\omega \leqslant C_{2}N^{v_{2}(\beta )}}}\frac{kD}{N}=0.
\end{equation*}

We now use Lemma \ref{Lemma:ComparingDeltaFunctions} (stated and proved in
Appendix \ref{A:Sobolev}), which compares the $\delta -$function and its
approximation, to prove 
\begin{equation}
\lim_{\substack{ N,\omega \rightarrow \infty  \\ C_{1}N^{v_{1}(\beta
)}\leqslant \omega \leqslant C_{2}N^{v_{2}(\beta )}}}D=b_{0}\int_{0}^{t}%
\limfunc{Tr}J_{z}^{(k)}U^{(k)}(t-s)\left[ \delta \left( r_{j}-r_{k+1}\right)
,\tilde{\gamma}^{(k+1)}\left( s\right) \right] ds,
\label{limit:converges to delta function}
\end{equation}%
Pick a probability measure $\rho \in L^{1}\left( \mathbb{R}^{3}\right) $ and
define $\rho _{\alpha }\left( r\right) =\alpha ^{-3}\rho \left( \frac{r}{%
\alpha }\right) .$ Let $J_{s-t}^{(k)}=J_{z}^{(k)}U^{(k)}\left( t-s\right) $,
we have 
\begin{align*}
\hspace{0.3in}& \hspace{-0.3in}\left\vert \limfunc{Tr}J_{z}^{(k)}U^{(k)}%
\left( t-s\right) \left( -V_{N,\omega }\left( r_{j}-r_{k+1}\right) \tilde{%
\gamma}_{N,\omega }^{(k+1)}\left( s\right) -b_{0}\delta \left(
r_{j}-r_{k+1}\right) \tilde{\gamma}^{(k+1)}\left( s\right) \right)
\right\vert \\
& =\text{I}+\text{II}+\text{III}+\text{IV}
\end{align*}%
where%
\begin{equation*}
\text{I}=\left\vert \limfunc{Tr}J_{s-t}^{(k)}\left( -V_{N,\omega }\left(
r_{j}-r_{k+1}\right) -b_{0}\delta \left( r_{j}-r_{k+1}\right) \right) \tilde{%
\gamma}_{N,\omega }^{(k+1)}\left( s\right) \right\vert ,
\end{equation*}%
\begin{equation*}
\text{II}=b_{0}\left\vert \limfunc{Tr}J_{s-t}^{(k)}\left( \delta \left(
r_{j}-r_{k+1}\right) -\rho _{\alpha }\left( r_{j}-r_{k+1}\right) \right) 
\tilde{\gamma}_{N,\omega }^{(k+1)}\left( s\right) \right\vert ,
\end{equation*}%
\begin{equation*}
\text{III}=b_{0}\left\vert \limfunc{Tr}J_{s-t}^{(k)}\rho _{\alpha }\left(
r_{j}-r_{k+1}\right) \left( \tilde{\gamma}_{N,\omega }^{(k+1)}\left(
s\right) -\tilde{\gamma}^{(k+1)}\left( s\right) \right) \right\vert ,
\end{equation*}%
\begin{equation*}
\text{IV}=b_{0}\left\vert \limfunc{Tr}J_{s-t}^{(k)}\left( \rho _{\alpha
}\left( r_{j}-r_{k+1}\right) -\delta \left( r_{j}-r_{k+1}\right) \right) 
\tilde{\gamma}^{(k+1)}\left( s\right) \right\vert .
\end{equation*}

Consider I. Write $V_{\omega }(r)=\frac{1}{\omega }V(\frac{x}{\sqrt{\omega }}%
,z),$ we have $V_{N,\omega }=\left( N\omega \right) ^{3\beta }V_{\omega
}(\left( N\omega \right) ^{\beta }r)$, Lemma \ref%
{Lemma:ComparingDeltaFunctions} then yields 
\begin{eqnarray*}
\text{I} &\leqslant &\frac{Cb_{0}}{\left( N\omega \right) ^{\beta \kappa }}%
\left( \int \left\vert V_{\omega }(r)\right\vert \left\vert r\right\vert
^{\kappa }dr\right) \\
&&\times \left( \left\Vert L_{j}J_{z}^{(k)}L_{j}^{-1}\right\Vert _{\func{op}%
}+\left\Vert L_{j}^{-1}J_{z}^{(k)}L_{j}\right\Vert _{\func{op}}\right)
L_{j}L_{k+1}\tilde{\gamma}_{N,\omega }^{(k+1)}\left( s\right) L_{j}L_{k+1} \\
&=&C_{J}\frac{\left( \int \left\vert V_{\omega }(r)\right\vert \left\vert
r\right\vert ^{\kappa }dr\right) }{\left( N\omega \right) ^{\beta \kappa }}.
\end{eqnarray*}%
Notice that $\left( \int \left\vert V_{\omega }(r)\right\vert \left\vert
r\right\vert ^{\kappa }dr\right) $ grows like $\left( \sqrt{\omega }\right)
^{\kappa }$, so $I\leqslant C_{J}\left( \frac{\sqrt{\omega }}{\left( N\omega
\right) ^{\beta }}\right) ^{\kappa }$ which converges to zero as $N,\omega
\rightarrow \infty $ in the way in which $N\geqslant \omega ^{\frac{1}{%
2\beta }-1+}$. So we have proved 
\begin{equation*}
\lim_{\substack{ N,\omega \rightarrow \infty  \\ C_{1}N^{v_{1}(\beta
)}\leqslant \omega \leqslant C_{2}N^{v_{2}(\beta )}}}I=0.
\end{equation*}%
Similarly, for $\text{II}$ and $\text{IV}$, via Lemma \ref%
{Lemma:ComparingDeltaFunctions}, we have 
\begin{eqnarray*}
\text{II} &\leqslant &Cb_{0}\alpha ^{\kappa }\left( \left\Vert
L_{j}J_{z}^{(k)}L_{j}^{-1}\right\Vert _{\func{op}}+\left\Vert
L_{j}^{-1}J_{z}^{(k)}L_{j}\right\Vert _{\func{op}}\right) \limfunc{Tr}%
L_{j}L_{k+1}\tilde{\gamma}_{N,\omega }^{(k+1)}\left( s\right) L_{j}L_{k+1} \\
&\leqslant &C_{J}\alpha ^{\kappa }\text{ (Corollary \ref{Corollary:Energy
Bound for Marginal Densities})} \\
\text{IV} &\leqslant &Cb_{0}\alpha ^{\kappa }\left( \left\Vert
L_{j}J_{z}^{(k)}L_{j}^{-1}\right\Vert _{\func{op}}+\left\Vert
L_{j}^{-1}J_{z}^{(k)}L_{j}\right\Vert _{\func{op}}\right) \limfunc{Tr}%
L_{j}L_{k+1}\tilde{\gamma}^{(k+1)}\left( s\right) L_{j}L_{k+1} \\
&\leqslant &C_{J}\alpha ^{\kappa }\text{ (Corollary \ref%
{Corollary:LimitMustBeAProduct})}
\end{eqnarray*}%
that is%
\begin{equation*}
\text{II}\leqslant C_{J}\alpha ^{\kappa }\text{ and IV}\leqslant C_{J}\alpha
^{\kappa },
\end{equation*}%
due to the energy estimate (Corollary \ref{Corollary:LimitMustBeAProduct}).
Hence II and IV converges to $0$ as $\alpha \rightarrow 0$, uniformly in $%
N,\omega .$

For $\text{III}$, 
\begin{eqnarray*}
\text{III} &\leqslant &b_{0}\left\vert \limfunc{Tr}J_{s-t}^{(k)}\rho
_{\alpha }\left( r_{j}-r_{k+1}\right) \frac{1}{1+\varepsilon L_{k+1}}\left( 
\tilde{\gamma}_{N,\omega }^{(k+1)}\left( s\right) -\tilde{\gamma}%
^{(k+1)}\left( s\right) \right) \right\vert \\
&&+b_{0}\left\vert \limfunc{Tr}J_{s-t}^{(k)}\rho _{\alpha }\left(
r_{j}-r_{k+1}\right) \frac{\varepsilon L_{k+1}}{1+\varepsilon L_{k+1}}\left( 
\tilde{\gamma}_{N,\omega }^{(k+1)}\left( s\right) -\tilde{\gamma}%
^{(k+1)}\left( s\right) \right) \right\vert .
\end{eqnarray*}%
The first term in the above estimate goes to zero as $N,\omega \rightarrow
\infty $ for every $\varepsilon >0$, since we have assumed condition {(\ref%
{condition:fast convergence})} and $J_{s-t}^{(k)}\rho _{\alpha }\left(
r_{j}-r_{k+1}\right) \left( 1+\varepsilon L_{k+1}\right) ^{-1}$ is a compact
operator. Due to the energy bounds on $\tilde{\gamma}_{N,\omega }^{(k+1)}$
and $\tilde{\gamma}^{(k+1)}$, the second term tends to zero as $\varepsilon
\rightarrow 0$, uniformly in $N$ and $\omega $.

Putting together the estimates for $\text{I}$-$\text{IV}$, we have justified
limit ({\ref{limit:converges to delta function})}. Hence, we have obtained
Theorem \ref{Theorem:Convergence to the Coupled Gross-Pitaevskii}.
\end{proof}

Combining Corollary \ref{Corollary:LimitMustBeAProduct} and Theorem \ref%
{Theorem:Convergence to the Coupled Gross-Pitaevskii}, we see that $\tilde{%
\gamma}_{z}^{(k)}$ in fact solves the 1D focusing Gross-Pitaevskii hierarchy
with the desired coupling constant $b_{0}\left( \int \left\vert h_{1}\left(
x\right) \right\vert ^{4}dx\right) .$

\begin{corollary}
\label{Theorem:Convergence to the 1D GP}Let $\Gamma (t)=\left\{ \tilde{\gamma%
}^{(k)}\right\} _{k=1}^{\infty }$ be a $N\geqslant \omega ^{v(\beta
)+\varepsilon }$ limit point of $\left\{ \Gamma _{N,\omega }(t)=\left\{ 
\tilde{\gamma}_{N,\omega }^{(k)}\right\} _{k=1}^{N}\right\} $ with respect
to the product topology $\tau _{prod}$, then $\left\{ \tilde{\gamma}%
_{z}^{(k)}=\limfunc{Tr}_{x}\tilde{\gamma}^{(k)}\right\} _{k=1}^{\infty }$ is
a solution to the 1D Gross-Pitaevskii hierarchy (\ref{hierarchy:1D GP in
differential form}) subject to initial data $\tilde{\gamma}_{z}^{(k)}\left(
0\right) =\left\vert \phi _{0}\right\rangle \left\langle \phi
_{0}\right\vert ^{\otimes k}$ with coupling constant $b_{0}\left( \int
\left\vert h_{1}\left( x\right) \right\vert ^{4}dx\right) $, which,
rewritten in integral form, is 
\begin{eqnarray}
&&\tilde{\gamma}_{z}^{(k)}  \label{hierarchy:1D GP} \\
&=&U^{(k)}(t)\tilde{\gamma}_{z}^{(k)}\left( 0\right)  \notag \\
&&+ib_{0}\left( \int \left\vert h_{1}\left( x\right) \right\vert
^{4}dx\right) \sum_{j=1}^{k}\int_{0}^{t}U^{(k)}(t-s)\limfunc{Tr}%
\nolimits_{z_{k+1}}\left[ \delta \left( z_{j}-z_{k+1}\right) ,\tilde{\gamma}%
_{z}^{(k+1)}\left( s\right) \right] ds.  \notag
\end{eqnarray}
\end{corollary}

\begin{proof}
This is a direct computation by plugging (\ref{E:e-8}) into (\ref%
{hierarchy:coupled Gross-Pitaevskii}).
\end{proof}

\appendix

\section{Basic Operator Facts and Sobolev-type Lemmas}

\label{A:Sobolev}

\begin{lemma}[{\protect\cite[Lemma A.3]{E-S-Y2}}]
\label{Lemma:ESYSoblevLemma}Let $L_{j}=\left( 1-\triangle _{r_{j}}\right) ^{%
\frac{1}{2}}$, then we have 
\begin{equation*}
\left\Vert L_{i}^{-1}L_{j}^{-1}V\left( r_{i}-r_{j}\right)
L_{i}^{-1}L_{j}^{-1}\right\Vert _{\func{op}}\leqslant C\left\Vert
V\right\Vert _{L^{1}}.
\end{equation*}
\end{lemma}

\begin{lemma}
\label{Lemma:ComparingDeltaFunctions}Let $f\in L^{1}\left( \mathbb{R}%
^{3}\right) $ such that $\int_{\mathbb{R}^{3}}\left\langle r\right\rangle ^{%
\frac{1}{2}}\left\vert f\left( r\right) \right\vert dr<\infty $ and $\int_{%
\mathbb{R}^{3}}f\left( r\right) dr=1$ but we allow that $f$ not be
nonnegative everywhere. Define $f_{\alpha }\left( r\right) =\alpha
^{-3}f\left( \frac{r}{\alpha }\right) .$ Then, for every $\kappa \in \left(
0,1/2\right) $ , there exists $C_{\kappa }>0$ s.t. 
\begin{align*}
\hspace{0.3in}& \hspace{-0.3in}\left\vert \limfunc{Tr}J^{(k)}\left(
f_{\alpha }\left( r_{j}-r_{k+1}\right) -\delta \left( r_{j}-r_{k+1}\right)
\right) \gamma ^{(k+1)}\right\vert \\
& \leqslant C_{\kappa }\left( \int \left\vert f\left( r\right) \right\vert
\left\vert r\right\vert ^{\kappa }dr\right) \alpha ^{\kappa }\left(
\left\Vert L_{j}J^{(k)}L_{j}^{-1}\right\Vert _{\func{op}}+\left\Vert
L_{j}^{-1}J^{(k)}L_{j}\right\Vert _{\func{op}}\right) \limfunc{Tr}%
L_{j}L_{k+1}\gamma ^{(k+1)}L_{j}L_{k+1}
\end{align*}%
for all nonnegative $\gamma ^{(k+1)}\in \mathcal{L}^{1}\left( L^{2}\left( 
\mathbb{R}^{3k+3}\right) \right) .$
\end{lemma}

\begin{proof}
Same as \cite[Lemma A.3]{C-HFocusing} and \cite[Lemma 2]{C-H3Dto2D}. See 
\cite{Kirpatrick, TChenAndNP,E-S-Y2} for similar lemmas.
\end{proof}

\begin{lemma}[some standard operator inequalities]
\quad \label{L:op-stuff}

\begin{enumerate}
\item \label{I:op-1} Suppose that $A\geq 0$, $P_j=P_j^*$, and $I=P_0+P_1$.
Then $A \leq 2P_0 A P_0 + 2P_1 A P_1$.

\item \label{I:op-2} If $A\geq B \geq 0$, and $AB=BA$, then $A^\alpha \geq
B^\alpha$ for any $\alpha \geq 0$.

\item \label{I:op-3} If $A_1 \geq A_2 \geq 0$, $B_1 \geq B_2 \geq 0$ and $%
A_iB_j = B_jA_i$ for all $1\leq i,j \leq 2$, then $A_1B_1 \geq A_2B_2$.

\item \label{I:op-4} If $A\geq 0$ and $AB=BA$, then $A^{1/2}B=B A^{1/2}$.
\end{enumerate}
\end{lemma}

\begin{proof}
For (1), $\Vert A^{1/2}f\Vert ^{2}=\Vert A^{1/2}(P_{0}+P_{1})f\Vert ^{2}\leq
2\Vert A^{1/2}P_{0}f\Vert ^{2}+2\Vert A^{1/2}P_{1}f\Vert ^{2}$. The rest are
standard facts in operator theory.
\end{proof}

\begin{lemma}
\label{L:coercivity}Recall%
\begin{equation*}
\tilde{S}=(1-\partial _{z}^{2}+\omega (-2-\triangle _{x}+\left\vert
x\right\vert ^{2}))^{1/2},
\end{equation*}%
we have 
\begin{gather}
\tilde{S}^{2}\gtrsim 1-\Delta _{r}  \label{E:tilde-S-1} \\
\tilde{S}^{2}P_{\geqslant 1}\gtrsim P_{\geqslant 1}(1-\partial
_{z}^{2}-\omega \triangle _{x}+\omega \left\vert x\right\vert
^{2})P_{\geqslant 1}  \label{E:tilde-S-2} \\
\tilde{S}^{2}P_{\geqslant 1}\gtrsim \omega P_{\geqslant 1}
\label{E:tilde-S-3}
\end{gather}
\end{lemma}

\begin{proof}
Directly from the definition of $\tilde{S}$, we have 
\begin{equation}
\underbrace{P_{\geqslant 1}(1-\partial _{z}^{2}-\omega \triangle _{x}+\omega
\left\vert x\right\vert ^{2})P_{\geqslant 1}}_{\text{all terms positive}%
}=2\omega P_{\geqslant 1}+\tilde{S}^{2}P_{\geqslant 1}.  \label{E:tS-6}
\end{equation}%
The eigenvalues of the 2D Hermite operator $-\Delta _{x}+\left\vert
x\right\vert ^{2}$ are $\left\{ 2k+2\right\} _{k=0}^{\infty }$. So 
\begin{equation}
2\omega P_{\geqslant 1}\leqslant \omega (-2-\triangle _{x}+\left\vert
x\right\vert ^{2})P_{\geqslant 1}\leqslant \tilde{S}^{2}P_{\geqslant 1}.
\label{E:tS-7}
\end{equation}%
\eqref{E:tilde-S-2} and \eqref{E:tilde-S-3} immediately follow from (\ref%
{E:tS-6}) and (\ref{E:tS-7}).

We now establish \eqref{E:tilde-S-1} using \eqref{E:tilde-S-2}. On the one
hand, we have 
\begin{equation}
\tilde{S}^{2}\geqslant (1-\partial _{z}^{2})  \label{E:tS-1}
\end{equation}%
On the other hand, 
\begin{equation}
P_{0}(-\triangle _{x})P_{0}\lesssim 1\leqslant \tilde{S}^{2}  \label{E:tS-2}
\end{equation}%
since $P_{0}$ is merely the projection onto the smooth function $Ce^{-\frac{%
\left\vert x\right\vert ^{2}}{2}}$. Moreover, by \eqref{E:tilde-S-2}, 
\begin{equation}
P_{\geqslant 1}(-\triangle _{x})P_{\geqslant 1}\leqslant \tilde{S}%
^{2}P_{\geqslant 1}\leqslant \tilde{S}^{2}  \label{E:tS-3}
\end{equation}%
Thus Lemma \ref{L:op-stuff}\eqref{I:op-1}, \eqref{E:tS-2} and \eqref{E:tS-3}
together imply, 
\begin{equation}
-\triangle _{x}\lesssim \tilde{S}^{2}  \label{E:tS-4}
\end{equation}%
The claimed inequality \eqref{E:tilde-S-1} then follows from \eqref{E:tS-1}
and \eqref{E:tS-4}.
\end{proof}

\begin{lemma}
\label{L:trace-of-tp-kernel} Suppose $\sigma: L^2(\mathbb{R}^{3k}) \to L^2(%
\mathbb{R}^{3k})$ has kernel 
\begin{equation*}
\sigma(\mathbf{r}_k, \mathbf{r}_k^{\prime }) = \int \psi(\mathbf{r}_k, 
\mathbf{r}_{N-k})\overline{\psi}(\mathbf{r}_k^{\prime }, \mathbf{r}_{N-k})\,
d\mathbf{r}_{N-k} \,,
\end{equation*}
for some $\psi \in L^2(\mathbb{R}^{3N})$, and let $A,B:L^2(\mathbb{R}^{3k})
\to L^2(\mathbb{R}^{3k})$. Then the composition $A\sigma B$ has kernel 
\begin{equation*}
(A\sigma B)(\mathbf{r}_k, \mathbf{r}_k^{\prime }) = \int (A\psi)(\mathbf{r}%
_k, \mathbf{r}_{N-k}) (\overline{B^* \psi})(\mathbf{r}_k^{\prime }, \mathbf{r%
}_{N-k}) \, d\mathbf{r}_{N-k}
\end{equation*}
It follows that 
\begin{equation*}
\func{Tr} A\sigma B = \langle A\psi, B^* \psi \rangle \,.
\end{equation*}
\end{lemma}

Let $\mathcal{K}_k$ denote the class of compact operators on $L^2(\mathbb{R}%
^{3k})$, $\mathcal{L}^1_k$ denote the trace class operators on $L^2(\mathbb{R%
}^{3k})$, and $\mathcal{L}^2_k$ denote the Hilbert-Schmidt operators on $L^2(%
\mathbb{R}^{3k})$. We have 
\begin{equation*}
\mathcal{L}_k^1 \subset \mathcal{L}_k^2 \subset \mathcal{K}_k
\end{equation*}
For an operator $J$ on $L^2(\mathbb{R}^{3k})$, let $|J| = (J^*J)^{1/2}$ and
denote by $J(\mathbf{r}_k, \mathbf{r}^{\prime }_k)$ the kernel of $J$ and $%
|J|(\mathbf{r}_k, \mathbf{r}^{\prime }_k)$ the kernel of $|J|$, which
satisfies $|J|(\mathbf{r}_k, \mathbf{r}_k^{\prime }) \geq 0$. Let 
\begin{equation*}
\mu_1 \geq \mu_2 \geq \cdots \geq 0
\end{equation*}
be the eigenvalues of $|J|$ repeated according to multiplicity (the \emph{%
singular values} of $J$). Then 
\begin{equation*}
\| J \|_{\mathcal{K}_k} = \| \mu_n \|_{\ell^\infty_n} = \mu_1 = \| \, |J|
\,\|_{\func{op}} = \|J\|_{\func{op}}
\end{equation*}
\begin{equation*}
\| J \|_{\mathcal{L}^2_k} = \| \mu_n \|_{\ell^2_n} = \|J(\mathbf{r}_k, 
\mathbf{r}_k^{\prime })\|_{L^2(\mathbf{r}_k, \mathbf{r}_k^{\prime })} = (%
\func{Tr} J^*J)^{1/2}
\end{equation*}
\begin{equation*}
\| J \|_{\mathcal{L}^1_k} = \| \mu_n \|_{\ell^1_n} = \| |J|(\mathbf{r}_k,%
\mathbf{r}_k) \|_{L^1({\mathbf{r}_k})} = \func{Tr} |J|
\end{equation*}
The topology on $\mathcal{K}_k$ coincides with the operator topology, and $%
\mathcal{K}_k$ is a closed subspace of the space of bounded operators on $%
L^2(\mathbb{R}^{3k})$.

\begin{lemma}
\label{L:compact-operator-truncation}On the one hand, let $\chi $ be a
smooth function on $\mathbb{R}^{3}$ such that $\chi (\xi )=1$ for $|\xi
|\leq 1$ and $\chi (\xi )=0$ for $|\xi |\geq 2$. Let 
\begin{equation*}
(Q_{M}f)(\mathbf{r}_{k})=\int e^{i\mathbf{r}_{k}\cdot \mathbf{\xi }%
_{k}}\prod_{j=1}^{k}\chi (M^{-1}\xi _{j})\hat{f}(\mathbf{\xi }_{k})\,d%
\mathbf{\xi }_{k}
\end{equation*}%
On the other hand, with respect to the spectral decomposition of $L^{2}(%
\mathbb{R}^{2})$ corresponding to the operator $H_{j}=-\triangle
_{x_{j}}^{2}+\left\vert x_{j}\right\vert ^{2}$, let $X_{M}^{j}$ be the
orthogonal projection onto the sum of the first $M$ eigenspaces (in the $%
x_{j}$ variable only) and let 
\begin{equation*}
R_{M}=\prod_{j=1}^{k}X_{M}^{j}.
\end{equation*}%
We then have the following:

\begin{enumerate}
\item Suppose that $J$ is a compact operator. Then $J_M \overset{\mathrm{def}%
}{=} R_M Q_M J Q_M R_M \to J$ in the operator norm.

\item $H_jJ_M$, $J_MH_j$, $\Delta_{r_j} J_M$ and $J_M \Delta_{r_j}$ are all
bounded.

\item There exists a countable dense subset $\{T_i\}$ of the closed unit
ball in the space of bounded operators on $L^2(\mathbb{R}^{3k})$ such that
each $T_i$ is compact and in fact for each $i$ there exists $M$ (depending
on $i$) and $Y_i\in \mathcal{K}_k$ with $\|Y_i\|_{\func{op}} \leq 1$ such
that $T_i = R_M Q_M Y_i Q_M R_M$.
\end{enumerate}
\end{lemma}

\begin{proof}
(1) If $S_{n}\rightarrow S$ strongly and $J\in \mathcal{K}_{k}$, then $%
S_{n}J\rightarrow SJ$ in the operator norm and $JS_{n}\rightarrow JS$ in the
operator norm. (2) is straightforward. For (3), start with a subset $%
\{Y_{n}\}$ of the closed unit ball in the space of bounded operators on $%
L^{2}(\mathbb{R}^{3k})$ such that each $Y_{n}$ is compact. Then let $%
\{T_{i}\}$ be an enumeration of the set $R_{M}Q_{M}Y_{n}Q_{M}R_{M}$ where $M$
ranges over the dyadic integers. By (1) this collection will still be dense.
The $\{ Y_i\}$ in the statement of (3) is just a reindexing of $\{Y_n\}$.
\end{proof}

\section{Deducing Theorem \protect\ref{Theorem:3D->2D BEC (Nonsmooth)} from
Theorem \protect\ref{Theorem:3D->2D BEC}}

\label{A:equivalence}

We first give the following lemma.

\begin{lemma}
\label{Prop:approximation of initial}Assume $\tilde{\psi}_{N,\omega }(0)$
satisfies (a), (b) and (c) in Theorem \ref{Theorem:3D->2D BEC (Nonsmooth)}.
Let $\chi \in C_{0}^{\infty }\left( \mathbb{R}\right) $ be a cut-off such
that $0\leqslant \chi \leqslant 1$, $\chi \left( s\right) =1$ for $%
0\leqslant s\leqslant 1$ and $\chi \left( s\right) =0$ for $s\geqslant 2.$
For $\kappa >0,$ we define an approximation of $\tilde{\psi}_{N,\omega }(0)$
by 
\begin{equation*}
\tilde{\psi}_{N,\omega }^{\kappa }(0)=\frac{\chi \left( \kappa \left( \tilde{%
H}_{N,\omega }-2N\omega \right) /N\right) \tilde{\psi}_{N,\omega }(0)}{%
\left\Vert \chi \left( \kappa \left( \tilde{H}_{N,\omega }-2N\omega \right)
/N\right) \tilde{\psi}_{N,\omega }(0)\right\Vert }.
\end{equation*}%
This approximation has the following properties:

(i) $\tilde{\psi}_{N,\omega }^{\kappa }(0)$ verifies the energy condition 
\begin{equation*}
\langle \tilde{\psi}_{N,\omega }^{\kappa }(0),(\tilde{H}_{N,\omega
}-2N\omega )^{k}\tilde{\psi}_{N,\omega }^{\kappa }(0)\rangle \leqslant \frac{%
2^{k}N^{k}}{\kappa ^{k}}.
\end{equation*}%
(ii) 
\begin{equation*}
\sup_{N,\omega }\left\Vert \tilde{\psi}_{N,\omega }(0)-\tilde{\psi}%
_{N,\omega }^{\kappa }(0)\right\Vert _{L^{2}}\leqslant C\kappa ^{\frac{1}{2}%
}.
\end{equation*}%
(iii) For small enough $\kappa >0$, $\tilde{\psi}_{N,\omega }^{\kappa }(0)$
is asymptotically factorized as well 
\begin{equation*}
\lim_{N,\omega \rightarrow \infty }\limfunc{Tr}\left\vert \tilde{\gamma}%
_{N,\omega }^{\kappa ,(1)}(0,x_{1},z_{1};x_{1}^{\prime },z_{1}^{\prime
})-h(x_{1})h(x_{1}^{\prime })\phi _{0}(z_{1})\overline{\phi _{0}}%
(z_{1}^{\prime })\right\vert =0,
\end{equation*}%
where $\tilde{\gamma}_{N,\omega }^{\kappa ,(1)}\left( 0\right) $ is the
one-particle marginal density associated with $\tilde{\psi}_{N,\omega
}^{\kappa }(0),$ and $\phi _{0}$ is the same as in assumption (b) in Theorem %
\ref{Theorem:3D->2D BEC (Nonsmooth)}.
\end{lemma}

\begin{proof}
Let us write $\chi \left( \kappa \left( \tilde{H}_{N,\omega }-2N\omega
\right) \right) $ as $\chi $ and $\tilde{\psi}_{N,\omega }(0)$ as $\tilde{%
\psi}_{N,\omega }.$ This proof closely follows \cite[Proposition 8.1 (i)-(ii)%
]{E-S-Y3} and \cite[Proposition 5.1 (iii)]{E-S-Y2}

(i) is from definition. In fact, denote the characteristic function of $%
\left[ 0,\lambda \right] $ with $\mathbf{1}(s\leqslant \lambda ).$ We see
that 
\begin{equation*}
\chi \left( \kappa \left( \tilde{H}_{N,\omega }-2N\omega \right) /N\right) =%
\mathbf{1}(\tilde{H}_{N,\omega }-2N\omega \leqslant 2N/\kappa )\chi \left(
\kappa \left( \tilde{H}_{N,\omega }-2N\omega \right) /N\right) .
\end{equation*}%
Thus%
\begin{eqnarray*}
&&\left\langle \tilde{\psi}_{N,\omega }^{\kappa }(0),\left( \tilde{H}%
_{N,\omega }-2N\omega \right) ^{k}\tilde{\psi}_{N,\omega }^{\kappa
}(0)\right\rangle \\
&=&\left\langle \frac{\chi \tilde{\psi}_{N,\omega }}{\left\Vert \chi \tilde{%
\psi}_{N,\omega }\right\Vert },\mathbf{1}(\tilde{H}_{N,\omega }-2N\omega
\leqslant 2N/\kappa )\left( \tilde{H}_{N,\omega }-2N\omega \right) ^{k}\frac{%
\chi \tilde{\psi}_{N,\omega }}{\left\Vert \chi \tilde{\psi}_{N,\omega
}\right\Vert }\right\rangle \\
&\leqslant &\left\Vert \mathbf{1}(\tilde{H}_{N,\omega }-2N\omega \leqslant
2N/\kappa )\left( \tilde{H}_{N,\omega }-2N\omega \right) ^{k}\right\Vert
_{op} \\
&\leqslant &\frac{2^{k}N^{k}}{\kappa ^{k}}.
\end{eqnarray*}%
We prove (ii) with a slightly modified proof of \cite[Proposition 8.1 (ii)]%
{E-S-Y3}. We still have%
\begin{eqnarray*}
&&\left\Vert \tilde{\psi}_{N,\omega }^{\kappa }-\tilde{\psi}_{N,\omega
}\right\Vert _{L^{2}} \\
&\leqslant &\left\Vert \chi \tilde{\psi}_{N,\omega }-\tilde{\psi}_{N,\omega
}\right\Vert _{L^{2}}+\left\Vert \frac{\chi \tilde{\psi}_{N,\omega }}{%
\left\Vert \chi \tilde{\psi}_{N,\omega }\right\Vert }-\chi \tilde{\psi}%
_{N,\omega }\right\Vert _{L^{2}} \\
&\leqslant &\left\Vert \chi \tilde{\psi}_{N,\omega }-\tilde{\psi}_{N,\omega
}\right\Vert _{L^{2}}+\left\vert 1-\left\Vert \chi \tilde{\psi}_{N,\omega
}\right\Vert \right\vert \\
&\leqslant &2\left\Vert \chi \tilde{\psi}_{N,\omega }-\tilde{\psi}_{N,\omega
}\right\Vert _{L^{2}},
\end{eqnarray*}%
where%
\begin{eqnarray*}
\left\Vert \chi \tilde{\psi}_{N,\omega }-\tilde{\psi}_{N,\omega }\right\Vert
_{L^{2}}^{2} &=&\left\langle \psi _{N},\left( 1-\chi \left( \frac{\kappa
\left( \tilde{H}_{N,\omega }-2N\omega \right) }{N}\right) \right) ^{2}\psi
_{N}\right\rangle \\
&\leqslant &\left\langle \psi _{N},\mathbf{1}(\frac{\kappa \left( \tilde{H}%
_{N,\omega }-2N\omega \right) }{N}\geqslant 1)\psi _{N}\right\rangle .
\end{eqnarray*}%
To continue estimating, we notice that if $C\geqslant 0$, then $\mathbf{1}%
(s\geqslant 1)\leqslant \mathbf{1}(s+C\geqslant 1)$ for all $s$. So%
\begin{eqnarray*}
\left\Vert \chi \tilde{\psi}_{N,\omega }-\tilde{\psi}_{N,\omega }\right\Vert
_{L^{2}}^{2} &\leqslant &\left\langle \tilde{\psi}_{N,\omega },\mathbf{1}(%
\frac{\kappa \left( \tilde{H}_{N,\omega }-2N\omega \right) }{N}\geqslant 1)%
\tilde{\psi}_{N,\omega }\right\rangle \\
&\leqslant &\left\langle \tilde{\psi}_{N,\omega },\mathbf{1}(\frac{\kappa
\left( \tilde{H}_{N,\omega }-2N\omega +N\alpha \right) }{N}\geqslant 1)%
\tilde{\psi}_{N,\omega }\right\rangle
\end{eqnarray*}%
With the inequality that $\mathbf{1}(s\geqslant 1)\leqslant s$ for all $%
s\geqslant 0$ and the fact that 
\begin{equation*}
\tilde{H}_{N,\omega }-2N\omega +N\alpha \geqslant 0
\end{equation*}%
proved in Theorem \ref{Theorem:Energy Estimate}, we arrive at%
\begin{eqnarray*}
\left\Vert \chi \tilde{\psi}_{N,\omega }-\tilde{\psi}_{N,\omega }\right\Vert
_{L^{2}}^{2} &\leqslant &\frac{\kappa }{N}\left\langle \tilde{\psi}%
_{N,\omega },\left( \tilde{H}_{N,\omega }-2N\omega +N\alpha \right) \tilde{%
\psi}_{N,\omega }\right\rangle \\
&\leqslant &\frac{\kappa }{N}\left\langle \tilde{\psi}_{N,\omega },\left( 
\tilde{H}_{N,\omega }-2N\omega \right) \tilde{\psi}_{N,\omega }\right\rangle
+\alpha \kappa \left\langle \tilde{\psi}_{N,\omega },\tilde{\psi}_{N,\omega
}\right\rangle ,
\end{eqnarray*}%
Using (a) and (c) in the assumptions of Theorem \ref{Theorem:3D->2D BEC
(Nonsmooth)}, we deduce that%
\begin{equation*}
\left\Vert \chi \tilde{\psi}_{N,\omega }-\tilde{\psi}_{N,\omega }\right\Vert
_{L^{2}}^{2}\leqslant C\kappa
\end{equation*}%
which implies%
\begin{equation*}
\left\Vert \tilde{\psi}_{N,\omega }^{\kappa }-\tilde{\psi}_{N,\omega
}\right\Vert _{L^{2}}\leqslant C\kappa ^{\frac{1}{2}}.
\end{equation*}

(iii) does not follow from the proof of \cite[Proposition 8.1 (iii)]{E-S-Y3}
in which the positivity of $V$ is used. (iii) follows from the proof of \cite%
[Proposition 5.1 (iii)]{E-S-Y2} which does not require $V$ to hold a
definite sign. Proposition \ref{Prop:approximation of initial} follows the
same proof as \cite[Proposition 5.1 (iii)]{E-S-Y2} if one replaces $H_{N}$
by $(\tilde{H}_{N,\omega }-2N\omega )$ and $\hat{H}_{N}$ by 
\begin{equation*}
\sum_{j\geqslant k+1}^{N}(-\partial _{z_{j}}+\omega (-2-\Delta
_{x_{j}}+\left\vert x_{j}\right\vert ^{2}))+\frac{1}{N}\sum_{k+1<i<j\leq
N}V_{N,\omega }(r_{i}-r_{j}).
\end{equation*}%
Notice that we are working with $V_{N,\omega }=\left( N\omega \right)
^{3\beta }V_{\omega }(\left( N\omega \right) ^{\beta }r)$ where $V_{\omega
}(r)=\frac{1}{\omega }V(\frac{x}{\sqrt{\omega }},z),$ thus we get a $\left(
N\omega \right) ^{\frac{3\beta }{2}}\left\Vert V_{\omega }\right\Vert
_{L^{2}}^{2}\sim \frac{\left( N\omega \right) ^{\frac{3\beta }{2}}}{\omega }$
instead of a $N^{\frac{3\beta }{2}}$ in \cite[(5.20)]{E-S-Y2} and hence we
get a $\left( N\omega \right) ^{\frac{3\beta }{2}-1}$ in the estimate of 
\cite[(5.18)]{E-S-Y2} which tends to zero as $N,\omega \rightarrow \infty $
for $\beta \in \left( 0,2/3\right) $.
\end{proof}

Via (i) and (iii) of Lemma \ref{Prop:approximation of initial}, $\tilde{\psi}%
_{N,\omega }^{\kappa }(0)$ verifies the hypothesis of Theorem \ref%
{Theorem:3D->2D BEC} for small enough $\kappa >0.$ Therefore, for $\tilde{%
\gamma}_{N,\omega }^{\kappa ,(1)}\left( t\right) ,$ the marginal density
associated with $e^{it\tilde{H}_{N,\omega }}\tilde{\psi}_{N,\omega }^{\kappa
}(0),$ Theorem \ref{Theorem:3D->2D BEC} gives the convergence 
\begin{equation}
\lim_{\substack{ N,\omega \rightarrow \infty  \\ C_{1}N^{v_{1}(\beta
)}\leqslant \omega \leqslant C_{2}N^{v_{2}(\beta )}}}\limfunc{Tr}\left\vert 
\tilde{\gamma}_{N,\omega }^{\kappa ,(k)}(t,\mathbf{x}_{k},\mathbf{z}_{k};%
\mathbf{x}_{k}^{\prime },\mathbf{z}_{k}^{\prime
})-\dprod\limits_{j=1}^{k}h_{1}(x_{j})h_{1}(x_{j}^{\prime })\phi (t,z_{j})%
\overline{\phi }(t,z_{j}^{\prime })\right\vert =0.
\label{convergence:smooth}
\end{equation}%
for all small enough $\kappa >0,$ all $k\geqslant 1$, and all $t\in \mathbb{R%
}$.

For $\tilde{\gamma}_{N,\omega }^{(k)}\left( t\right) $ in Theorem \ref%
{Theorem:3D->2D BEC (Nonsmooth)}, we notice that, $\forall J^{(k)}\in 
\mathcal{K}_{k}$, $\forall t\in \mathbb{R}$, we have 
\begin{eqnarray*}
&&\left\vert \limfunc{Tr}J^{(k)}\left( \tilde{\gamma}_{N,\omega
}^{(k)}\left( t\right) -\left\vert h_{1}\otimes \phi \left( t\right)
\right\rangle \left\langle h_{1}\otimes \phi \left( t\right) \right\vert
^{\otimes k}\right) \right\vert \\
&\leqslant &\left\vert \limfunc{Tr}J^{(k)}\left( \tilde{\gamma}_{N,\omega
}^{(k)}\left( t\right) -\tilde{\gamma}_{N,\omega }^{\kappa ,(k)}\left(
t\right) \right) \right\vert +\left\vert \limfunc{Tr}J^{(k)}\left( \tilde{%
\gamma}_{N,\omega }^{\kappa ,(k)}\left( t\right) -\left\vert h_{1}\otimes
\phi \left( t\right) \right\rangle \left\langle h_{1}\otimes \phi \left(
t\right) \right\vert ^{\otimes k}\right) \right\vert \\
&=&\text{I}+\text{II}.
\end{eqnarray*}%
Convergence \eqref{convergence:smooth} then takes care of $\text{II}$. To
handle $\text{I}$ , part (ii) of Lemma \ref{Prop:approximation of initial}
yields 
\begin{equation*}
\left\Vert e^{it\tilde{H}_{N,\omega }}\tilde{\psi}_{N,\omega }(0)-e^{it%
\tilde{H}_{N,\omega }}\tilde{\psi}_{N,\omega }^{\kappa }(0)\right\Vert
_{L^{2}}=\left\Vert \tilde{\psi}_{N,\omega }(0)-\tilde{\psi}_{N,\omega
}^{\kappa }(0)\right\Vert _{L^{2}}\leqslant C\kappa ^{\frac{1}{2}}
\end{equation*}%
which implies 
\begin{equation*}
I=\left\vert \limfunc{Tr}J^{(k)}\left( \tilde{\gamma}_{N,\omega
}^{(k)}\left( t\right) -\tilde{\gamma}_{N,\omega }^{\kappa ,(k)}\left(
t\right) \right) \right\vert \leqslant C\left\Vert J^{(k)}\right\Vert
_{op}\kappa ^{\frac{1}{2}}.
\end{equation*}%
Since $\kappa >0$ is arbitrary, we deduce that 
\begin{equation*}
\lim_{\substack{ N,\omega \rightarrow \infty  \\ C_{1}N^{v_{1}(\beta
)}\leqslant \omega \leqslant C_{2}N^{v_{2}(\beta )}}}\left\vert \limfunc{Tr}%
J^{(k)}\left( \tilde{\gamma}_{N,\omega }^{(k)}\left( t\right) -\left\vert
h_{1}\otimes \phi \left( t\right) \right\rangle \left\langle h_{1}\otimes
\phi \left( t\right) \right\vert ^{\otimes k}\right) \right\vert =0.
\end{equation*}%
i.e. as trace class operators 
\begin{equation*}
\tilde{\gamma}_{N,\omega }^{(k)}\left( t\right) \rightarrow \left\vert
h_{1}\otimes \phi \left( t\right) \right\rangle \left\langle h_{1}\otimes
\phi \left( t\right) \right\vert ^{\otimes k}\text{ weak*.}
\end{equation*}%
Then again, the Gr\"{u}mm's convergence theorem upgrades the above weak*
convergence to strong. Thence, we have concluded Theorem \ref{Theorem:3D->2D
BEC (Nonsmooth)} via Theorem \ref{Theorem:3D->2D BEC} and Lemma \ref%
{Prop:approximation of initial}.

\end{document}